\documentclass[11pt,reqno]{article} \usepackage{authblk}

\usepackage{graphicx}
\usepackage[top=1in, bottom=1in, left=1in, right=1in] {geometry}
\usepackage{amscd,amsmath,amstext,amsfonts,amsbsy,
amssymb,amsthm,mathrsfs,float,latexsym}
\numberwithin{equation}{section}

\usepackage{multicol,graphicx,array,multirow,color,lineno}
\usepackage{sidecap,url,fancybox,fancyhdr}
\usepackage[colorlinks=true,linkcolor=cyan,citecolor=magenta]{hyperref}

\usepackage{extpfeil}

\usepackage{subcaption}

\usepackage{enumitem}
\newcommand{\eps}{\varepsilon}

\usepackage{tikz}
\usetikzlibrary{arrows.meta}

\newcommand{\wn}{\widetilde{n}_{\varepsilon}}
\newcommand{\wc}{\widetilde{c}_{\varepsilon}}
\newcommand{\ww}{\widetilde{w}_{\varepsilon}}

\renewcommand{\O}{\Omega}
\renewcommand{\epsilon}{\varepsilon}

\newcommand{\intO}{\int_{\Omega}}

\newcommand{\intQt}{\iint_{\Omega_{t}}}

\newcommand{\intQT}{\iint_{\Omega_{T}}}

\newtheorem{theorem}{{\bf Theorem}}[section]
\theoremstyle{definition} 
\newtheorem{assumption}{\bf Assumption}[section]
\newtheorem{corollary}{{\bf Corollary}}[]
\theoremstyle{plain}
\newtheorem{lemma}{Lemma}[section]

\newtheorem{remark}{Remark}[section]

\newcommand{\hs}{\hspace*{-0.2cm}}

\DeclareMathOperator*{\esssup}{ess\,sup}

\title{\bf \LARGE Parabolic-elliptic and indirect-direct simplifications in   chemotaxis systems driven by indirect signalling}

\author{Le Trong Thanh Bui$^{a}$\footnote{thanhblt@ueh.edu.vn}\,, Thi Kim Loan Huynh$^{b,c,d}$\footnote{loanhtk@tlu.edu.vn}\,, Bao Quoc Tang$^{e}$\footnote{quoc.tang@uni-graz.at, baotangquoc@gmail.com}\,, Bao-Ngoc Tran$^{e,f}$\footnote{bao-ngoc.tran@uni-graz.at, tranbaongoc@hcmuaf.edu.vn (corresponding author)}}
 
\date{}

\begin{document}

\maketitle 

\vspace{-1cm}

\begin{center}
{\small
$^{a}$University of Economics Ho Chi Minh City, Ho Chi Minh City,
Vietnam\vspace{0.2cm}\\
$^{b}$Faculty of Mathematics and Computer Science, University of Science, Ho Chi Minh City, Vietnam \vspace{0.2cm} \\
$^{c}$Vietnam National University, Ho Chi Minh City, Vietnam\vspace{0.2cm}\\ 
$^{d}$Department of Basic Science, ThuyLoi University - Southern Campus, Ho Chi Minh City, Vietnam \vspace{0.2cm}\\
$^{e}$Department of Mathematics and Scientific Computing, University of Graz, \\ Heinrichstrasse 36, 8010 Graz, Austria \vspace{0.2cm}
\\
$^{f}$Department of Mathematics, Faculty of Science, Nong Lam University,    Ho Chi Minh City, Vietnam
}
\end{center}

\begin{abstract} 
{Singular limits for the following indirect signalling chemotaxis system}
\begin{align*} 
	\left\{ \begin{array}{lllllll}
		\partial_t n  = \Delta n  -  \nabla \cdot (n  \nabla c ) & \text{in } \Omega\times(0,\infty) , \\ 
	\varepsilon \partial_t c   = \Delta c - c  + w & \text{in } \Omega\times(0,\infty), 	\\ 
	\varepsilon	\partial_t w   = \tau \Delta w  - w  + n  & \text{in } \Omega\times (0,\infty), \\
    \partial_\nu n = \partial_\nu c = \partial_\nu w = 0, &\text{on } \partial\Omega\times (0,\infty)
    %(n,c,w)_{t=0} = (n_0,c_0,w_0) & \text{on } \Omega, 
	\end{array} \right.
\end{align*}
{are investigated. More precisely, we study parabolic-elliptic simplification, or PES, $\varepsilon\to 0^+$ with fixed $\tau>0$ up to the critical dimension $N=4$, and indirect-direct simplification, or IDS, $(\varepsilon,\tau)\to (0^+,0^+)$ up to the critical dimension $N=2$. These are relevant in biological situations where the signalling process is on a much faster time scale compared to the species diffusion and all interactions. Showing singular limits in critical dimensions is challenging. To deal with the PES, we carefully combine the entropy function, an Adam-type inequality, the regularisation of slow evolution, and an energy equation method to obtain strong convergence in representative spaces. For the IDS, a bootstrap argument concerning the $L^p$-energy function is devised, which allows us to obtain suitable uniform bounds for the singular limits. Moreover, in both scenarios, we also present the convergence rates, where the effect of the initial layer and the convergence to the critical manifold are also revealed.}
% .  Observably, we expect that 
% \begin{align*}
%     \begin{array}{lll}
%     (\varepsilon \partial_t c, \varepsilon \partial_t w)=(\Delta c - c + w, \, \tau \Delta w - w + n) \rightarrow (0,0)   \quad   \text{as } \tau>0 \text{ is fixed and }  \varepsilon\to 0^+,  \\
%     (\varepsilon \partial_t c, \varepsilon \partial_t w - \tau \Delta w)=(\Delta c - c + w, - w + n)  \rightarrow (0,0)  \hfill \text{as } (\varepsilon,\tau)\to (0^+,0^+), 
%     \end{array} 
% \end{align*}
% and therefore, the system is formally simplified to 
% \begin{align*} 
% 	\left\{ \begin{array}{lllllll}
% 		\partial_t n  = \Delta n  -  \nabla \cdot (n  \nabla c ) & \text{in } \Omega\times (0,\infty), \\ 
% 	 \Delta c - c  + w  = 0 & \text{in } \Omega\times (0,\infty),	\\ 
% 	 \tau \Delta w  - w  + n =0 & \text{in } \Omega\times (0,\infty), \\
%     n|_{t=0} = n_0  & \text{on } \Omega, 
% 	\end{array} \right. \quad \text{and} \quad 
%     \left\{ \begin{array}{lllllll}
% 		\partial_t n  = \Delta n  -  \nabla \cdot (n  \nabla c ) & \text{in } \Omega\times (0,\infty), \\ 
% 	 \Delta c - c  + n =0 & \text{in } \Omega\times (0,\infty), \\ 
%     n|_{t=0} = n_0  & \text{on } \Omega. 
% 	\end{array} \right.
% \end{align*}
% We provide rigorous analysis for these simplifications, including passage to the limits, convergence rate estimates with the initial layer effect, and the convergence to critical manifolds.    

\medskip


\end{abstract}  

\tableofcontents

%%%%%%%%%%%%%%%%%%%%%%%%
% Not including subsubsection in the table of contents
%\setcounter{tocdepth}{2} 
%%%%%%%%%%%%%%%%%%%%%%%%
%\tableofcontents

\section{Introduction} 

The term chemotaxis has been widely used to describe the directed movement of a species responding to a stimulus, with numerous applications in bacterial aggregation \cite{berg1972chemotaxis,erban2004individual}, cell invasion \cite{roussos2011chemotaxis,bellomo2015toward}, food chains \cite{tao2022existence,reisch2024global}, and other contexts. In mathematical modelling, it turns into cross-diffusive terms in parabolic-parabolic or parabolic-elliptic systems of PDEs. Recently, chemotaxis systems with indirect signalling mechanisms have gained a lot of attention, where a system may include one species and two signals, or two species and one signal.  
Besides the suggestion of better responses of a species to the environment, see e.g. \cite{neumann2010differences}, the differences between the direct and indirect signalling also raise many interesting analytical questions, regarding the global solvability and uniform boundedness \cite{fujie2017application,ren2022global}, infinite-time aggregation \cite{tao2017critical,tao2025switch}, large-time behaviours \cite{zhang2019large,liu2020global}, or singular limits \cite{li2023convergence,laurenccot2024singular}.    

\medskip
Let $\Omega \subset \mathbb{R}^N$, $1\le N\le 4$, be a bounded domain with sufficiently smooth boundary {$\Gamma:= \partial\Omega$}. {In this work, we study the singular limits $\eps\to 0^+$ and $(\eps,\tau) \to (0^+, 0^+)$ of the following} indirect signalling chemotaxis system     
\begin{align}\label{Sys:Eps:Main}
	\left\{ \begin{array}{lllllll}
		\partial_t n &\hspace{-0.3cm}=&\hspace{-0.2cm} \Delta n -  \nabla \cdot (n \nabla c) & \text{in } \Omega \times(0,\infty) , \vspace{0.1cm} \\ 
	\varepsilon \partial_t c &\hspace{-0.3cm}=&\hspace{-0.2cm}   \Delta c - c + w  & \text{in } \Omega \times(0,\infty), \vspace{0.1cm} 	\\ 
	\varepsilon \partial_t w &\hspace{-0.3cm}=&\hspace{-0.2cm}  \tau \Delta w - w + n   & \text{in } \Omega \times(0,\infty),
	\end{array} \right.
\end{align}
which is subjected to the no-flux boundary conditions 
\begin{align}\label{Sys:Eps:BounCond}
\frac{\partial n}{\partial \nu} = \frac{\partial c}{\partial \nu} = \frac{\partial w}{\partial \nu}= 0 \quad \text{on } \Gamma \times(0,\infty) , 
\end{align}
and the initial condition 
\begin{align}\label{Sys:Eps:InitCond}
	(n,c,w)|_{t=0}=(n_{0},c_{0},w_{0})  \quad \text{on }  \Omega,  
\end{align}
where    $n_{0},c_{0},w_{0}$ are given smooth data.  {This system has been studied in  \cite{strohm2013pattern,li2023convergence} to model the movement of Mountain Pine Beetles in a forest habitat $\Omega$, with $\eps>0$ and $\tau=0$}, where $n$ and $w$ represent the densities of the flying and nesting species, and $c$ is the concentration of beetle pheromones. {In \cite{fujie2017application},  the authors studied System \eqref{Sys:Eps:Main}, with $\varepsilon=\tau=1$, which models the aggregation phenomena of microglia cells in the Alzheimer disease}, where $n$ represents a species density and $c,w$ are the {concentrations of two different chemicals}. A variant of \eqref{Sys:Eps:Main} with the setting in the whole spatial domain $\mathbb{R}^4$ can be found in \cite{hosono2025global}. For related models {concerning indirect signalling}, we refer the reader to  \cite{tao2017critical,zhang2019large,liu2020global,ren2022global,laurenccot2024singular,tao2025switch} and references therein.  
\medskip

Biologically, signals can diffuse on a much faster time scale than the species self-diffusion, which leads to mathematical models that include a sufficiently small parameter $0<\varepsilon \ll 1$ appearing in front of the time evolution of the signal concentration  (i.e., its time derivatives). This scenario has been discussed for the last several decades, where parabolic-parabolic chemotaxis systems had been simplified to their parabolic-elliptic relatives \cite{corrias2004global,kiselev2022chemotaxis}. This type of simplification is well-known as the notion of \textit{fast signal diffusion limits} or \textit{parabolic-elliptic simplification} (PES for short) \cite{wang2019fast,reisch2024global}, which offers significant benefits not only in mathematical analysis but also in computational simulations. A PES is formally achieved by removing the signal evolution from the considered chemotaxis models, or equivalently, by formally assigning $\eps = 0$, leading to an elliptic instead of a parabolic equation for the chemical/signal concentration. However, rigorous analysis of PES has only been conducted in recent works, such as \cite{mizukami2018fast,
mizukami2019fast,freitag2020fast,wang2019fast,ogawa2023maximal,reisch2024global}. {On the other hand, by setting $\eps = \tau = 0$, we see from the third equation of \eqref{Sys:Eps:Main} that $c \equiv w$, i.e. the two signals coincide, and \eqref{Sys:Eps:Main} is reduced to a chemotaxis system with a direct signal. Thus, the singular limit problem $(\eps,\tau)\to (0^+,0^+)$ is called \textit{indirect-direct simplification} (IDS for short), and has also been considered for related problems in e.g. \cite{li2023convergence,laurenccot2024singular}.} 

\medskip

{The main goals of this work are to study PES and IDS for \eqref{Sys:Eps:Main} up to the critical dimensions, $N=4$ and $N=2$, respectively, where we prove the convergence and estimate the convergence rates including the initial layer effect. In the following, we first give the state of the art, which helps to highlight the motivation and novelty of our work. Then, we present our main results as well as the key ideas.}
%In this work, we provide a rigorous analysis of this simplification, highlighting the crucial role of the initial layer in determining its accuracy. In addition, a possible simplification from the indirect mechanism to the corresponding direct one will also be analysed. In more detail, we will shortly introduce these simplifications in the sequel. 

\subsection{State of the art} 

\medskip
{The study of PES has been initiated in recent years, with the first work focusing on the classical parabolic-parabolic Keller-Segel model}
\begin{align}
\label{Ma:2018:eps}
\left\{ \begin{array}{llll}
\partial_t u_\lambda =  \Delta u_\lambda - \chi \nabla \cdot ( u_\lambda \nabla v_\lambda)  & \text{in } \Omega \times(0,\infty),  \\
\lambda \partial_t v_\lambda = \Delta v_\lambda - v_\lambda + u_\lambda & \text{in } \Omega \times(0,\infty),   \\ 	
(u_\lambda,v_\lambda)|_{t=0}=(u_{0},v_{0}) &  \text{on }  \Omega, 
\end{array} \right.
\end{align}
(subjected to the no-flux boundary conditions) and its parabolic-elliptic relative  
\begin{align}
\left\{ \begin{array}{llll}
\partial_t u  =  \Delta u - \chi \nabla \cdot (u  \nabla v )  & \text{in } \Omega \times(0,\infty),  \\
 \Delta v - v + u = 0 & \text{in } \Omega \times(0,\infty),  \\ 	
    u|_{t=0}=u_{0} &  \text{on }  \Omega.  
\end{array} \right.
\label{Ma:2018:0}
\end{align}
In \cite{mizukami2019fast}, the author positively answered the question: \textit{Does the solution of \eqref{Ma:2018:eps} converge to that of \eqref{Ma:2018:0} as $\lambda \to 0$?} With {sufficiently small and regular} initial data $u_0,v_0$,  
the author showed for $N\ge 2$ that %\begin{align}
    %\left\{\begin{array}{llll}
    $u_\lambda \to u \text{ in } C_{\mathsf{loc}}(\overline{\Omega}\times[0,\infty))$ and $v_\lambda \to v \text{ in } C_{\mathsf{loc}}(\overline{\Omega}\times(0,\infty)) \cap L^2_{\mathsf{loc}}((0,\infty);W^{1,2}(\Omega))$
    %\end{array}\right.
    %\label{Ma:2018:Conv}
%\end{align} (\textcolor{red}{probably write this in inline style?}) 
as $\lambda\to 0$, where the limit $(u,v)$ {is the classical solution of} \eqref{Ma:2018:0}. {When the chemotactic flux is of the form $u_\lambda S(v_\lambda) \nabla v_\lambda$ (instead of $u_\lambda\nabla v_\lambda$), \cite{mizukami2018fast} showed that for a sensitivity} $S\in C^{1+\vartheta}((0,\infty))$,   $\vartheta\in(0,1)$, {satisfying} $0\le S(v) \le \chi(a+v)^{-k}$ for  $a\ge 0$, $k>1$, {the above convergence holds provided} $\chi<\chi_*$ for some $\chi_*>0$ depending on $k,a,N,u_0,v_0$.  In \cite{freitag2020fast}, the author investigated PES for \eqref{Ma:2018:eps} but with non-degenerate diffusion of porous medium type. For the whole domain setting $\Omega = \mathbb{R}^N$, we refer the reader, for instance, to \cite{kurokiba2020singular,ogawa2023maximal}. This PES has also been investigated also in \cite{wang2019fast} in the context of Keller-Segel-(Navier-)Stokes system 
\begin{align*} 
\left\{ \begin{array}{llll}
\partial_t n_\varepsilon + u_\varepsilon \cdot \nabla n_\varepsilon =  \Delta n_\varepsilon - \nabla \cdot ( n_\varepsilon S(x,n_\varepsilon,c_\varepsilon) \cdot \nabla c_\varepsilon) + f(x,n_\varepsilon,c_\varepsilon),  \\
\varepsilon \partial_t c_\varepsilon + u_\varepsilon \cdot \nabla c_\varepsilon = \Delta c_\varepsilon - c_\varepsilon + n_\varepsilon  ,   \\ 	
\partial_t u_\varepsilon + \kappa (u_\varepsilon \cdot \nabla) u_\varepsilon = \Delta u_\varepsilon + \nabla P_\varepsilon + n_\varepsilon  \nabla \phi, \; \kappa \in \mathbb R, \hfill \nabla \cdot u_\varepsilon = 0  , \\
(n_\varepsilon,c_\varepsilon,u_\varepsilon)|_{t=0} = (n_0,c_0,u_0),
\end{array} \right.
\end{align*}
subjected {$\partial_{\nu}n_\varepsilon = \partial_{\nu}c_\varepsilon = 0$} and $u_\varepsilon = 0$ on the boundary. It was (conditionally) shown therein that this system can be rigorously simplified  to its relative 
\begin{align*}
\left\{ \begin{array}{llll}
\partial_t n + u \cdot \nabla n =  \Delta n - \nabla \cdot ( n S(x,n,c) \cdot \nabla c) + f(x,n,c),  \\
u \cdot \nabla c = \Delta c - c + n  ,   \\ 	
\partial_t u + \kappa (u \cdot \nabla) u = \Delta u + \nabla P + n   \nabla \phi, \hfill \nabla \cdot u = 0  , \\
(n,u)|_{t=0} = (n_0,u_0),
\end{array} \right.
\end{align*}
via the limit as $\varepsilon\to 0$, {provided the following uniform-in-$\varepsilon$ boundedness of $\nabla c_\varepsilon$ and $u_\varepsilon$}
\begin{align*}
    \sup_{\varepsilon>0} \Big( \|\nabla c_\varepsilon\|_{L^p((0,T);L^q(\Omega))} + \|u_\varepsilon\|_{L^\infty((0,T);L^r(\Omega))} \Big) <\infty,
\end{align*}
for some $p,q,r$ such that {$2<p\le \infty$, $q>N$, $r>\max\{2;N\}$ such that $\frac{1}{p} + \frac{N}{2q} < \frac{1}{2}$}. 
{Related results can be found} in \cite{li2021convergence,
li2023stability,
wu2025fast}. 

\medskip
Besides PES, the investigation of IDS has also attracted considerable attention recently. A first work in this direction seems to be \cite{painter2023phenotype}, where the authors considered a phenotype-switching chemotaxis model, which represents an indirect signalling scheme, of the form
\begin{equation}
    \begin{cases}
        \partial_t u_\gamma = \Delta u_\gamma -  \nabla\cdot(u_\gamma\nabla v_\gamma) - \gamma u_\gamma + \gamma w_\gamma, &x\in\Omega,\\
        \partial_t v_\gamma = \Delta v_\gamma - v_\gamma + w_\gamma, &x\in\Omega,\\
        \partial_tw_\gamma = \Delta w_\gamma - \gamma w_\gamma + \gamma u_\gamma, &x\in\Omega,\\
        \partial_{\nu}u_\gamma = \partial_{\nu}v_\gamma = \partial_{\nu}w_\gamma = 0, &x\in\Gamma.
    \end{cases}
\end{equation}
As $\gamma \to \infty$, one expects the limit $(n_\gamma:= u_\gamma + w_\gamma, v_\gamma) \to (n, v)$ where the latter solves the classical Keller-Segel model with direct signalling
\begin{equation*}
    \begin{cases}
        \partial_tn = \Delta n - \frac{\theta}{1+\theta}\nabla \cdot(n\nabla v), &x\in\Omega,\\
        \partial_t v = \Delta v - v + \frac{n}{1 + \theta}, &x\in\Omega,\\
        \partial_\nu n = \partial_\nu v = 0, &x\in\Gamma.
    \end{cases}
\end{equation*}
This convergence was partially shown in \cite{painter2023phenotype}, and later fully proved in \cite{laurenccot2024singular}. A similar problem was considered in \cite{li2023convergence}, where the authors studied  the following system
\begin{align*} 
	\left\{ \begin{array}{lllllll}
		\partial_t n_\varepsilon = \Delta n_\varepsilon -  \nabla \cdot (n_\varepsilon \nabla c_\varepsilon)  , \\ 
	\varepsilon_1 \partial_t c_\varepsilon  = \Delta c_\varepsilon - c_\varepsilon + w_\varepsilon  , 	\\ 
	\varepsilon_2	\partial_t w_\varepsilon  =   - w_\varepsilon + n_\varepsilon , \\
    (n_\varepsilon,c_\varepsilon,w_\varepsilon)|_{t=0}=(n_0,c_0,w_0).
	\end{array} \right.
\end{align*}
{Under the assumption that the initial mass $\int_\Omega n_0$ is sub-critical, i.e. smaller than $4\pi$, this system is shown to converge to either} 
\begin{align*} 
	\left\{ \begin{array}{lllllll}
		\partial_t n = \Delta n  -  \nabla \cdot (n  \nabla c )   , \\ 
	  \partial_t c   = \Delta c  - c  + w , \\
      (n,c)|_{t=0}=(n_0,c_0),
	\end{array} \right. 
    \quad \text{ or } \quad 
    \left\{ \begin{array}{lllllll}
		\partial_t n = \Delta n  -  \nabla \cdot (n  \nabla c )   , \\ 
	  \Delta c  - c  + w = 0, \\
      n|_{t=0}=n_0,
	\end{array} \right.
\end{align*}
corresponding to $\varepsilon_1 =\varepsilon_2 \to 0$ or $\varepsilon_1=1$, $\varepsilon_2 \to 0$, {respectively}.

\medskip
It's worthwhile to mention that the modelling and analysis of chemotaxis systems with indirect signalling of the type \eqref{Sys:Eps:Main}, both in the parabolic-parabolic and parabolic-elliptic settings, have been subjected to extensive investigation, see e.g. \cite{ahn2021global,fuest2023critical,laurenccot2018global,strohm2013pattern,white1998spatial,wu22global} and references therein. Even the question of global existence can be challenging, especially in the critical dimension $N=4$, see e.g. \cite{fujie2017application,hosono2025global}. 

\medskip
Our current work adequately contributes to this literature by investigating the PES and IDS for chemotaxis systems with indirect signalling \eqref{Sys:Eps:Main}-\eqref{Sys:Eps:InitCond} up to the critical dimensions $N=4$ and $N=2$, respectively. Furthermore, we also provide the convergence rates, which have been seemingly completely left out in the literature, and reveal the effect of the initial layer. 
   
\subsection{Main results, challenges and key ideas}

{\bf Notations}: We denote by $L^p$, $W^{k,p}$, for $1\le p\le \infty$ and $k\ge 0$, the usual Lebesgue and Sobolev spaces. Moreover, a general constant $C$ is used for any positive constant that does not depend on spatial and temporal variables, all the unknowns, as well as the relaxation parameters $\varepsilon,\tau$. This general constant can vary from line to line, or even within the same line. In case where a dependence is important, such as the dependence on a terminal time $T$ or the diffusion coefficient $\tau$, we will write $C_T$ or $C_\tau$, etc. {For $0<T\le \infty$, we denote by $\Omega_T:= \Omega\times(0,T)$}

\medskip
To study singular limits for \eqref{Sys:Eps:Main}, we impose the following assumption on initial data throughout this work.
\begin{assumption} \label{Ass:InitialData} The initial data {$(n_0,c_0,w_0) \in C^1(\bar \Omega)\times C^2(\bar \Omega)^2$} is nonnegative and satisfied the compatible condition, i.e., $\frac{\partial n_0}{\partial \nu} = \frac{\partial c_0}{\partial \nu} = \frac{\partial w_0}{\partial \nu} = 0$ {on the boundary $\Gamma$}. 
\end{assumption}

{\bf Our first main results} are about the PES from \eqref{Sys:Eps:Main}-\eqref{Sys:Eps:InitCond} to \eqref{Sys:Lim:Main}-\eqref{Sys:Lim:InitCond}.
{Fix $\tau>0$ and denote by $(n_\eps, c_\eps, w_\eps)$ the solution of \eqref{Sys:Eps:Main} with respect to $\eps>0$.
As  $\varepsilon \to 0$}, we {formally expect that $(n_\varepsilon,c_\varepsilon,w_\varepsilon) \rightarrow (n,c,w)$, and the limit vector   $(n,c,w)$ solves the system}
\begin{align}\label{Sys:Lim:Main}
	\left\{ \begin{array}{lllllll}
		\partial_t n = \Delta n -  \nabla \cdot (n \nabla c) & \text{in } \Omega \times(0,\infty) , \\ 
	\Delta c - c + w = 0 & \text{in } \Omega \times(0,\infty), 	\\ 
	\tau \Delta w - w + n = 0 & \text{in } \Omega \times(0,\infty), \vspace{0.1cm} \\
    \dfrac{\partial n}{\partial \nu} = \dfrac{\partial c}{\partial \nu} = \dfrac{\partial w}{\partial \nu} = 0 & \text{on } \Gamma \times(0,\infty), 
	\end{array} \right.
\end{align}
equipped with the  initial value condition
\begin{align}\label{Sys:Lim:InitCond}
	n|_{t=0}=n_{0} \quad \text{on }  \Omega.
\end{align}
{One of the main challenges} when connecting solutions of \eqref{Sys:Eps:Main}-\eqref{Sys:Eps:InitCond} and  \eqref{Sys:Lim:Main}-\eqref{Sys:Lim:InitCond} or  \eqref{Sys:LimKappa:Main} is the different structures between the parabolicity and ellipticity and the initial layer, especially in the critical dimensions, {$N=4$ for PES and $N=2$ for IDS}, see \cite{nagai1997application}. 
First, to pass to the limit in a strong sense, the slow evolution (i.e., the products of $\varepsilon$ and the time derivatives of $c_\varepsilon,w_\varepsilon$) make the Aubin-Lions lemma difficult to apply. For example, the $L^p$ maximal regularity applied to the slow-evolution equation $\varepsilon \partial_t u_\varepsilon - d \Delta u_\varepsilon + u_\varepsilon = f(x,t)$, associated with the no-flux boundary condition, reads as
\begin{align*}
 \sup_{\varepsilon>0} \Big( \|\varepsilon \partial_t u_\varepsilon\|_{L
 ^{p}(\Omega\times(0,T))} +  \|\Delta u_\varepsilon\|_{L
 ^{p}(\Omega\times(0,T))} \Big)   \le  \left( \frac{\varepsilon}{p} \right)^{\frac{1}{p}} \| u_0\|_{W^{2,p}(\Omega)} +  C_{d,p} \|f \|_{L^p(\Omega\times(0,T))} , 
\end{align*}
see \cite[Lemma 3.4]{reisch2024global}, 
which do not directly give a uniform-in-$\varepsilon$ boundedness for the time derivative $\partial_t u_\varepsilon$. {Obtaining strong convergence for the slow evolution is tricky and usually requires considerable effort, see e.g. \cite{wang2019fast}}. Second, for fixed $\varepsilon>0$ and $\tau>0$, {even} the global solvability for the system \eqref{Sys:Eps:Main}-\eqref{Sys:Eps:InitCond} in the critical dimension $N=4$ is difficult, see \cite{fujie2017application,laurenccot2018global}.
Some steps in that {proof}, involving e.g. the use of the heat semigroup or testing the equations for $c_\varepsilon,w_\varepsilon$ by $c_\varepsilon,-\Delta c_\varepsilon,w_\varepsilon,-\Delta w_\varepsilon$ {heavily depend on $\eps$, and therefore do not yield the required} uniform-in-$\varepsilon$ estimates. For instance, the Duhamel principle  for the latter slow-evolution equation, represented via the Neumann heat semigroup, is written as  \begin{align*}
    u_\varepsilon(x,t) = e^{\frac{1}{\varepsilon}t(d\Delta-I)} u_\varepsilon(x,0) + \frac{1}{\varepsilon}\int_0^t e^{\frac{1}{\varepsilon}(t-s)(d\Delta-I)} f(x,s)ds,
\end{align*}
which yields that a uniform-in-$\varepsilon$ estimate can only be obtained if the regularity of $f$ is sufficiently regular, at least essentially bounded in time, {which is not the case in our situation}. Third, it has been numerically demonstrated in \cite{reisch2024global} that {initial data} starting far away from the critical manifold {$\mathcal C_{\mathsf{PES}}$ (see \eqref{mathcalM})} can lead to a significant loss of simplification accuracy. Hence, to achieve simplification accuracy, an analysis of the initial layer is required.    

In order to rigorously justify this simplification, we exploit the multiple time scale Lyapunov function, see Lemma \ref{L:FSL:Lya:Id},   \begin{align}
\mathcal E(n_\varepsilon,c_\varepsilon) := \int_\Omega \left( n_\varepsilon(\log n_\varepsilon - c_\varepsilon) +  \frac{1}{2}|\Delta c_\varepsilon - c_\varepsilon + w_\varepsilon|^2 + \frac{\tau}{2} |\Delta c_\varepsilon|^2 + \frac{1+\tau}{2} |\nabla c_\varepsilon|^2 + \frac{1}{2} c_\varepsilon^2 \right),
\label{FSL:Lya:Define}
\end{align}
with {its dissipation given by}
\begin{equation}
\begin{aligned}
\mathcal D(n_\varepsilon,c_\varepsilon) &:= {-\frac{d}{dt}\mathcal{E}(n_\eps,c_\eps)}\\
&= \int_\Omega \Big( n_\varepsilon|\nabla (\log n_\varepsilon - c_\varepsilon)|^2 + \frac{1+\tau}{\varepsilon}  |\nabla(\Delta c_\varepsilon - c_\varepsilon + w_\varepsilon)|^2 + \frac{2}{\varepsilon} |\Delta c_\varepsilon - c_\varepsilon + w_\varepsilon|^2 \Big). 
\end{aligned} 
\label{FSL:Lya:Id}
\end{equation}
{It is remarked that the term $n_\varepsilon (\log n_\varepsilon - c_\varepsilon)$ in the Lyapunov function $\mathcal{E}(n_\eps,c_\eps)$ has} no sign and needs to be estimated from below. If $1\le N\le 3$, the Sobolev embedding is sufficient to absorb the norm of  $n_\varepsilon c_\varepsilon$ in $L^\infty((0,T);L^1(\Omega))$ into the $L^\infty((0,T);H^2(\Omega))$-norm of $c_\varepsilon$ in $\mathcal E(n_\varepsilon,c_\varepsilon)$, cf. Lemma \ref{L:FSL:Est:c}, and to obtain an $L^\infty(\Omega_T)$-estimate for $n_\varepsilon$. In the critical dimension $N=4$, the method of using the Adam-type inequality, see \cite[Section 7]{fujie2017application}, can be adapted to balance the energy-dissipation equality. Unfortunately, because of the slow evolution, the locally spatial truncation argument in \cite[Section 8]{fujie2017application} does not work to control the $L^p$-energy. {We overcome this issue} by {adapting the idea} of combining the Sobolev, Gagliardo-Nirenberg, and Young inequalities in \cite[Proof of Theorem 1.2]{hosono2025global}. Then, {some feedback arguments}, using the heat semigroup as well as maximal regularity with slow evolution, help us to estimate
the slow evolution's components $w_\varepsilon,c_\varepsilon$. 

\medskip

{The strong convergence $c_\varepsilon \to c$ in $L^2((0,T);H^1(\Omega))$ is challenging, see e.g. \cite[Section 5]{wang2019fast}, where this was proved by heavily exploiting the higher regularity of $c_\eps$. In this work, we provide a shortened and more direct proof by employing the argument} from \eqref{Theo:FSL:1:P4}-\eqref{Theo:FSL:1:P8}, which is basically based on the so-called \textit{energy equation method}, see e.g.  \cite{ball2004global,henneke2016fast}. This method uses the equation obtained by considering an $L^2$ energy of $(c_\varepsilon-c)$, instead of the energy inequality, and then shows the convergence in norms before using the uniform convexity of $L^2((0,T);H^1(\Omega))$ to get the strong convergence.  

\begin{theorem}[PES for \eqref{Sys:Eps:Main}]
\label{Theo:FSL:1} Let $1\le N\le 4$ and fix $\tau>0$. Assume that  $(n_0,c_0,w_0)$ is complied with Assumption \ref{Ass:InitialData}, and furthermore in the critical dimension $N=4$ that $\Omega = B_R$ for some $R>0$ and 
\begin{align}
    M{:= \int_{\Omega} n_0} < 64 \tau  \pi^2 .
\label{T:1:4:MassCond}
\end{align}
For each $\varepsilon>0$, let  $(n_\varepsilon,c_\varepsilon,w_\varepsilon)$ be the global classical solution to parabolic-parabolic system \eqref{Sys:Eps:Main}-\eqref{Sys:Eps:InitCond}, given by Theorem \ref{Theo:GloEX}. Then, for any $0<T<\infty$,    
\begin{align}
\begin{gathered}
    \sup_{\varepsilon>0} \Big( \|n_\varepsilon\|_{C^{\gamma,\gamma/2}(\overline{\Omega}\times[0,T])} + \|n_\varepsilon\|_{L^2((0,T);H^1(\Omega))}  \Big) \le C_{\tau,T},\\ 
    \sup_{\varepsilon>0} \Big( \|w_\varepsilon \|_{L^\infty((0,T);W^{1,\infty}(\Omega))} + \|\Delta w_\varepsilon \|_{L^p(\Omega_T)} + \|c_\varepsilon \|_{L^\infty((0,T);W^{2,\infty}(\Omega))} \Big) \le C_{\tau,T,p},
\end{gathered}
\label{Theo:FSL:1:S1}
\end{align}
for some $\gamma\in(0,1)$ and any $1\le p<\infty$. As $\varepsilon\to 0$, {we have the following limits} 
\begin{align}
    \begin{aligned}
\begin{array}{llllcl}
n_\varepsilon  &\longrightarrow& n & \text{strongly in}& C(\overline{\Omega}\times [0,T]), \\
\nabla n_\varepsilon &\xrightharpoonup{\hspace{0.38cm}}& \nabla n & \text{weakly in}& L^2(\Omega_T) , \\
c_\varepsilon  &\longrightarrow& c & \text{strongly in}& L^2((0,T);H^1(\Omega)), \\
w_\varepsilon  &\longrightarrow& w & \text{strongly in}& L^2((0,T);H^1(\Omega)),
\end{array}
\end{aligned}
\label{Theo:FSL:1:S2}
\end{align}
and the limit vector $(n,c,w)$ is the unique global classical solution to the indirect signalling parabolic-elliptic system \eqref{Sys:Lim:Main}-\eqref{Sys:Lim:InitCond}.
\end{theorem}

As a by-product of the proof of Theorem \ref{Theo:FSL:1}, we have the following convergences, which also explain the mechanism of the PES
\begin{equation}\label{Theo:FSL:1:S3}
\begin{aligned}
    \|\varepsilon \partial_t c_\varepsilon \|_{L^2((0,T);H^1(\Omega))}=\|\Delta c_\varepsilon - c_\varepsilon + w_\varepsilon \|_{L^2((0,T);H^1(\Omega))} \le C_\tau \sqrt{\varepsilon} , \\
    \varepsilon \partial_t w_\varepsilon = \tau\Delta w_\varepsilon - w_\varepsilon + n_\varepsilon \quad  \xrightharpoonup{\hspace{0.38cm}} \quad 0  \quad{\text{ in distributional sense}}.  
\end{aligned}
\end{equation}
{Up to now, we have only obtained the weak convergence for the equation of $w_\eps$ due to a lack of uniform regularity information of $\partial_t w_\varepsilon$}. {We show that this strong convergence will be a consequence of the next part, where the accuracy of the PES provided in Theorem \ref{Theo:FSL:1} is investigated}. By subtracting the corresponding equations of solution components of the systems \eqref{Sys:Eps:Main}-\eqref{Sys:Eps:InitCond}
and 
\eqref{Sys:Lim:Main}-\eqref{Sys:Lim:InitCond}, we see that the vector {$(\wn,\wc,\ww)  := (n_{\varepsilon}-n,c_{\varepsilon}-c, w_{\varepsilon}-w)$} is the solution of the so-called \textit{rate system}
\begin{align}\label{Sys:Rate:Convergence}
	\left\{ \begin{array}{lllllll}
		\partial_t \wn &\hs=\hs& \Delta \wn -  \nabla \cdot (\wn \nabla c_\varepsilon + n\nabla \wc) & \text{in } \Omega_\infty , \\ 
		\varepsilon \partial_t \wc  &\hs=\hs& \Delta \wc - \wc + \ww - \varepsilon \partial_t c & \text{in } \Omega_\infty, 	\\ 
		\varepsilon	\partial_t \ww  &\hs=\hs& \tau \Delta \ww - \ww + \wn - \varepsilon \partial_t w & \text{in } \Omega_\infty,
	\end{array} \right.
\end{align}
which is subjected to the boundary conditions 
\begin{align}\label{Sys:Rate:BounCond}
	\frac{\partial \wn} {\partial \nu} = \frac{\partial \wc} {\partial \nu} = \frac{\partial \ww}{\partial \nu} = 0 \quad \text{on } \Gamma_\infty, 
\end{align}
and the initial value condition
\begin{align}\label{Sys:Rate:InitialCond}
	(\wn (0), \wc(0), \ww (0)) = (0, c_0 - c(0), w_0 - w(0)).
\end{align}
It is obvious to see that $c(0)$ and $w(0)$ are not given a priori, and they may be well different from $c_0$ and $w_0$, respectively. These missing initial values can only be recovered, thanks to the last two equations in \eqref{Sys:Lim:Main}-\eqref{Sys:Lim:InitCond}, as
\begin{align}
     w(x,0) = (-\tau \Delta + I)^{-1}n_0, \quad c(x,0) = (- \Delta + I)^{-1}w(x,0).
     \label{InititalLayer0}
\end{align}
This difference in the initial values is referred to as the \textit{initial layer}. {It has been usually assumed to be zero in the literature, see e.g. \cite{li2024characterization}. However, this turns out to be important in studying the accuracy of the PES (or IDS), which is evidenced in the recent work \cite{reisch2024global}, where the effect of the initial layer has been carefully analysed for the PES of a competitive prey-predator chemotaxis system. This effect is especially relevant when the original initial data $(n_0, c_0, w_0)$ do not lie on the critical manifold, which is defined by}
\begin{align}
    \mathcal C_{\mathsf{PES}}:= \Big\{(n,c,w) \in L^2(\Omega) \times H^2(\Omega)^2: \, (\Delta c - c + w , \, \tau \Delta w - w + n) =(0,0)  \Big\}.
    \label{mathcalM}
\end{align}
{We define} the distance from the initial data $(n_0,c_0,w_0)$ to the critical manifold $\mathcal C_{\mathsf{PES}}
$ with respect to the topology $W^{k,p}(\Omega) \times W^{l,p}(\Omega)$ by 
\begin{align}
\label{DistanceDef}
\mathrm{dist}^{k,l}_p[(n_0,c_0,w_0);\mathcal C_{\mathsf{PES}}] := {\sqrt{\|- \Delta c_0 + c_0 -w_0\|_{W^{k,p}(\Omega)}^2 + \|- \tau \Delta w_0 + w_0-n_0\|_{W^{l,p}(\Omega)}^2}}\,,
\end{align}
for $k,l\in \mathbb N$ and $1\le p\le \infty$. {When $k=l=0$, $p=2$, we conveniently write $\mathrm{dist}:= \mathrm{dist}^{0,0}_2$}. {By using} the following representations of the inverse operators $(-\Delta + I)^{-1}$ and $(-\tau\Delta + I)^{-1}$, see e.g. \cite{reisch2024global},  
\begin{align} 
\left\{
\begin{array}{llll}
\displaystyle \wc(x,0) 
	\; = \int_0^{\infty} e^{s(\Delta - I)}[-\Delta c_0(x) + c_0(x) -w_0(x)] ds , & x \in \Omega, \vspace{0.15cm} \\
\displaystyle      \ww(x,0) =   \int_0^{\infty} e^{s(\tau \Delta - I)}[-\tau \Delta w_0(x) + w_0(x) - n_0(x)] ds , & x \in \Omega,
\end{array}\right.
\label{InitialLayer}
\end{align}
we can estimate {these the initial layers}  by the distance {$\mathrm{dist}^{1,2}_p [(n_0,c_0,w_0);\mathcal C_{\mathsf{PES}}]$}, see Lemma \ref{Lem:Layer}. Then, we can employ the {uniform-in-$\eps$ estimates in Theorem \ref{Theo:FSL:1} to obtain for each $1\le k\in \mathbb N$} (see Lemma \ref{est:engrate}), 
\begin{align*}
		\dfrac{d}{dt} \intO \wn ^{2k}(t) \leq -\dfrac{2k-1}{k}  \intO |\nabla \wn^{k}|^2 
		+ C_{k, T}\intO \wn ^{2k} + C_{k, T}\intO |\nabla \wc|^2,
	\end{align*}
{to test} the equations for $\wc,\ww$, and apply the fundamental differential inequality given in Lemma \ref{aDI} to obtain convergence rates as follows.  
\begin{theorem}[Convergence rates and the initial layer's effect]
\label{Theo:FSL:3} Let $1\le N\le 4$,  and fix $\tau>0$. For each $\varepsilon>0$, let $(n_\varepsilon,c_\varepsilon,w_\varepsilon)$ be the global classical solution to the system \eqref{Sys:Eps:Main}-\eqref{Sys:Eps:InitCond}, given by Theorem \ref{Theo:GloEX}. 

\medskip
 
\noindent a) Assuming that the distance $\mathrm{dist}^{2,1}_2 [(n_0,c_0,w_0);\mathcal C_{\mathsf{PES}}]$ 
is finite. Then,  
\begin{align} 
  \|\wn\|_{L^{\infty}((0,T);L^2(\Omega))} + \|\wn\|_{L^{2}((0,T); H^1(\Omega))} &\leq C_T  \big( \varepsilon + \sqrt{\varepsilon}  \, 
\mathrm{dist}  [(n_0,c_0,w_0);\mathcal C_{\mathsf{PES}}]  \big),
% \|n_\varepsilon-n\|_{L^{\infty}((0,T);L^2(\Omega))} + \|n_\varepsilon-n\|_{L^{2}((0,T); H^1(\Omega))} &\leq C_T \left( \varepsilon + \sqrt{\varepsilon}
% \mathrm{dist}^{k,p}[c_0;\mathcal C_{\mathsf{PES}}] + \sqrt{\varepsilon} \varepsilon_{w}^{\mathsf{in}} \right),  
\label{Theo:RC:Parta:1}
\end{align}
and 
\begin{align}  
\|\ww\|_{L^{\infty}((0,T);H^1(\Omega))} + \|\ww\|_{L^{2}((0,T);H^2(\Omega))}   \le \,&\, C_{T,\tau}  \big( \varepsilon +  
\mathrm{dist}^{0,1}_2 [(n_0,c_0,w_0);\mathcal C_{\mathsf{PES}}]  \big),  
    % \|w_\varepsilon-w\|_{L^{\infty}((0,T);H^1(\Omega))}   &\hspace{-0.2cm}\leq&\hspace{-0.2cm} C_{T,\tau} \left( \varepsilon  + \sqrt{\varepsilon}   \mathrm{dist}^{k,p}[c_0;\mathcal C_{\mathsf{PES}}]   + \varepsilon_{w}^{\mathsf{in}}  \right), \\
    % \|w_\varepsilon-w\|_{L^{2}((0,T);H^2(\Omega))} &\hspace{-0.2cm}\leq&\hspace{-0.2cm}  C_{T,\tau} \left( \varepsilon  + \sqrt{\varepsilon}   \mathrm{dist}^{k,p}[c_0;\mathcal C_{\mathsf{PES}}]   + \sqrt{\varepsilon} \varepsilon_{w}^{\mathsf{in}}  \right),   
\label{Theo:RC:Parta:2} \\
    \|\wc\|_{L^{\infty}((0,T);H^2(\Omega))} + \|\wc\|_{L^{2}((0,T);H^3(\Omega))} \; \le \,&\,  C_{T,\tau}  \big( \varepsilon +  
\mathrm{dist}^{2,1}_2 [(n_0,c_0,w_0);\mathcal C_{\mathsf{PES}}]  \big) . 
\label{Theo:RC:Parta:3}
\end{align}
 
\noindent b) Assuming that 
the distance $\mathrm{dist}^{4,2}_p [(n_0,c_0,w_0);\mathcal C_{\mathsf{PES}}]$ { is finite for some}  $2\le p<\infty$. Then, 
\begin{align}\label{est_rate_convergence_3_L2Lq}
\|\wn\|_{L^{\infty}((0,T); L^{p}(\Omega))} &\leq C_{p,T,\tau}  \left( \varepsilon^{\frac{2}{p}}  + \varepsilon^{\frac{1}{p}} \big(
\mathrm{dist}  [(n_0,c_0,w_0);\mathcal C_{\mathsf{PES}}]\big)^{\frac{2}{p}} \right)  , 
		\end{align}
and
\begin{align}
\begin{array}{lllll}
\displaystyle \| \ww \|_{L^p((0,T);W^{2,p}(\Omega))} &\hspace{-0.2cm} \displaystyle \le C_{p,\tau,T}   \left( \varepsilon^{\frac{2}{p}} +
 \varepsilon^{\frac{1}{p}} \big(
\mathrm{dist}^{0,2}_p  [(n_0,c_0,w_0);\mathcal C_{\mathsf{PES}}]\big)^{\frac{2}{p}} \right), \vspace{0.1cm} \\
\displaystyle \| \wc  \|_{L^p((0,T);W^{4,p}(\Omega))} &\hspace{-0.2cm} \displaystyle \le C_{p,\tau,T}   \left( \varepsilon^{\frac{2}{p}} +
 \varepsilon^{\frac{1}{p}} \big(
\mathrm{dist}^{4,2}_p  [(n_0,c_0,w_0);\mathcal C_{\mathsf{PES}}]\big)^{\frac{2}{p}} \right). 
\end{array} 
\label{est_rate_convergence2b}
\end{align} 

\end{theorem}

\begin{remark}\hfill
    \begin{itemize}
        \item {In the above estimates, the general constants $C_{T,\tau}$, $C_{p,T,\tau}$ may tend to infinity as $\tau\to 0$.}
        \item Thanks to the estimate \eqref{Theo:RC:Parta:1}, the rate $\|\wn\|_{L^{\infty}((0,T);L^2(\Omega))}$  is of order $O(\varepsilon)$ if the distance $\mathrm{dist}  [(n_0,c_0,w_0);\mathcal C_{\mathsf{PES}}]$ is at least of the order $\sqrt{\varepsilon}$. Even if $\mathrm{dist}  [(n_0,c_0,w_0);\mathcal C_{\mathsf{PES}}]$ is large (i.e., the system starts far away from the critical manifold $\mathcal C_{\mathsf{PES}}$), $n_\varepsilon$ always converges to $n$ in $L^{\infty}((0,T);L^2(\Omega))$ at least in the order $O(\sqrt{\varepsilon})$.     
        However, this is {not true} for $c_\varepsilon$ and $w_\varepsilon$.

        In \cite{reisch2024global}, it has been shown numerically that, if a system starts far away from its critical manifold, then the slow evolution's components do not converge to their expected limits in $L^\infty((0,T);L^2(\Omega))$, {and, in contrast, the distances between the solutions can be even sufficiently large}.
        
        \item Since the initial conditions are a major difference between the $\varepsilon$-dependent and limiting systems, a non-zero distance from the initial data to  $\mathcal C_{\mathsf{PES}}$ corresponds to an initial layer. Therefore, Theorem \ref{Theo:FSL:3} also claims that the parabolic-elliptic system    \eqref{Sys:Lim:Main}-\eqref{Sys:Lim:InitCond} is a ``good" approximation of the parabolic-parabolic system \eqref{Sys:Eps:Main}-\eqref{Sys:Eps:InitCond} whenever there is no initial layer, which is recently discussed in \cite{li2024characterization}. This suggests that skipping the slow time evolution should be associated with well-prepared initial data. 
    \end{itemize}
\end{remark}

As discussed after Theorem \eqref{Theo:FSL:1}, we see that the weak convergence in \eqref{Theo:FSL:1:S3} can, in fact, be proved in the strong topology. The following corollary is \textit{understood} as the strong convergence to the critical manifold $C_{\textsf{PES}}$.
\begin{corollary}[Strong convergence to the critical manifold] 
\label{Coro:1} For each $\varepsilon>0$, let $(n_\varepsilon,c_\varepsilon,w_\varepsilon)$ be the global classical solution to the system \eqref{Sys:Eps:Main}-\eqref{Sys:Eps:InitCond}. {Then it holds} 
\begin{align}
    \|\Delta c_\varepsilon - c_\varepsilon + w_\varepsilon \|_{L^2((0,T);H^1(\Omega))} + \|
    \tau\Delta w_\varepsilon - w_\varepsilon + n_\varepsilon \|_{L^2(\Omega_T)} \le C_{T,\tau}  \sqrt{\varepsilon} .
    \label{Coro:1:S}
\end{align}
{Furthermore, if $\mathrm{dist}^{0,1}_2 [(n_0,c_0,w_0);\mathcal C_{\mathsf{PES}}]= O(\eps)$ then we have the improved convergence rate}
\begin{equation*}
    {\|\Delta c_\varepsilon - c_\varepsilon + w_\varepsilon \|_{L^2((0,T);H^1(\Omega))} + \|
    \tau\Delta w_\varepsilon - w_\varepsilon + n_\varepsilon \|_{L^2(\Omega_T)} \le C_{T,\tau}  {\varepsilon}.}
\end{equation*}
\end{corollary}

\begin{proof} By the triangle inequality and the fact that $\tau\Delta w - w + n=0$, 
\begin{align*}
   \|\tau\Delta w_\varepsilon - w_\varepsilon + n_\varepsilon\|_{L^2(\Omega_T)} \le \,&\,  {\|\tau\Delta \ww - \ww + \wn \|_{L^2(\Omega_T)}}\\
   \le \,&\, {\tau\|\ww\|_{L^2(0,T;H^2(\Omega))} + \|\ww\|_{L^2(\Omega_T)} + \|\wn\|_{L^2(\Omega_T)}} \\
   \le \,&\, C_{T,\tau}  \big( \varepsilon +  \sqrt{\eps}
\mathrm{dist}^{0,1}_2 [(n_0,c_0,w_0);\mathcal C_{\mathsf{PES}}]  \big)
\end{align*}
thanks to \eqref{est_rate_convergence_3_L2Lq} and \eqref{est_rate_convergence2b}. The convergence for $\Delta c_\varepsilon - c_\varepsilon + w_\varepsilon$ follows similarly.
\end{proof}

\medskip
{\textbf{Our second main results} concerning rigorous IDS} for \eqref{Sys:Eps:Main}-\eqref{Sys:Eps:InitCond} will be presented in Theorem \ref{Theo:I2D}. 
{More precisely}, 
 we study the limit as  $\kappa =( \varepsilon,\tau) \to (0,0)$, or in other words, both parameters $\varepsilon$ and $\tau$ tend to zero at the same time. Here, the subscript $\kappa$ in $(n_\kappa,c_\kappa,w_\kappa)$ is used to indicate the dependence of the solution on both parameters. We formally expect   \begin{gather}
(n_\kappa,c_\kappa,w_\kappa) \rightarrow (n,c,w) \; \text{and} \; 
(\varepsilon \partial_t c_\kappa, \varepsilon \partial_t w_\kappa - \tau \Delta w_\kappa)=(\Delta c_\kappa - c_\kappa + w_\kappa, - w_\kappa + n_\kappa)  \rightarrow (0,0), 
\label{Conver:Crit2}
\end{gather}   
and subsequently, at the limit level $w=n$. Therefore, the vector $(n,c)$ is expected to be the  solution to  
\begin{align}\label{Sys:LimKappa:Main}
	\left\{ \begin{array}{lllllll}
		\partial_t n = \Delta n -  \nabla \cdot (n \nabla c) & \text{in } \Omega \times(0,\infty) , \\ 
	\Delta c - c + n = 0 & \text{in } \Omega \times(0,\infty), \vspace{0.1cm} \\ 	
    \dfrac{\partial n}{\partial \nu} = \dfrac{\partial c}{\partial \nu}  = 0 & \text{on } \Gamma \times(0,\infty), \vspace{0.1cm} \\
    n|_{t=0}=n_{0} &  \text{on }  \Omega , 
	\end{array} \right.
\end{align}
which describes a direct signalling mechanism and is {the well-known} \textit{Keller-Segel system}. Particularly, if $\tau=\varepsilon$, or  $\tau,\varepsilon$ are given in the same time scale, the equation for $w_\varepsilon$ can be rewritten as 
\begin{align*}
    \partial_t w_\varepsilon - \Delta w_\varepsilon = - \frac{1}{\varepsilon} \left(w_\varepsilon - n_\varepsilon \right) 
\end{align*}
in which the kinetics of $w_\varepsilon$ is on a much faster time scale compared to its evolution and diffusion. The limit as $\varepsilon\to 0$ then falls into the topic of \textit{fast reaction limits}, which has usually been studied in reaction-diffusion systems with fast interaction,     see e.g. \cite{bothe2012cross,perthame2023fast,tang2024rigorous,morgan2024singular}, and recently in chemotaxis systems 
\cite{laurenccot2024singular,li2023convergence}. {To rigorously prove IDS,} similarly to Theorem \ref{Theo:FSL:1}, it is important to control the Lyapunov functional $\mathcal E(n_\kappa,c_\kappa)$ as well as to obtain the uniform-in-$\kappa$ {estimates} in  $L^\infty(\Omega_T)$, {and therefore, we face similar challenges as in the first part}. {Furthermore, due to $\tau\to 0^+$, the Lyapunov structure from Theorem \ref{Theo:FSL:1} only gives} the uniform-in-$\kappa$ boundedness in $L^\infty((0,T);H^1(\Omega))$ since the term of second order derivatives of $c_\kappa$ now depends explicitly on $\kappa$. {Obtaining uniform-in-$\kappa$ estimates is quite tricky since now both $\eps$ and $\tau$ can be degenerate. Our idea is to} adapt the bootstrap argument proposed in \cite{morgan2024singular}. The starting point in this argument is given in Lemma \ref{L:PL:Est:nL2+}, where we show there is {a small constant} $\delta>0$ such that
\begin{align*}
\sup_{\kappa\in(0,\infty)^2} \left( \esssup_{t\in(0,T)} \intO n_\kappa^{1+\frac{\delta}{2}}(t) + \|n_\kappa \|_{L^{2+\delta}(\Omega_T)} + \iint_{\Omega_T} n_\kappa ^{\frac{\delta}{2}-1} |\nabla n_\kappa |^2 \right) \le C_{T}.
\end{align*}
Then, based on a combination of the heat regularisation, the Gagliardo-Nirenberg inequality, as well as the maximal regularity with slow evolution, we {obtain a recursive} increasing sequence $\{p_j\}_{j=0,1,...}$ with $p_0:=1+\delta/2$ satisfying:  if  
\begin{align*}
    \sup_{\kappa \in (0,\infty)^2} \left( \|n_\kappa \|_{L^{2p_j}(\Omega_T)} \right) \le C_T, 
\end{align*}
then 
\begin{align*}
    \sup_{\kappa \in (0,\infty)^2} \left( \esssup_{t\in(0,T)} \int_\Omega n_\kappa ^{p_{j+1}}(t) +  \|n_\kappa \|_{L^{2p_{j+1}}(\Omega_T)} \right) \le C_{T,p_{j+1}}, 
\end{align*}
see Lemma \ref{L:I2D:Feed}. {This is  sufficient to perform a bootstrap argument} to have the uniform-in-$\kappa$ $L^p(\Omega_T)$-boundedness for any $1\le p < \infty$ that turns into the $L^\infty(\Omega_T)$-boundedness due to the use of the Neumann heat semigroup.   
Finally, the convergence rate is obtained similarly to Theorem \ref{Theo:FSL:3} by tracking carefully the dependence of all constants on both  $\varepsilon$ and $\tau$, as well as the distance from the initial data to the critical manifold $\mathcal C_{\mathsf{IDS}}$, {which is defined by}
\begin{align}
    \mathcal C_{\mathsf{IDS}}:= \Big\{(n,c,w) \in L^2(\Omega) \times H^2(\Omega)  \times L^2(\Omega) : \, (\Delta c - c + w , - w + n) =(0,0)  \Big\}. 
\end{align}
The distance $\mathrm{dist}^{k,l}_p [(n_0,c_0,w_0);\mathcal C_{\mathsf{IDS}}]$ is defined similarly to \eqref{DistanceDef} due to the replacement of $\mathcal C_{\mathsf{PES}}$ by $\mathcal C_{\mathsf{IDS}}$.

 \begin{theorem}[IDS for \eqref{Sys:Eps:Main}]
\label{Theo:I2D} Let $N=1,2$. Assume that $(n_0,c_0,w_0)$ is complied with Assumption \ref{Ass:InitialData}, and furthermore in the critical dimension $N=2$ that 
\begin{align}
    M {:= \int_{\Omega}n_0} < 4 \pi .
\label{Theo:I2D:MassCond}
\end{align}
For each   $\kappa= (\varepsilon,\tau)\in (0,\infty)^2$, let $(n_\kappa,c_\kappa,w_\kappa)$ be the global classical solution to the system \eqref{Sys:Eps:Main}-\eqref{Sys:Eps:InitCond}, given by Theorem  \ref{Theo:GloEX}. 
 Then, for any $0<T<\infty$,    
\begin{align}
\begin{gathered}
   \sup_{\kappa\in(0,\infty)^2} \Big( \|n_\kappa\|_{C^{\gamma,\gamma/2}(\overline{\Omega}\times[0,T])} + \|n_\kappa\|_{L^2((0,T);H^1(\Omega))}  \Big) \le C_{T},\\ 
    \sup_{\kappa\in(0,\infty)^2} \Big(
 \|w_\kappa \|_{L^\infty(\Omega_T)} + \|w_\kappa \|_{L^2((0,T);H^{1}(\Omega))} \Big) \le C_{T},   
  \\
\sup_{\kappa\in(0,\infty)^2} \Big(
 \|c_\kappa \|_{L^\infty((0,T);W^{1,\infty}(\Omega))} + \|c_\kappa \|_{L^p((0,T);W^{2,p}(\Omega))} \Big) \le C_{T,p}.   
\end{gathered}
\label{Theo:I2D:S1}
\end{align}
for some $\gamma\in(0,1)$ and any $1\le p<\infty$. As $\kappa= (\varepsilon,\tau)\to (0,0)$, {we have the following limits}
\begin{align}
    \begin{aligned}
\begin{array}{llllcl}
n_\kappa  &\longrightarrow& n & \text{strongly in}& C(\overline{\Omega}\times [0,T]), \\
\nabla n_\kappa &\xrightharpoonup{\hspace{0.4cm}}& \nabla n & \text{weakly in}& L^2(\Omega_T) , \\
c_\kappa  &\longrightarrow& c & \text{strongly in}& L^2((0,T);H^1(\Omega)), \\
w_\kappa  &\longrightarrow& w & \text{strongly in}& L^2(\Omega_T),
\end{array}
\end{aligned}
\label{Theo:FSL:3:S2}
\end{align}
% and 
% \begin{align}
% \begin{gathered}
%      \|\varepsilon \partial_t c_\kappa  \|_{L^2((0,T);H^1(\Omega))} = \|\Delta c_\kappa - c_\kappa + w_\kappa \|_{L^2((0,T);H^1(\Omega))} \le C \sqrt{|\kappa|} , \\
%      \varepsilon \partial_t w_\kappa - \tau \Delta w_\kappa =  - w_\kappa + n_\kappa \quad \xrightharpoonup{\hspace{0.38cm}} \quad 0  \text{ weakly in } L^2(\Omega_T),  
% \end{gathered}
% \label{Theo:FSL:3:S3}
% \end{align}
% {where $|\kappa| = \eps + \tau$} (\textcolor{red}{Question: from Corollary 2 we in fact have the last convergence is strong. Why do we state it only in the weak sense here?})
and the limit vector $(n,c)$ is the unique global classical solution to the direct signalling parabolic-elliptic system \eqref{Sys:LimKappa:Main}. 
Moreover, assuming that the distance $\mathrm{dist}^{1,0}_2 [(n_0,c_0,w_0);\mathcal C_{\mathsf{IDS}}]$ is finite. Then, {for $|\kappa| = \eps + \tau$} we have
\begin{align} 
\|\wn\|_{L^{\infty}((0,T);L^2(\Omega))} + \|\wn\|_{L^{2}((0,T); H^1(\Omega))} &\leq C_T \left( |\kappa| + \sqrt{|\kappa|} \,
\mathrm{dist}[(n_0,c_0,w_0);\mathcal C_{\mathsf{IDS}}] \right),  
\label{Theo:FSL:3:1}
\end{align}
and 
\begin{align}  
    \|\ww\|_{L^{\infty}((0,T);L^2(\Omega))} + 
    \|\ww\|_{L^{2}((0,T);H^1(\Omega))} &\le   C_{T} \Big( |\kappa| + \mathrm{dist}[(n_0,c_0,w_0);\mathcal C_{\mathsf{IDS}}]  \Big),  
\label{Theo:FSL:3:2} \\ 
    \|\wc\|_{L^{\infty}((0,T);H^1(\Omega))} + \|\wc\|_{L^{2}((0,T);H^2(\Omega))} &\le   C_{T} \Big( |\kappa| + \mathrm{dist}^{1,0}_2[(n_0,c_0,w_0);\mathcal C_{\mathsf{IDS}}]  \Big).  
\label{Theo:FSL:3:3}
\end{align}
\end{theorem}
The case $\tau=0$ was investigated in \cite{li2023convergence}, where only the convergence of  $n_\varepsilon$ to $n$ as $\varepsilon \to 0$ had been showed in a strong sense while $c_\varepsilon \xrightharpoonup{} c$ weakly in $L^4((0,T);W^{1,4}(\Omega))$ and weakly-star in $L^\infty((0,T);H^2(\Omega))$ and  
$w_\varepsilon \xrightharpoonup{} w$ weakly-star in $L^\infty(\Omega_T)$. {Our results improve those of \cite{li2023convergence} by proving this convergence in the strong topology, and furthermore provide the convergence rate}. Similarly to Corollary \ref{Coro:1}, we have the following strong convergence to the critical manifold $\mathcal{C}_{\mathsf{IDS}}$.
\begin{corollary}[Strong convergence to the critical manifold] 
\label{Coro:2} For each   $\kappa= (\varepsilon,\tau)\in (0,\infty)^2$, let $(n_\kappa,c_\kappa,w_\kappa)$ be the globally classical solution to the system \eqref{Sys:Eps:Main}-\eqref{Sys:Eps:InitCond}. Then,
\begin{align*}
    \|\Delta c_\kappa - c_\kappa + w_\kappa \|_{L^2((0,T);H^1(\Omega))} + \|- w_\kappa + n_\kappa\|_{L^2(\Omega_T)} \le C \sqrt{|\kappa|}. 
\end{align*}
% {If $\text{dist}[(n_0,c_0,w_0);\mathcal{C}_{\mathsf{IDS}}] = 0$, we have the improved convergence rate}
% \begin{align*}
%     {\|\Delta c_\kappa - c_\kappa + w_\kappa \|_{L^2((0,T);H^1(\Omega))} + \|- w_\kappa + n_\kappa\|_{L^2(\Omega_T)} \le C {|\kappa|}}. 
% \end{align*}
\end{corollary}

\medskip
{\bf The rest of this paper is organised as follows}: In Section \ref{sec:PES}, we rigorously simplify from \eqref{Sys:Eps:Main}-\eqref{Sys:Eps:InitCond} to \eqref{Sys:Lim:Main}-\eqref{Sys:Lim:InitCond} in which both subcritical case $1\le N\le 3$ and critical case $N=4$ are considered. The accuracy of this simplification is studied in Section \ref{sec:PES_rate}. In Section \ref{sec:IDS}, the analysis of the indirect-direct simplification from \eqref{Sys:Eps:Main}-\eqref{Sys:Eps:InitCond} to \eqref{Sys:LimKappa:Main}, as well as its accuracy, will be investigated. Finally, we place some auxiliary results in the Appendix \ref{sec:appendix}.

\section{Rigorous parabolic-elliptic simplification}\label{sec:PES}
  
%\subsection{Global solvability for each relaxation parameter's value}
 
{We start this section by the global existence and boundedness of solutions to \eqref{Sys:Eps:Main}-\eqref{Sys:Eps:InitCond} for fixed $\eps>0$ and $\tau>0$, which is done in \cite{fujie2017application}}. {We remark that the constant $C_{\varepsilon,\tau}$ in the following theorem may tend to infinity as either $\varepsilon \to 0$ or $\tau \to 0$.}
 
\begin{theorem}[{\cite[Theorem 1.1]{fujie2017application}}]
\label{Theo:GloEX}
 Suppose that \begin{equation}\label{Condition_initial}
n_0, \, c_0, \,w_0 \geq 0 \textrm{ on } \overline{\Omega}, \textrm{ and } n_0 \in C(\bar{\Omega}), \, c_0, w_0 \in C^2(\overline{\Omega}).
\end{equation} 
For each pair $(\varepsilon,\tau) \in (0,\infty)^2$, System \eqref{Sys:Eps:Main}-\eqref{Sys:Eps:InitCond} admits a unique classical positive solution $(n,c,w)$ which exists globally in time. Moreover, it satisfies 
 \begin{equation}
 	\sup_{t \in [0,\infty)} \Big{(} \|n(t)\|_{L^{\infty}(\Omega)} + \|c(t)\|_{W^{2,\infty}(\Omega)} + \|w(t)\|_{W^{2,\infty}(\Omega)} \Big{)} \le C_{\varepsilon,\tau} < \infty.
 \end{equation}
\end{theorem}

\subsection{Multiple time scale Lyapunov functional}
\label{Sec:Lya}
 By integrating the equation for $n_\varepsilon$ and using the homogeneous Neumann boundary condition, {we have the conservation}
\begin{align}
\int_\Omega n_\varepsilon(x,t) = \int_\Omega n_0(x) = M, \quad \text{for all } t\ge 0,
\label{FSL:TotalMass}
\end{align}
which also reads that  $n_\varepsilon$ is uniformly-in-$\varepsilon$ bounded in  $L^\infty((0,T);L^1(\Omega))$. However, this regularity is not sufficient to gain necessary estimates for $w_\varepsilon,c_\varepsilon$ and then improve again the uniform-in-$\varepsilon$ regularity of $n_\varepsilon$. In this part, we present an a priori estimate for solutions by considering a Lyapunov functional according to the system structure. Since the equation for $n_\varepsilon$ can be rewritten as  
$$ \partial_t n_\varepsilon   = \nabla \cdot \big( n_\varepsilon \nabla (\log n_\varepsilon-c_\varepsilon) \big), $$  we 
multiply two sides by $(\log n_\varepsilon - c_\varepsilon)$ and integrate over the spatial domain to get that 
\begin{align*}
\int_\Omega \partial_t n_\varepsilon (\log n_\varepsilon -c_\varepsilon)  = - \int_\Omega n_\varepsilon|\nabla (\log n_\varepsilon-c_\varepsilon)|^2 .
\end{align*}
This suggests considering the Lyapunov functional below for $n_\varepsilon$ 
\begin{align*}
E(n_\varepsilon)=\int_\Omega n_\varepsilon(\log n_\varepsilon - c_\varepsilon),
\end{align*}
which, after differentiating in time and taking into account that $\int_\Omega \partial_t n_\varepsilon =0$, gives
\begin{align}
\frac{d}{dt} E(n_\varepsilon) = - \int_\Omega n_\varepsilon|\nabla (\log n_\varepsilon-c)|^2 -  \int_\Omega n_\varepsilon \partial_t c_\varepsilon . 
\label{FSL:Lya:E}
\end{align}
An estimate for this type of functional was established corresponding to $N=2$ and $\tau=0$ in \cite[Section 4.1]{li2023convergence}. {The analysis in our case is significantly more challenging since $\tau>0$ and $1\le N\le 4$, where $N=4$ is the critical dimension}. {Concerning} the last term of \eqref{FSL:Lya:E}, {we have the following computations}. 

\begin{lemma} 
\label{L:FSL.Compute:nc_t}
For $t>0$, it holds 
\begin{align}
\begin{aligned}
- \int_\Omega n_\varepsilon \partial_t c_\varepsilon   = \,&\, - \frac{d}{dt} \int_\Omega \left( \frac{1}{2}|\Delta c_\varepsilon - c_\varepsilon + w_\varepsilon|^2 + \frac{\tau}{2} |\Delta c_\varepsilon|^2 + \frac{1+\tau}{2} |\nabla c_\varepsilon|^2 + \frac{1}{2} c_\varepsilon^2 \right) \\
\,&\, -  \frac{1+\tau}{\varepsilon} \intO|\nabla(\Delta c_\varepsilon - c_\varepsilon + w_\varepsilon)|^2 - \frac{2}{\varepsilon} \intO |\Delta c_\varepsilon - c_\varepsilon + w_\varepsilon|^2 .
\end{aligned}
\label{L:FSL.Compute:nc_t.S1}
\end{align}
\end{lemma}

\begin{proof} Using the equation for $c_\varepsilon$, we have 
\begin{align*}
\left\{ \begin{array}{rlll}
\varepsilon \partial_t w_\varepsilon &=& \varepsilon^2 \partial_{tt}^2 c_\varepsilon - \varepsilon  \Delta \partial_t c_\varepsilon + \varepsilon \partial_t c_\varepsilon, \\
\tau \Delta w_\varepsilon &=& \tau\varepsilon \Delta \partial_t c_\varepsilon - \tau \Delta^2 c_\varepsilon + \tau \Delta c_\varepsilon, \\
w_\varepsilon &=& \varepsilon \partial_t c_\varepsilon - \Delta c_\varepsilon +  c_\varepsilon. 
\end{array}\right.
\end{align*}
Then, we imply from the equation for $w_\varepsilon$ that  
\begin{align*}
n_\varepsilon = \varepsilon^2 \partial_{tt}^2 c_\varepsilon - (1+\tau)\varepsilon \Delta \partial_t c_\varepsilon + 2\varepsilon \partial_t c_\varepsilon + \tau \Delta^2  c_\varepsilon - (1+\tau)\Delta c_\varepsilon + c_\varepsilon. 
\end{align*}
Therefore, due to the integration by parts, 
\begin{align*}
- \int_\Omega n_\varepsilon \partial_t c_\varepsilon 
&= - \int_\Omega \Big( \varepsilon^2 \partial_{tt}^2 c_\varepsilon - (1+\tau)\varepsilon \Delta \partial_t c_\varepsilon + 2\varepsilon \partial_t c_\varepsilon + \tau \Delta^2  c_\varepsilon - (1+\tau)\Delta c_\varepsilon + c_\varepsilon \Big) \partial_t c_\varepsilon \\
&= - \frac{d}{dt} \left( \frac{\varepsilon^2}{2} \int_\Omega |\partial_tc_\varepsilon|^2 + \frac{\tau}{2} \int_\Omega |\Delta c_\varepsilon|^2 + \frac{1+\tau}{2} \int_\Omega |\nabla c_\varepsilon|^2 + \frac{1}{2} \int_\Omega c_\varepsilon^2 \right) \\
& \hspace{0.45cm} - \left((1+\tau)\varepsilon \int_\Omega |\nabla \partial_t c_\varepsilon|^2 + 2\varepsilon \int_\Omega |\partial_t c_\varepsilon|^2 \right).  
% & = - \frac{d}{dt} \left( \frac{1}{2} \int_\Omega  |\Delta c_\varepsilon - c_\varepsilon +  w_\varepsilon|^2 + \frac{\tau}{2} \int_\Omega  |\Delta c_\varepsilon|^2 + \frac{1+\tau}{2} |\nabla c_\varepsilon|^2 +  \frac{1}{2} \int_\Omega c_\varepsilon^2 \right) \\
% & \hspace{0.45cm} -  \frac{1+\tau}{\varepsilon} \intO|\nabla(\Delta c_\varepsilon - c_\varepsilon + w_\varepsilon)|^2 - \frac{2}{\varepsilon} \intO |\Delta c_\varepsilon - c_\varepsilon + w_\varepsilon|^2 ,
\end{align*}
{By using the equation for $c_\varepsilon$ at the last step, we obtain \eqref{L:FSL.Compute:nc_t.S1}}. 
\end{proof}

The time derivatives appearing above suggest that a combination of $n_\varepsilon(\log n_\varepsilon - c_\varepsilon)$ and $$\frac{1}{2}|\Delta c_\varepsilon - c_\varepsilon + w_\varepsilon|^2 + \frac{\tau}{2} |\Delta c_\varepsilon|^2 + \frac{1+\tau}{2} |\nabla c_\varepsilon|^2 + \frac{1}{2} c_\varepsilon^2$$ forms the relevant structure of a multiple time scale  Lyapunov functional for the whole system. 
The following lemma is a direct consequence of Lemma \ref{L:FSL.Compute:nc_t} and the identity \eqref{FSL:Lya:E}.

\begin{lemma}\label{L:FSL:Lya:Id}
For $t>0$, it holds  
\begin{equation}
\frac{d}{dt} \mathcal{E}(n_{\varepsilon}(t),c_{\varepsilon}(t)) = -  \mathcal{D}(n_{\varepsilon}(t), c_{\varepsilon}(t)) \le 0
\label{L:FSL:Lya:Id:S}
\end{equation}
{where $\mathcal{E}(n_{\varepsilon},c_{\varepsilon})$ and $\mathcal{D}(n_{\varepsilon}, c_{\varepsilon})$ are defined in \eqref{FSL:Lya:Define} and \eqref{FSL:Lya:Id}, respectively.}
\end{lemma} 

Lemma \ref{L:FSL:Lya:Id} suggests an estimate  for $c_\varepsilon$ in $L^\infty((0,T);H^2(\Omega))$ uniformly in $\varepsilon$, as well as in $L^\infty((0,T);H^1(\Omega))$ uniformly in $\kappa$. However, we note here that the lower boundedness of $\mathcal{E}$ has not been guaranteed since it contains $-n_\varepsilon c_\varepsilon$. Therefore, to apply Lemma \ref{L:FSL:Lya:Id}, a lower bound for $-n_\varepsilon c_\varepsilon$ or $ n_\varepsilon (\log n_\varepsilon - c_\varepsilon)$ in $L^1(\Omega_T)$ must be established first. {This will be done separately for the cases $1\le N\le 3$ and $N=4$ in the following subsections, as the latter case is in the critical dimension and a different strategy needs to be employed.}

\subsection{The case of subcritical dimensions $1\le N \le 3$}
\label{Sec:LimEp:Sub}  

\subsubsection{Balancing the Lyapunov functional}

\begin{lemma}  
\label{L:FSL:Est:c}  
There exists a constant   $C=C_{n_0,c_0,\Omega,M}>0$ independently of $\varepsilon,\tau$ such that 
\begin{align}
\sup_{\varepsilon>0} \bigg(  \sup_{t>0} \int_\Omega \left( \frac{\tau}{4} |\Delta c_\varepsilon(t)|^2 + \frac{2+\tau}{4} |\nabla c_\varepsilon(t)|^2 + \frac{2-\tau}{4} c_\varepsilon^2(t) \right) \bigg) \le  \frac{C}{\tau} ,
\label{L:FSL:DissIneqn}
\end{align}  
and 
\begin{align}
    \begin{aligned}
      \sup_{\varepsilon>0} \left( \frac{1}{\varepsilon} \iint_{\Omega_T} \Big( |\nabla(\Delta c_\varepsilon - c_\varepsilon + w_\varepsilon)|^2 +  |\Delta c_\varepsilon - c_\varepsilon + w_\varepsilon|^2 \Big) \right) \le C_{\tau}.
\end{aligned}
\label{L:FSL:Est:c:S2}
\end{align}
\end{lemma}

\begin{proof} Under the assumption \ref{Ass:InitialData} on the initial data $(n_0,c_0)$, the term  $\mathcal E(n_0,c_0)$ is clearly finite.  By Lemma \ref{L:FSL:Lya:Id}, for all $t>0$,  
\begin{equation*}
\mathcal{E}(n_{\varepsilon}(t),c_{\varepsilon}(t)) \leq  \mathcal{E}(n_0,c_0) - \int_0^t \mathcal D(n_\varepsilon(s),c_\varepsilon(s)) ,  
\end{equation*}
in more detail, which is equivalent to 
\begin{align}
\begin{gathered}
\int_\Omega \left(  (n_\varepsilon \log n_\varepsilon + e^{-1} ) +   \frac{1}{2}|\Delta c_\varepsilon - c_\varepsilon + w_\varepsilon|^2 + \frac{\tau}{2} |\Delta c_\varepsilon|^2 + \frac{1+\tau}{2} |\nabla c_\varepsilon|^2 + \frac{1}{2} c_\varepsilon^2 \right) \\
\le \mathcal E(n_0,c_0) - \int_0^t \mathcal D(n_\varepsilon(s),c_\varepsilon(s)) +  e^{-1}|\Omega| + \int_\Omega n_\varepsilon c_\varepsilon.
\end{gathered}
\label{L:FSL:DissIneqn.P1}
\end{align}
It is necessary to estimate the product $n_\varepsilon c_\varepsilon$ in $L^\infty((0,T);L^1(\Omega))$. By the Sobolev embedding $H^2(\Omega) \hookrightarrow L^\infty(\Omega)$, we have 
\begin{align*}
\int_\Omega n_\varepsilon c_\varepsilon \le M \|c_\varepsilon\|_{L^{\infty}(\Omega)} \le C M \|c_\varepsilon\|_{H^{2}(\Omega)} \le \frac{\tau}{4} \|c_\varepsilon\|_{H^{2}(\Omega)}^2 + \frac{C^2M^2}{\tau}. 
\end{align*}
Therefore, we deduce from \eqref{L:FSL:DissIneqn.P1} that 
\begin{align}
\begin{gathered}
\int_\Omega \left( (n_\varepsilon \log n_\varepsilon + e^{-1} ) +  \frac{1}{2}|\Delta c_\varepsilon - c_\varepsilon + w_\varepsilon|^2 + \frac{\tau}{4} |\Delta c_\varepsilon|^2 + \frac{2+\tau}{4} |\nabla c_\varepsilon|^2 + \frac{2-\tau}{4} c_\varepsilon^2 \right) \\
+ \int_0^t \mathcal D(n_\varepsilon(s),c_\varepsilon(s)) \le \mathcal E(n_0,c_0)   + |\Omega| - M  + \frac{C^2M^2}{\tau},
\end{gathered}
\label{L:FSL:DissIneqn.P2}
\end{align}
and hence, estimate \eqref{L:FSL:DissIneqn} follows. In particular, by paying attention to the last two terms of $\mathcal D(n_\varepsilon,c_\varepsilon)$,  we observe that 
\begin{align*}
      \frac{1+\tau}{\varepsilon} \intQt |\nabla(\Delta c_\varepsilon - c_\varepsilon + w_\varepsilon)|^2 + \frac{2}{\varepsilon} \intQt |\Delta c_\varepsilon - c_\varepsilon + w_\varepsilon|^2 \le \frac{C}{\tau} 
\end{align*}
and obtain \eqref{L:FSL:Est:c:S2}, 
where $C$ depends on $n_0,c_0,
\Omega$ and $M$ and does not on $\varepsilon,\tau$. 
\end{proof}

\subsubsection{Uniform boundedness in sup-norms} 

Thanks to the Sobolev embedding,  Lemma \ref{L:FSL:Est:c} implies that the $L^6(\Omega)^N$-norm of $\nabla v$ is uniformly-in-$(\varepsilon,t)$ bounded. This will help us to obtain the uniform-in-$\varepsilon$ boundedness of  $n_\varepsilon$ in $$L^\infty((0,\infty);L^2(\Omega)) \cap L^3(\Omega_T) \cap L^2((0,T);H^1(\Omega))$$ via testing {its equation} by $n_\varepsilon$, see Lemma \ref{L:FSL:nReg}. Moreover, in Lemma \ref{L:FSL:Est:n}, this boundedness of $\nabla c_\varepsilon$ will show the  uniform-in-$\varepsilon$ boundedness of $n_\varepsilon$ in $L^\infty(\Omega_T)$ via exploiting $L^p-L^q$ estimates for the Neumann heat semigroup.

\begin{lemma} 
\label{L:FSL:nReg} It holds     
    \begin{align}
\sup_{\varepsilon>0}   \left( \sup_{0<t<T} \intO  n_\varepsilon^2  +  \intQT n_\varepsilon^3 + \intQT |\nabla n_\varepsilon|^2 \right) \le C_\tau.
\label{L:FSL:nReg:State}
\end{align}  
\end{lemma}

\begin{proof} Multiplying the equation for $n_\varepsilon$ by itself, integrating by parts over $\Omega$ and using the Young inequality, we obtain 
\begin{equation*}
 \frac{d}{dt} \int_{\Omega} n_{\varepsilon}^2 +  \int_{\Omega} |\nabla n_{\varepsilon}|^2  \leq  \int_{\Omega}  n_{\varepsilon}^2 |\nabla c_{\varepsilon}|^2 ,
\end{equation*}
for all $t>0$. Then, by the H\"older inequality,  
\begin{equation}
\frac{d}{dt} \int_{\Omega} n_{\varepsilon}^2 + \int_{\Omega} |\nabla n_{\varepsilon}|^2  \leq \|n_{\varepsilon}\|_{L^{3}(\Omega)}^2 \|\nabla c_{\varepsilon}\|_{L^6(\Omega)^N}^2.
\label{L:FSL:Est:n:P1}
\end{equation}
Noting that the estimate \eqref{L:FSL:DissIneqn} and the Sobolev embedding imply the uniform boundedness for $\nabla c_{\varepsilon}$ in $L^\infty((0,T);W^{1,6}(\Omega)^N)$, where the bound is proportional to $1/\tau^2$. Moreover, by applying the Gagliardo-Nirenberg interpolation inequality,  
\begin{align*}
\|n_{\varepsilon}\|_{L^3(\Omega)}^2 \|\nabla c_{\varepsilon}\|_{L^6(\Omega)^N}^2 &\leq \frac{C}{\tau^2} \left( \frac{C}{\tau} \|\nabla n_{\varepsilon}\|^{4/5}_{L^2(\Omega)^N} \|n_{\varepsilon}\|^{1/5}_{L^1(\Omega)} + \|n_{\varepsilon}\|_{L^1(\Omega)} \right)^2,  
\end{align*} 
and by the Young inequality, 
\begin{align*}
\|n_{\varepsilon}\|_{L^3(\Omega)}^2 \|\nabla c_{\varepsilon}\|_{L^6(\Omega)^N}^2 \le \frac{C_M}{\tau^4} \|\nabla n_{\varepsilon}\|^{8/5}_{L^2(\Omega)^N} + \frac{C_M}{\tau^2}   \leq \frac{1}{2} \int_{\Omega} |\nabla n_{\varepsilon}|^2 + \frac{C_M(\tau^{18}+1)}{\tau^{20}}  .
\end{align*}
Hence, estimate \eqref{L:FSL:nReg:State} is obtained directly from \eqref{L:FSL:Est:n:P1}.
\end{proof}

\begin{lemma}\label{L:FSL:Est:n} For any $0<T<\infty$, it holds  
\begin{align*}
\sup_{\varepsilon>0} \Big( \|n_\varepsilon\|_{L^\infty(\Omega_T)} \Big) \le C_{\tau,T}.
\end{align*}
\end{lemma}
\begin{proof}
To prove the uniform-in-$\varepsilon$ boundedness of $n_\varepsilon$ in $L^\infty(\Omega_T)$, we will estimate the quantity
$$\Lambda_T := \sup_{0<t<T} \|n_{\varepsilon}(t)\|_{L^{\infty}(\Omega)}.$$  
Let $3<p<6$ and take $\frac{3}{2p}< \beta < \frac{1}{2}$. Then,  $D((-\Delta+I)^\beta) \hookrightarrow L^\infty(\Omega)$, thanks to Theorem 1.6.1 in \cite{henry2006geometric}.
Using the Duhamel formula and the estimate \eqref{GloEx:HS:Contraction} for the heat Neumann semigroup, we have that
\begin{align*}
\begin{aligned}
\|n_{\varepsilon}(t)\|_{L^{\infty}(\Omega)} & \leq \|e^{t\Delta}n_0\|_{L^{\infty}(\Omega)} + \bigg{\|}\int_0^t e^{(t-s)\Delta} \nabla \cdot \left(n_{\varepsilon}(s) \nabla c_{\varepsilon}(s)  \right) ds \bigg{\|}_{L^{\infty}(\Omega)}   \\ 
& \leq \|e^{t\Delta}n_0\|_{L^{\infty}(\Omega)} +  \int_0^t \big\| (-\Delta +I)^\beta e^{(t-s)\Delta} \nabla \cdot \left(n_{\varepsilon}(s) \nabla c_{\varepsilon}(s)  \right) \big\|_{L^{p}(\Omega)} ds  \\ 
& \leq \|n_0\|_{L^{\infty}(\Omega)} + C\int_0^t  (t-s)^{-\beta-\frac{1}{2} - \eta}  e^{-\lambda s} \|n_{\varepsilon}(s) \nabla c_{\varepsilon}(s)\|_{L^{p}(\Omega)^N} ds
\end{aligned} 
\end{align*}
for any $\eta>0$, being chosen later. Using the H\"older inequality, 
\begin{align}
& \| n_{\varepsilon}(s) \nabla c_{\varepsilon}(s)\|_{L^p(\Omega)^N}  \leq C  \|n_{\varepsilon}(s) \|_{L^{\frac{6p}{6-p}}(\Omega)} \|\nabla c_{\varepsilon}(s)\|_{L^{6}(\Omega)^N} \nonumber \\
& \leq C  \|n_{\varepsilon}(s) \|^{\frac{7p-6}{6p}}_{L^{\infty}(\Omega)} \|n_{\varepsilon}(s) \|^{\frac{6-p}{6p}}_{L^1(\Omega)} \|\nabla c_{\varepsilon}(s)\|_{L^6(\Omega)^N} \nonumber  \leq C_\tau   \Lambda_T^{\frac{7p-6}{6p}},
\end{align}
where $\sup_{t>0}\|\nabla c_{\varepsilon}(t)\|_{L^6(\Omega)^N}$ is bounded due to estimate \eqref{L:FSL:DissIneqn} and the Sobolev embedding for the dimensions $1\le N\le 3$. Combining the above estimates, we deduce that 
\[
\|n_{\varepsilon}(t)\|_{L^{\infty}(\Omega)} \leq C  + C_\tau  \Lambda_T^{\frac{7p-6}{6p}} \int_0^t  (t-s)^{-\beta-\frac{1}{2} - \eta}  e^{-\lambda s} ds.
\]
Since $\beta<1/2$, we can choose $\eta$ such that $\eta <1/2-\beta$, which guarantees that the above improper integral is finite. Thus, we obtain $\Lambda_T \leq C  + C_{\tau,T} \Lambda_T^{(7p-6)/(6p)},$   
and therefore, the quantity $\Lambda_T$ must be bounded since its exponent on the right-hand side is strictly less than $1$.  
\end{proof}

\begin{lemma}\label{L:FSL:Est:w} 
For any $1<p<\infty$, it holds
\begin{align}
\sup_{\varepsilon>0} \Big(
 \|w_\varepsilon \|_{L^\infty((0,T);W^{1,\infty}(\Omega))} + \|\Delta w_\varepsilon \|_{L^p(\Omega_T)} \Big) \le C_{\tau,T} , \label{L:FSL:Est:w:S1}
\end{align}
and 
\begin{align}
\sup_{\varepsilon>0} \Big(
 \|c_\varepsilon \|_{L^\infty((0,T);W^{2,\infty}(\Omega))}  \Big) \le C_{\tau,T} , \label{L:FSL:Est:w:S2}
\end{align}
\end{lemma}

\begin{proof} Thanks to the parabolic maximal regularity with slow evolution, cf. Lemma  \ref{L:MR:SlowEvolution}, applied to the equation for $w_\varepsilon$, we have 
\begin{align}
\|\Delta w_\varepsilon \|_{L^p(\Omega_T)}  \le  C_p \, \varepsilon^{\frac{1}{p}} \|\Delta w_0\|_{L^p(\Omega)} +  C_{p,\tau} \|n_\varepsilon\|_{L^p(\Omega_T)} \le C_{p,\tau,T},
\label{L:FSL:Est:P1}
\end{align}
for any $1<p<\infty$. Now, using the Neumann heat semigroup, from the equation for $w_\varepsilon$ we can represent this component as 
\begin{align*}
    w_\varepsilon(t)= e^{\frac{1}{\varepsilon}t(\tau \Delta - I)} w_0 + \frac{1}{\varepsilon} \int_0^t e^{\frac{1}{\varepsilon}(t-s)(\tau \Delta - I)} n_\varepsilon (s) ds . 
\end{align*}
Therefore, for any $1\le p_1 \le p_2 \le \infty$ and $k=0,1$, an application of estimate \eqref{GloEx:HS:LpLq} shows
\begin{align*}
\|\nabla^k w_\varepsilon(t)\|_{L^{p_2}(\Omega)} & \le \left\|\nabla^k e^{\frac{1}{\varepsilon}t(\tau \Delta - I)} w_0 \right\|_{L^{p_2}(\Omega)} + \frac{1}{\varepsilon} \left\| \int_0^t \nabla^k e^{\frac{s}{\varepsilon}(\tau \Delta - I)} n_\varepsilon (t-s) ds \right\|_{L^{p_2}(\Omega)} \\ 
& \le C_\tau\| w_0 \|_{W^{k,p_2}(\Omega)} + \frac{C_\tau}{\varepsilon} \int_0^t e^{-\frac{s}{\varepsilon}} \min(s/\varepsilon;1)^{-\frac{N}{2}\big(\frac{1}{p_1}-\frac{1}{p_2}\big)-\frac{k}{2}} \| n_\varepsilon (t-s)\|_{L^{p_1}(\Omega)} ds  \\
& \le C_\tau\| w_0 \|_{W^{k,p_2}(\Omega)} + \frac{C_\tau}{\varepsilon} \| n_\varepsilon\|_{L^\infty((0,T);L^{p_1}(\Omega))} \int_0^t e^{-\frac{s}{\varepsilon}} \min(s/\varepsilon;1)^{-\frac{N}{2}\big(\frac{1}{p_1}-\frac{1}{p_2}\big)-\frac{k}{2}}  ds  
    \end{align*}
using the uniform boundedness of $n_\varepsilon$ in Lemma \ref{L:FSL:Est:n}. By taking $p_1=p_2=\infty$, the latter term is bounded in $L^\infty((0,T))$ since it is obvious that 
\begin{align}
    \frac{1}{\varepsilon}  \int_0^t e^{-\frac{s}{\varepsilon}} \min(s/\varepsilon;1)^{-\frac{k}{2}}  ds \le \int_0^\infty e^{-s} \min(s;1)^{-\frac{k}{2}}  ds \le C,
    \label{L:FSL:Est:P2}
\end{align}
where the constant $C$ does not depend on $\varepsilon$. This shows the uniform boundedness of $w_\varepsilon$ in $L^\infty((0,T);W^{1,\infty}(\Omega))$, which in combination with \eqref{L:FSL:Est:P1} shows \eqref{L:FSL:Est:w:S1}. 

\medskip

For the component $c_\varepsilon$, it follows from its equation that 
\begin{align*}
    c_\varepsilon(t)= e^{\frac{1}{\varepsilon}t(\Delta - I)} c_0 + \frac{1}{\varepsilon} \int_0^t e^{\frac{1}{\varepsilon}(t-s)(\Delta - I)} w_\varepsilon (s) ds , 
\end{align*}
and thus, for any $1\le q_1 \le q_2 \le \infty$, using estimate \eqref{GloEx:HS:LpLq} again gives 
\begin{align*}
    \|\Delta c_\varepsilon(t)\|_{L^{q_2}(\Omega)} &\le C_\tau\| c_0 \|_{W^{2,q_2}(\Omega)} + \frac{C_\tau}{\varepsilon} \int_0^t e^{-\frac{s}{\varepsilon}} \min(s/\varepsilon;1)^{-\frac{N}{2}\big(\frac{1}{q_1}-\frac{1}{q_2}\big)} \|\Delta w_\varepsilon (t-s)\|_{L^{q_1}(\Omega)} ds \\
    &\le C_\tau\| c_0 \|_{W^{2,q_2}(\Omega)} + \frac{C_\tau}{\varepsilon} \|\Delta w_\varepsilon\|_{L^{q_1}(\Omega_T)} \left\| \int_0^t e^{-\frac{s}{\varepsilon}} \min(s/\varepsilon;1)^{-\frac{N}{2}\big(\frac{1}{q_1}-\frac{1}{q_2}\big)}  ds \right\|_{L^{q_1/(q_1-1)}((0,T))}.
\end{align*}
Then, by choosing {$q_1 \gg 1$ and $q_2 = \infty$}, the latter temporal norm is finite, similarly to \eqref{L:FSL:Est:P2}. Hence, $\Delta c_\varepsilon$ is uniformly bounded in $L^\infty(\Omega_T)$, and in the same way, we have the same conclusion for $c_\varepsilon$ and its gradient $\nabla c_\varepsilon$. Consequently, we obtain  \eqref{L:FSL:Est:w:S2}. 
\end{proof}

\begin{lemma} 
\label{L:HolReg}
There exists $\gamma\in (0,1)$ such that 
\begin{align}
    \sup_{\varepsilon>0} \left(  \|n_\varepsilon\|_{C^{\gamma,\gamma/2}(\overline{\Omega}\times[0,T])} \right) \le C_{\tau,T} .
    \label{L:HolReg:S}
\end{align}
    
\end{lemma}

\begin{proof} Recalling for each $\varepsilon>0$, $(n_\varepsilon,c_\varepsilon,w_\varepsilon)$ is the globally classical solution to \eqref{Sys:Eps:Main}-\eqref{Sys:Eps:InitCond}, so that it is continuous with respect to both time and space variables. Therefore, one can apply \cite[Theorem 1.3 and Remark 1.4]{porzio1993holder} or \cite[Lemma 2.1, Part iv]{lankeit2017locally} to claim \eqref{L:HolReg:S}, where $C_{\tau,T}$ does not depend on $\varepsilon$ due to the uniform boundedness of $n_\varepsilon$ in Lemma \ref{L:FSL:Est:n} and of $c_\varepsilon$ in Lemma \ref{L:FSL:Est:w}.  
\end{proof}

\subsubsection{Passage to the limit} 
\label{Sec:PassToLim}

\begin{lemma}
\label{L:FSL:Lim:GloEx} Assume that $(n,c,w)$ is a globally weak solution to System \eqref{Sys:Lim:Main}-\eqref{Sys:Lim:InitCond} in the sense that 
\begin{align}
    n \in C(\overline{\Omega}\times [0,T]) \cap L^\infty(\Omega_T) \cap L^2((0,T);H^1(\Omega)), \quad c,w\in L^2((0,T);H^1(\Omega)),\label{Sys:FSL:Lim:WeakForm0}
\end{align} 
and 
\begin{align}
\begin{gathered}
		- \int_0^T \langle n , \partial_t \xi \rangle - \int_\Omega n_0 \xi(0) =  \intQT (- \nabla n + n \nabla c) \cdot \nabla \xi , \\ 
	\intQT( \nabla c \cdot \nabla \zeta  + c\zeta)   = \intQT  w\zeta, 	\\ 
	\intQT( \tau \nabla w \cdot \nabla \zeta  + w\zeta) = \intQT n\zeta  ,
\end{gathered}
\label{Sys:FSL:Lim:WeakForm}
\end{align}
for all $\xi, \zeta \in C_c^\infty(\overline{\Omega}\times[0,T))$.  
 Then, it is the unique global classical solution to \eqref{Sys:Lim:Main}-\eqref{Sys:Lim:InitCond}.
\end{lemma}

\begin{proof} We first note that it is straightforward to check the initial condition in $L^2(\Omega)$ for $n$. Since the weak formulations for $c$ and $w$ are standard weak forms of the linear elliptic equations (specifically, the last two equations of \eqref{Sys:Lim:Main}), it is obvious that they become the strong solutions to 
\begin{align} 
	\left\{ \begin{array}{lllllll}
\Delta c - c + w = 0 & \text{in } \Omega_\infty, 	\\ 
	\tau \Delta w - w + n = 0 & \text{in } \Omega_\infty, \\
    \partial_\nu c  = \partial_\nu w = 0 & \text{on } \Gamma_\infty, 
	\end{array} \right.
\end{align} 
Using the representation of the inverse operators $(-\Delta + I)^{-1}$ and $(-\tau\Delta + I)^{-1}$, for example, see  \cite[Appendix B]{reisch2024global}, we have 
\begin{align}
    \left\{ \begin{array}{llll}
    c(x,t) = \displaystyle \int_0^\infty e^{s(\Delta - I)} w(x,t) ds, & (x,t)\in \overline{\Omega}\times [0,T], \vspace{0.15cm} \\
    w(x,t) = \displaystyle \int_0^\infty e^{s(\tau\Delta - I)} n(x,t) ds, & (x,t)\in \overline{\Omega}\times [0,T]. 
    \end{array} \right.
\label{L:FSL:Lim:GloEx:P}
\end{align}
Therefore, the continuity of $n$ implies the continuity of $w$ and, then, of $c$. Consequently, the H\"older continuity of $n$ is obtained using the results in \cite[Theorem 1.3 and Remark 1.4]{porzio1993holder} or in \cite[Lemma 2.1, Part iv]{lankeit2017locally}. By the representation \eqref{L:FSL:Lim:GloEx:P} again, we claim the H\"older continuity of $w$ and $c$. This allows us to apply \cite[Lemma 2.1, Part v]{lankeit2017locally} that $n\in C^{2,1}(\Omega \times (0,T))$, and so $(n,c,w)$ becomes the unique classical solution to \eqref{Sys:Lim:Main}-\eqref{Sys:Lim:InitCond}. 
\end{proof}

In the following, we present the proof of Theorem \ref{Theo:FSL:1} for setting subcritical dimensions.

\begin{proof}[\underline{Proof of Theorem \ref{Theo:FSL:1} with subcritical dimensions $N=1,2,3$}]
We first note that boundedness \eqref{Theo:FSL:1:S1} has been obtained in Lemmas \ref{L:FSL:nReg}, \ref{L:FSL:Est:w} and \ref{L:HolReg}. In the following, we will prove the convergence of the sequence $\{(n_\varepsilon,c_\varepsilon,w_\varepsilon)\}_{\varepsilon>0}$ as $\varepsilon \to 0$. 
Thanks to the estimate for $n_\varepsilon$ in the space of H\"older continuous functions obtained in Lemma \ref{L:HolReg}, the Arzelà–Ascoli theorem yields that there exists a subsequence of $\{n_\varepsilon\}_{\varepsilon>0}$ (being denoted by the same notation) such that   
\begin{align}
    n_\varepsilon  \longrightarrow n \quad \text{ strongly in } C(\overline{\Omega}\times [0,T])
    \label{Theo:FSL:1:P1}
\end{align} 
as $\varepsilon \to 0$. 
Moreover, the estimate for this component in Lemma \ref{L:FSL:nReg} also implies that 
\begin{align}
\nabla n_\varepsilon \xrightharpoonup{\hspace{0.4cm}} \nabla n \quad \text{ weakly in } L^2(\Omega_T) .
\label{Theo:FSL:1:P2}
\end{align} 
Testing the equation for $n_\varepsilon$ by $\xi \in C_c^\infty(\overline{\Omega}\times[0,T))$, we derive 
\begin{align*}
    - \int_0^T \langle n_\varepsilon , \partial_t \xi \rangle - \int_\Omega n_0 \xi(0) =  \intQT (- \nabla n_\varepsilon + n_\varepsilon \nabla c_\varepsilon) \cdot \nabla \xi , 
\end{align*}
which, after using the convergence \eqref{Theo:FSL:1:P1}-\eqref{Theo:FSL:1:P2}, shows 
\begin{align*}
    - \int_0^T \langle n , \partial_t \xi \rangle - \int_\Omega n_0 \xi(0) =  \intQT (- \nabla n + n \nabla c) \cdot \nabla \xi . 
\end{align*}

Next, we will consider the limits of $c_\varepsilon$ and $w_\varepsilon$. We note from the previous subsections that the uniform boundedness of $\partial_t c_\varepsilon$ and $\partial_t w_\varepsilon$ is lacking. Therefore, the compactness of $\{c_\varepsilon\}_{\varepsilon>0}$ and $\{w_\varepsilon\}_{\varepsilon>0}$ does not make the Arzelà–Ascoli theorem or the Aubin-Lions lemma applicable. Thanks to Lemma \ref{L:FSL:Est:w}, 
\begin{align}
\begin{array}{lllllll}
(c_\varepsilon, \nabla c_\varepsilon) &\xrightharpoonup{\hspace{0.4cm}}& (c, \nabla c) & \text{weakly in } L^2(\Omega_T)^{N+1}, \\
(w_\varepsilon, \nabla w_\varepsilon) &\xrightharpoonup{\hspace{0.4cm}}& (w, \nabla w) \quad & \text{weakly in } L^2(\Omega_T)^{N+1}.
\end{array}
\label{Theo:FSL:1:P3}
\end{align}
Testing the equation for $w_\varepsilon$ by $\zeta \in C_c^\infty(\overline{\Omega}\times[0,T))$ gives
\begin{align}
    - \varepsilon \int_\Omega  w_\varepsilon(0)\zeta(0) -  \varepsilon  \iint_{\Omega_T}  w_\varepsilon \partial_t \zeta + \intQT( \tau \nabla w_\varepsilon \cdot \nabla \zeta  + w_\varepsilon\zeta) = \intQT n_\varepsilon\zeta .
    \label{Theo:FSL:1:P4}
\end{align}
With the boundedness of $w_\varepsilon$ obtained in  
Lemma \ref{L:FSL:Est:w}, we can pass $\varepsilon \to 0$ to obtain the weak formulation for $w$ in \eqref{Sys:FSL:Lim:WeakForm}. Note that this can be done similarly for the component $c_\varepsilon$. Thus, the limit vector $(n,c,w)$ is a globally weak solution to System \eqref{Sys:Lim:Main}-\eqref{Sys:Lim:InitCond} in the sense \eqref{Sys:FSL:Lim:WeakForm0}-\eqref{Sys:FSL:Lim:WeakForm}. Then, Lemma \ref{L:FSL:Lim:GloEx} yields that this solution becomes the unique globally classical solution of \eqref{Sys:Lim:Main}-\eqref{Sys:Lim:InitCond}. 

\medskip 

We now improve the convergence of $w_\varepsilon,c_\varepsilon$ to a strong sense, which will be basically based on the so-called \textit{energy equation method}, see e.g. \cite{ball2004global,henneke2016fast}, presented as follows.   
Recall that  
\begin{gather}
     \intQT( \nabla w \cdot \nabla \zeta  + w \zeta)   = \intQT  n \zeta,   \quad \text{for all } \zeta \in C_c^\infty(\overline{\Omega}\times[0,T)), 
    \label{Theo:FSL:1:P5}
\end{gather}
and for each $\varepsilon>0$, $w_\varepsilon$ is sufficiently smooth since $(n_\varepsilon,c_\varepsilon,w_\varepsilon)$ is the globally classical solution to System \eqref{Sys:Eps:Main}-\eqref{Sys:Eps:InitCond}. Due to an argument of dense spaces, we can choose $w_\varepsilon$ to be a test function in \eqref{Theo:FSL:1:P4}, which yields 
\begin{gather}
    \intQT( |\nabla w_\varepsilon|^2  + w_\varepsilon^2)   = \intQT  n_\varepsilon w_\varepsilon - \frac{\varepsilon}{2} \intO (w_\varepsilon^2 - w_0^2).  
    \label{Theo:FSL:1:P6}
\end{gather}
Then, choosing $\xi = w$ in \eqref{Theo:FSL:1:P5} gives 
\begin{gather*}
     \intQT( |\nabla w|^2  + w^2)   = \intQT  n w,  
\end{gather*}
which is combined with \eqref{Theo:FSL:1:P6} to show that 
\begin{align*}
    \left| \|w_\varepsilon\|_{L^2((0,T);H^1(\Omega))}^2 - \|w\|_{L^2((0,T);H^1(\Omega))}^2 \right|  
    \le \left| \intQT  (n_\varepsilon w_\varepsilon - n w) \right| + \frac{\varepsilon}{2} \left| \intO (w_\varepsilon^2 - w_0^2)   \right|. 
\end{align*}
Using the convergence \eqref{Theo:FSL:1:P1}, \eqref{Theo:FSL:1:P3}, and the uniform boundedness of $w_\varepsilon$ in $L^\infty((0,T);L^2(\Omega))$, cf. Lemma \ref{L:FSL:Est:w}, the latter right-hand side tends to zero as $\varepsilon\to 0$. Therefore,  
\begin{align*}
\|w_\varepsilon\|_{L^2((0,T);H^1(\Omega))}  \longrightarrow  \|w\|_{L^2((0,T);H^1(\Omega))}. 
\end{align*}
Since $L^2((0,T);H^1(\Omega))$ is uniformly convex, this implies
\begin{align}
w_\varepsilon  \longrightarrow w \quad \text{strongly in } L^2((0,T);H^1(\Omega)). 
\label{Theo:FSL:1:P7}
\end{align}
Similarly, one can show the convergence 
\begin{align}
c_\varepsilon  \longrightarrow c \quad \text{strongly in } L^2((0,T);H^1(\Omega)) ,
\label{Theo:FSL:1:P8}
\end{align}
and, for the same test functions as in \eqref{Theo:FSL:1:P4},  
\begin{gather*}
     \intQT( \nabla c \cdot \nabla \xi  + c \xi)   = \intQT  w \xi. 
\end{gather*}

We obtain the convergence stated at  \eqref{Theo:FSL:1:S2} by collecting \eqref{Theo:FSL:1:P1}-\eqref{Theo:FSL:1:P2} and \eqref{Theo:FSL:1:P7}-\eqref{Theo:FSL:1:P8}. While the first line in \eqref{Theo:FSL:1:S3} is straightforward from estimate  \eqref{L:FSL:Est:c:S2} in Lemma \ref{L:FSL:Est:c} by recalling $\Delta c_\varepsilon - c_\varepsilon + w_\varepsilon = \varepsilon \partial_t c_\varepsilon$, the second one is directly derived from \eqref{Theo:FSL:1:P4} after integrating by parts in space. Since $(n,c,w)$ is the unique solution to System \eqref{Sys:Lim:Main}-\eqref{Sys:Lim:InitCond}, the above convergences hold for the whole sequences.   
\end{proof}

\subsection{The case of critical dimension $N=4$}
{When $N\le 3$, it is sufficient to use the $L^1$-norm of $n_\eps$ and the embedding $H^2(\Omega)\hookrightarrow L^{\infty}(\Omega)$ to control the term $\int_{\Omega}n_\eps c_\eps$. In the critical dimension, we ought to exploit the control of $\int_{\Omega}n_\eps\log n_\eps$ as well as an Adam-type inequality (see Lemmas \ref{L:Ineqn:LlogL} and \ref{L:Ineqn:Adam}) to balance the multiple time-scale entropy}. This {also} leads to a restriction on the size of the initial mass $M$ as \eqref{T:1:4:MassCond}.

\begin{lemma}  \label{L:FSL:Est:c4D} 
Assume that $M$ satisfies \eqref{T:1:4:MassCond}.   
Then,   
\begin{align}
 \sup_{\varepsilon>0} \left(   \sup_{t>0}   \int_{B_R} (n_\varepsilon \log n_\varepsilon + e^{-1}) +  \sup_{t>0}  {\|(\Delta-I)c_\varepsilon(t)\|_{L^2(B_R)}^2}   \right) \le C_{\tau}  ,
\label{L:FSL:Est:c4D:S1}
\end{align} 
and 
\begin{align}
    \begin{aligned}
    \sup_{\varepsilon>0} \left( \frac{1}{\varepsilon} \iint_{B_R\times(0,\infty)} \Big( |\nabla(\Delta c_\varepsilon - c_\varepsilon + w_\varepsilon)|^2 +  |\Delta c_\varepsilon - c_\varepsilon + w_\varepsilon|^2 \Big) \right) \le C_{\tau} .
\end{aligned}
\label{L:FSL:Est:c4D:S2}
\end{align}
\end{lemma}

\begin{proof} For any positive real numbers  $\alpha>0$ and $\eta=\eta(\alpha)>0$ being chosen later, an application of the inequalities \eqref{Ineqn:Pre:Adam} and \eqref{Ineqn:Adam} gives 
\begin{align*}
    \int_{B_R} n_\varepsilon c_\varepsilon & \le \frac{1}{e} +  \frac{1}{\alpha} \int_{B_R} n_\varepsilon \log n_\varepsilon + \frac{\|n_\varepsilon\|_{L^1(\Omega)}}{\alpha} \log \left( \int_{B_R} e^{\alpha c_\varepsilon} \right) \\
    & \le \frac{1}{e} +  \frac{1}{\alpha} \int_{B_R} n_\varepsilon \log n_\varepsilon + \frac{M}{\alpha} \left[ \left( \frac{\alpha^2}{128\pi^2} + \eta \right) \|(\Delta-I)c_\varepsilon\|_{L^2(B_R)}^2 + C_{R,\eta,\alpha} \right]  .
\end{align*}
Then, similarly to estimate \eqref{L:FSL:DissIneqn.P1}, we have    
\begin{align*}
    \begin{aligned} 
& \int_{B_R} \left( (n_\varepsilon \log n_\varepsilon + e^{-1}) + \frac{1}{2}|\Delta c_\varepsilon - c_\varepsilon + w_\varepsilon|^2 + \frac{\tau}{2} |\Delta c_\varepsilon|^2 + \frac{1+\tau}{2} |\nabla c_\varepsilon|^2 + \frac{1}{2} c_\varepsilon^2 \right) \\
& + \int_0^t \hspace{-0.15cm} \int_{B_R} \Big( n_\varepsilon|\nabla (\log n_\varepsilon - c_\varepsilon)|^2 + \frac{1+\tau}{\varepsilon}  |\nabla(\Delta c_\varepsilon - c_\varepsilon + w_\varepsilon)|^2 + \frac{2}{\varepsilon} |\Delta c_\varepsilon - c_\varepsilon + w_\varepsilon|^2 \Big) \\
& \le C  + \frac{1}{e} +  \frac{1}{\alpha} \int_{B_R} n_\varepsilon \log n_\varepsilon + \frac{M}{\alpha} \left[ \left( \frac{\alpha^2}{128\pi^2} + \eta \right) \|(\Delta-I) c_\varepsilon \|_{L^2(B_R)}^2 + C_{R,\eta,\alpha} \right],
\end{aligned} 
\end{align*}
where $C$ is the initial value of the entropy $\mathcal E$. 
 Since $M<64 \tau \pi^2$, we can choose $\alpha>0$ and a sufficiently small number $\eta>0$ such that 
\begin{align*}
\frac{1}{\alpha}< 1 \quad \text{and} \quad
\frac{M}{\alpha} \left( \frac{\alpha^2}{128\pi^2} + \eta \right) < \frac{\tau }{2}, 
\end{align*}
which allows us to imply  that 
\begin{align*}
    & \left( 1- \frac{1}{\alpha} \right) \int_{B_R}  n_\varepsilon \log n_\varepsilon + \left( \frac{\tau }{2} - \frac{M}{\alpha} \left( \frac{\alpha^2}{128\pi^2} + \eta \right) \right) { \|(\Delta-I)c_\varepsilon\|_{L^2(B_R)}^2} \\
    & + \int_0^t \hspace{-0.15cm} \int_{B_R} \Big(  \frac{1+\tau}{\varepsilon}  |\nabla(\Delta c_\varepsilon - c_\varepsilon + w_\varepsilon)|^2 + \frac{2}{\varepsilon} |\Delta c_\varepsilon - c_\varepsilon + w_\varepsilon|^2 \Big)  \le C_{R,\eta,\alpha,M}, 
\end{align*}
for all $0<t<\infty$. The estimates \eqref{L:FSL:Est:c4D:S1}-\eqref{L:FSL:Est:c4D:S2} are consequently obtained. 
\end{proof}

\begin{lemma}
\label{4:L:L43} Assume that $M$ satisfies \eqref{T:1:4:MassCond}. 
Then,   
    \begin{align}
        \sup_{\varepsilon>0} \left( 
        \iint_{B_R\times(0,T)}    \frac{|\nabla n_\varepsilon|^2}{n_\varepsilon} ds + \int_0^T  \|n_\varepsilon\|_{L^{4/3}(B_R)}^2 ds
        \right) \le C_{T,\tau}. 
        \label{4:L:L43:S}
    \end{align}
\end{lemma}

\begin{proof} This proof will be based on balancing a logarithmic energy below. For $x >0$, let us denote $h(x) := x \log x - x + 1$. By direct computations, we have   
\begin{align*}
    \frac{d}{dt} \int_{B_R} h(n_\varepsilon)  = - \int_{B_R} \frac{|\nabla n_\varepsilon|^2}{n_\varepsilon} - \int_{B_R} n_\varepsilon \Delta c_\varepsilon,  
\end{align*} 
which, after integrating over time, gives  
\begin{align}
    \begin{aligned}
    \int_{B_R}  h(n_\varepsilon) ds + \iint_{B_R\times(0,t)}   \frac{|\nabla n_\varepsilon|^2}{n_\varepsilon}  ds  
    \le \int_{B_R}  h(n_0)     - \underbrace{\iint_{B_R\times(0,t)}   n_\varepsilon \Delta c_\varepsilon ds}_{=:\,-I_\varepsilon(t)} .   
\end{aligned}
    \label{4:L:L43:P1}
\end{align} 
In the remaining, we will control the quantity $I_\varepsilon(t)$ using the norm of $n_\varepsilon$ in $L^2((0,T);L^{4/3}(\Omega))$, and then balance the estimate  \eqref{4:L:L43:P1} of the above logarithmic energy. 

\medskip

\noindent \underline{Estimating $I_\varepsilon(t)$}: Using the equations for $c_\varepsilon,w_\varepsilon$, {we have the  following computations}
\begin{align*}
    - \int_{B_R} n_\varepsilon \Delta c_\varepsilon &= \int_{B_R} n_\varepsilon (- \varepsilon \partial_t c_\varepsilon - c_\varepsilon + w_\varepsilon) = - \varepsilon \int_{B_R} n_\varepsilon \partial_t c_\varepsilon - \int_{B_R} n_\varepsilon c_\varepsilon + \int_{B_R} n_\varepsilon w_\varepsilon   \\
    &=  - \varepsilon \int_{B_R} n_\varepsilon \partial_t c_\varepsilon - \int_{B_R} n_\varepsilon c_\varepsilon + \int_{B_R} (\varepsilon \partial_t w_\varepsilon - \tau \Delta w_\varepsilon + w_\varepsilon) w_\varepsilon \\
    &=   - \varepsilon \int_{B_R} n_\varepsilon \partial_t c_\varepsilon -  \int_{B_R} n_\varepsilon c_\varepsilon  + \int_{B_R} \left( \frac{\varepsilon}{2} \partial_t w_\varepsilon^2 + \tau |\nabla w_\varepsilon|^2 + w_\varepsilon^2 \right)  \\
    & \le \varepsilon   \|n_\varepsilon(t)\|_{L^{\frac{4}{3}}(B_R)} \|\partial_tc_\varepsilon\|_{L^{4}(B_R)}  + \int_{B_R} \left( \frac{\varepsilon}{2} \partial_t w_\varepsilon^2 + \tau |\nabla w_\varepsilon|^2 + w_\varepsilon^2 \right) . 
\end{align*}
By the Young inequality, we have     
\begin{align*}
\varepsilon\|n_\varepsilon(t)\|_{L^{4/3}(B_R)} \|\partial_tc_\varepsilon\|_{L^{4}(B_R)}  \le \frac{\varepsilon}{2}  \|n_\varepsilon(t)\|_{L^{4/3}(B_R)}^2 + \frac{\varepsilon}{2}   \|\partial_tc_\varepsilon\|_{L^{4}(B_R)}^2 ,
\end{align*}
and by the Sobolev embedding, 
\begin{align*}
\frac{\varepsilon}{2} \|\partial_tc_\varepsilon\|_{L^{4}(B_R)}^2 \le    C \varepsilon \left(    \|\nabla  \partial_tc_\varepsilon\|_{L^{2}(B_R)}^2 +   \|\partial_tc_\varepsilon\|_{L^{2}(B_R)}^2 \right).
\end{align*}
Consequently, we get
\begin{align*}
    \begin{aligned}
    I_\varepsilon(t)   
    & \le \frac{\varepsilon}{2} \int_0^t    \|n_\varepsilon(s)\|_{L^{4/3}(B_R)}^2 ds   +   C\varepsilon \iint_{B_R\times(0,t)}  \left(  |\nabla  \partial_sc_\varepsilon|^2  +  |\partial_sc_\varepsilon|^2 \right)  ds \\
    & + \frac{\varepsilon}{2} \iint_{B_R\times(0,t)}    \partial_s w_\varepsilon^2   ds    + \iint_{B_R\times(0,t)}  \left(  \tau |\nabla w_\varepsilon|^2 + w_\varepsilon^2 \right) ds \\
    & \le \frac{\varepsilon}{2} \int_0^t    \|n_\varepsilon(s)\|_{L^{4/3}(B_R)}^2 ds   +   C\varepsilon \iint_{B_R\times(0,t)}  \left(  |\nabla  \partial_sc_\varepsilon|^2  +  |\partial_sc_\varepsilon|^2 \right) ds  \\
    & + \frac{\varepsilon}{2}  \int_{B_R}  w_\varepsilon^2   +  \iint_{B_R\times(0,t)}  \left( \tau |\nabla w_\varepsilon|^2 + w_\varepsilon^2 \right) ds .
\end{aligned} 
\end{align*}
Recalling the equation $ \partial_tc_\varepsilon= (1/\varepsilon)(\Delta c_\varepsilon - c_\varepsilon + w_\varepsilon)$,  the dissipation in Lemma \ref{L:FSL:Est:c4D} is rewritten as   
\begin{align*}
  \varepsilon \iint_{B_R\times(0,t)}  \Big(  |\nabla  \partial_sc_\varepsilon|^2 ds   + |\partial_sc_\varepsilon|^2   ds \Big) \le C,
\end{align*}
On the other hand, by applying Lemma \ref{Local:Pre:Lem0} to the equation for $w_\varepsilon$, 
\begin{align*}
    & \frac{\varepsilon}{2}   \int_{B_R} w_\varepsilon^2   +  \iint_{B_R\times(0,t)} \left( |\nabla w_\varepsilon|^2 + w_\varepsilon^2 \right) ds \le   \int_{B_R} w_0^2 + \frac{C}{\tau^2}  \int_0^t  \|n_\varepsilon\|_{L^{\frac{4}{3}}(B_R)}^2 ds  . 
\end{align*} 
Therefore, $I_\varepsilon(t)$ is estimated as  
\begin{align*}
    I_\varepsilon(t)   
    & \le  \left( \frac{\varepsilon}{2} + \frac{C}{\tau^2} \right) \int_0^t \|n_\varepsilon \|_{L^{\frac{4}{3}}(B_R)}^2 ds  + C +    \int_{B_R} w_0^2. 
\end{align*}  

\noindent \underline{Balancing  the logarithmic energy}: {We will apply Lemma \ref{L:Ineqn:Bala} to control the term $\int_{0}^t\|n_\eps\|_{L^{\frac 43}(B_R)}^2ds$}. Due to the computation \eqref{4:L:L43:P1} and the estimate for $I_\varepsilon(t)$, 
\begin{align}
    \int_{B_R} h(n_\varepsilon) ds + \iint_{B_R\times(0,t)}   \frac{|\nabla n_\varepsilon|^2}{n_\varepsilon} ds \le \left( \frac{3}{2} + \frac{C}{\tau^2} \right) \int_0^t \|n_\varepsilon\|_{L^{\frac{4}{3}}(B_R)}^2 ds  +   C_{\tau},
    \label{4:L:L43:P2}
\end{align}
where $C_\tau$ includes the value of the logarithmic entropy at the initial time and the last two terms in the estimate for $I_\varepsilon(t)$. 
By Lemma \ref{L:FSL:Est:c4D}, we have $n_\varepsilon \in L^\infty((0,T);L\log L(B_R))$. Therefore, an application of  Lemma \ref{L:Ineqn:Bala} gives
\begin{equation}\label{b1}    
\begin{aligned}
    \|n_\varepsilon(t)\|_{L^{\frac{4}{3}}(B_R)}^2 & 
\le \alpha    
\left(  \int_{B_R} (n_\varepsilon(t) \log n_\varepsilon(t) +  e^{-1}) \right) \int_{B_R} \frac{|\nabla n_\varepsilon|^2}{n_\varepsilon} +   C_\alpha \\  
& 
\le \alpha    
\left( \sup_{t>0} \int_{B_R} (n_\varepsilon(t) \log n_\varepsilon(t) +  e^{-1}) \right) \int_{B_R} \frac{|\nabla n_\varepsilon|^2}{n_\varepsilon} +   C_\alpha \\
&  \le \frac{1}{2} \left( \frac{3}{2} + \frac{C}{\tau^2} \right)^{-1} \int_{B_R} \frac{|\nabla n_\varepsilon|^2}{n_\varepsilon} + C_\tau ,
\end{aligned}
\end{equation}  
where we take a constant $\alpha$ such that 
\begin{align*}
   \alpha    
\left( \sup_{t>0} \int_{B_R} (n_\varepsilon(t) \log n_\varepsilon(t) +  e^{-1}) \right) \le \frac{1}{2} \left( \frac{3}{2} + \frac{C}{\tau^2} \right)^{-1}. 
\end{align*}
Hence, we can absorb the term including $\|n_\varepsilon\|_{L^{\frac{4}{3}}(B_R)}^2$ in the estimate \eqref{4:L:L43:P2} into the left-hand side, which consequently implies   
\begin{align*}
    \int_{B_R}  h(n_\varepsilon(t))  ds + \iint_{B_R\times(0,t)}  \frac{|\nabla n_\varepsilon|^2}{n_\varepsilon} ds  \le    C_{\tau} t.
\end{align*}
This directly shows {the first estimate in} \eqref{4:L:L43:S}. {The second one follows immediately by integrating \eqref{b1} with respect to $t$ and using the first estimate.}
\end{proof}

\begin{lemma} 
\label{L:4:Lp}
    For $1<p<\infty$, it holds that 
\begin{align}
\sup_{\varepsilon>0}   \left( \sup_{0<t<T} \int_{B_R}  n_\varepsilon^p(t)  +  \iint_{B_R\times(0,T)}     |\nabla n_\varepsilon|^2 \right)   \le    C_{p,T}. 
\label{L:4:Lp:S}
\end{align}
\end{lemma}

\begin{proof} Using the equation for $n_\varepsilon$, one can check that 
\begin{align}
\begin{aligned}
 \frac{d}{dt}  \int_{B_R}
   n_\varepsilon^p(t) &= -\frac{4(p-1)}{p}   \int_{B_R}
   | \nabla n_\varepsilon^{p/2}(t)|^2 + p(p-1) \int_{B_R} n_\varepsilon^{p-1}(t) \nabla n_\varepsilon(t) \cdot \nabla c_\varepsilon(t) \\
 &\le - \frac{4(p-1)}{p}   \int_{B_R}
   | \nabla n_\varepsilon^{p/2}(t)|^2 + \frac{p(p-1)}{2} \int_{B_R} n_\varepsilon^{p}(t) |\nabla c_\varepsilon(t)|^2 . 
\end{aligned}
    \label{L:4:Lp:P1}
\end{align}
To estimate this energy, we will control the latter term by the product of the integral of $n_\varepsilon^p$ and a suitable norm of $\nabla c_\varepsilon$. Indeed, using the H\"older, the Gagliardo-Nirenberg and the Young inequalities, it can be dealt with as 
\begin{align}
\begin{aligned}
    &\int_{B_R} (n_\varepsilon^{p/2}(t))^2 |\nabla c_\varepsilon(t)|^2   \le \|n_\varepsilon^{p/2}(t)\|_{L^{\frac{8}{3}}(B_R)}^2 \|\nabla c_\varepsilon(t)\|_{L^{8}(B_R)}^2 \\
    & \le \bigg(C  \|n_\varepsilon^{p/2}(t)\|_{L^{2}(B_R)}^{\frac{1}{2}} \|\nabla n_\varepsilon^{p/2}(t)\|_{L^{2}(B_R)}^{\frac{1}{2}} + C  \|n_\varepsilon^{p/2}(t)\|_{L^{2}(B_R)} \bigg)^2 \|\nabla c_\varepsilon(t)\|_{L^{8}(B_R)}^2 \\
    & \le \bigg(C  \|n_\varepsilon^{p/2}(t)\|_{L^{2}(B_R)}  \|\nabla n_\varepsilon^{p/2}(t)\|_{L^{2}(B_R)}  + C  \|n_\varepsilon^{p/2}(t)\|_{L^{2}(B_R)}^2 \bigg) \|\nabla c_\varepsilon(t)\|_{L^{8}(B_R)}^2 \\
    & \le \frac{4}{p^2} \int_{B_R}
   | \nabla n_\varepsilon^{p/2}(t)|^2 + C \left( \frac{p^2}{16}     \|\nabla c_\varepsilon(t)\|_{L^{8}(B_R)}^4 +  \|\nabla c_\varepsilon(t)\|_{L^{8}(B_R)}^2 \right) \int_{B_R} n^p_\varepsilon(t).
   \end{aligned}
    \label{L:4:Lp:P2}
\end{align} 
Thus, we deduce from  \eqref{L:4:Lp:P1} that 
\begin{align}
 \frac{d}{dt}  \int_{B_R}
   n_\varepsilon^p(t)  
 &\le - \frac{2(p-1)}{p}   \int_{B_R}
   | \nabla n_\varepsilon^{p/2}(t)|^2 + C_p \left( 1+ \|\nabla c_\varepsilon(t)\|_{L^{8}(B_R)}^4 \right) \int_{B_R} n^p_\varepsilon(t)  .
   \label{L:4:Lp:P3}
\end{align}
{It remains} to estimate $\nabla c_\varepsilon$ in $L^4((0,T);L^8(B_R))$, which will be done using Lemmas \ref{L:FSL:Est:c4D}-\ref{4:L:L43} and \ref{Local:Pre:Lem0}. Indeed, thanks to the uniform boundedness of $n_\varepsilon$ in $L^2((0,T);L^{4/3}(B_R))$ obtained in Lemma \ref{4:L:L43}, we can apply Lemma \ref{Local:Pre:Lem0} to have 
\begin{align*}
\iint_{B_R\times(0,T)} |\nabla w_\varepsilon|^2    \le   \int_{B_R} u_0^2 +  \frac{C}{\tau ^2}  \int_0^{T}  \|n_\varepsilon\|_{L^{\frac{4}{3}}(B_R)}^2 \le C_{\tau,T}.  
\end{align*}
On the other hand, by Lemma \ref{L:FSL:Est:c4D},  
\begin{align*}
\int_0^T \hspace{-0.15cm} \int_{B_R} |\nabla(\Delta c_\varepsilon - c_\varepsilon + w_\varepsilon)|^2   \le  C_{\tau} \varepsilon .
\end{align*}
Therefore, it follows from the uniform boundedness of $c_\varepsilon$ in $L^\infty((0,T);H^2(B_R))$ the triangle estimate 
\begin{align*}
     \int_0^T \hspace{-0.15cm} \int_{B_R} |\nabla \Delta c_\varepsilon|^2 \le C 
 \int_0^T \hspace{-0.15cm} \int_{B_R} \Big( |\nabla(\Delta c_\varepsilon - c_\varepsilon + w_\varepsilon)|^2 + |\nabla c_\varepsilon|^2 + |\nabla w_\varepsilon|^2 \Big) \le C_{\tau,T}. 
\end{align*}
This yields that  $c_\varepsilon$ is uniformly bounded in $L^2((0,T);H^3(B_R))$. Using the Gagliardo–Nirenberg inequality and the Sobolev embedding,  
\begin{align*}
    \|\nabla c_\varepsilon(t)\|_{L^{8}(B_R)}^4 \le C \|c_\varepsilon(t)\|_{H^{3}(B_R)}^2 \|\nabla c_\varepsilon(t)\|_{L^{4}(B_R)}^2 \le C \|c_\varepsilon(t)\|_{H^{3}(B_R)}^2 \|  c_\varepsilon(t)\|_{H^{2}(B_R)}^2.  
\end{align*}
Subsequently, using the boundedness of $c_\varepsilon$ in $L^\infty((0,\infty);H^2(\Omega))$ again, we get 
\begin{align}
    \int_0^T \|\nabla c_\varepsilon(t)\|_{L^{8}(B_R)}^4 \le C  \int_0^T \|c_\varepsilon(t)\|_{H^{3}(B_R)}^2  \le C_{\tau,T}.   
    \label{L:4:Lp:P4}
\end{align}
Finally, by the boundedness \eqref{L:4:Lp:P4}, an application of the Gr\"onwall inequality to  \eqref{L:4:Lp:P3} shows
\begin{align*}
    \int_{B_R}  n_\varepsilon^p(t) \le \left( \int_{B_R}  n_0^p \right) \exp\left(\sup_{\varepsilon>0} \left( \int_0^T \|\nabla c_\varepsilon(t)\|_{L^{8}(B_R)}^4 \right) \right) \le C_{\tau,T},  
\end{align*}
for all $0<t<T$. {The gradient estimate in 
\eqref{L:4:Lp:S} is obtained by choosing $p=2$}. 
\end{proof}

\begin{lemma} 
\label{L:4:LInf}
It holds that 
    \begin{align} 
\sup_{\varepsilon>0} \Big( \|n_\varepsilon\|_{L^\infty(B_R\times(0,T))} \Big) \le C_{\tau,T},
\label{L:4:LInf:S1}
\end{align}
and, for any $1<p<\infty$, 
\begin{align}
\sup_{\varepsilon>0} \Big(
 \|w_\varepsilon \|_{L^\infty((0,T);W^{1,\infty}(B_R))} + \|\Delta w_\varepsilon \|_{L^p(B_R\times(0,T))} + 
 \|c_\varepsilon \|_{L^\infty((0,T);W^{2,\infty}(B_R))}  \Big) \le C_{\tau,T}.  
\label{L:4:LInf:S2}
\end{align}
Consequently, there exists $\gamma \in (0,1)$ such that  
\begin{align}
\sup_{\varepsilon>0} \left(  \|n_\varepsilon\|_{C^{\gamma,\gamma/2}(\overline{B_R}\times[0,T])} \right) \le C_{\tau,T} . 
\label{L:4:LInf:S3}
\end{align}
\end{lemma}

\begin{proof} Using the boundedness of $n_\varepsilon$ in $L^\infty((0,T);L^p(B_R))$ for any $1\le p <\infty$, we {can similar arguments to Lemma \ref{L:FSL:Est:n}, with a suitable H\"older inequality to account for different regularities of $n_\eps$ and $c_\eps$ in this case, and the estimate \eqref{L:FSL:Est:n:P1}} to prove the estimate \eqref{L:4:LInf:S1}.  Then, by repeating the techniques of Lemma \ref{L:FSL:Est:w} with the maximal regularity and the smoothing effect of the Neumann heat semigroup, we obtain \eqref{L:4:LInf:S2}, which allows us to derive \eqref{L:4:LInf:S3} similarly to Lemma \ref{L:HolReg}. 
\end{proof}
 
We are ready to prove the remaining case of
Theorem \ref{Theo:FSL:1}. 

\begin{proof}
    [\underline{Proof of Theorem \ref{Theo:FSL:1} in critical dimension $N=4$}] Based on the uniform regularity in Lemma \ref{L:4:LInf}, we can {repeat all the steps and arguments} in the proof for the subcritical case in Subsection \ref{Sec:PassToLim}.  
\end{proof}

\section{Convergence rates and the initial layer's effect {for PES}}\label{sec:PES_rate}

In this section, we investigate the accuracy of the {parabolic-elliptic} simplification presented in Theorem \ref{Theo:FSL:1}, with the main result stated in Theorem \ref{Theo:FSL:3}. We begin with estimates for the initial layers in Lemma \ref{Lem:Layer} and obtain needed regularity for the limiting solution in Lemma \ref{est:dtw,dtc}. Then, we present energy estimates for the rate system \eqref{Sys:Rate:Convergence} in Lemmas \ref{est:engrate}-\ref{Sys:Rate:Energies}, which help us prove Theorem \ref{Theo:FSL:3}. {Recall that we consider all dimensions $1\le N\le 4$ in this section.}

\begin{lemma}
\label{Lem:Layer} %Let $\epsc$ and $\mathrm{dist}^{l,p}[w_0;\mathcal C_{\mathsf{PES}}]$ be defined as Theorem \ref{Theo:FSL:3}. Then,  
{There exists a constant $C>0$ such that}
    \begin{equation*}
    \begin{aligned}
{\|\wc(0)\|_{W^{k+1,p}(\Omega)} 
	 \le} \,&\,  {C \|-\Delta c_0 + c_0 -w_0\|_{W^{k,p}(\Omega)},} \\ 
    {\|\ww(0)\|_{W^{l+1,p}(\Omega)} 
	\le} \,&\, {C \|-\tau \Delta w_0 + w_0 - n_0\|_{W^{l,p}(\Omega)}}. 
    \end{aligned}
    \end{equation*}
\end{lemma}

\begin{proof} The values $c(0)$ and $w(0)$ will be calculated from the equations for $c$ and $w$ in System \eqref{Sys:Lim:Main}-\eqref{Sys:Lim:InitCond}, using the representations of the inverse operators $(-\Delta + I)^{-1}$ and $(-\tau\Delta + I)^{-1}$, similarly to the proof of Lemma  \ref{L:FSL:Lim:GloEx}. 
Indeed, it follows from
\begin{align}
	\left\{
	\begin{array}{lll}
		c(x,t) \; = 
		\displaystyle \int_0^{\infty} e^{s(\Delta - I)}w(x,t) ds, \vspace{0.15cm} \\
		w(x,t)  = \displaystyle \int_0^{\infty} e^{s(\tau \Delta - I)}n(x,t) ds,
	\end{array}\right.
    \label{Repre:cw}
\end{align}
that 
\begin{align}\label{Ini_Cond_RateSys_c}
	\wc(x,0) 
	%= \,&\,  c_0(x) - \int_0^{\infty} e^{s(\Delta - I)}w_0(x) ds \nonumber\\
	= \,&\,  - \int_0^{\infty} e^{s(\Delta - I)}(\Delta - I)c_0(x) ds - \int_0^{\infty} e^{s(\Delta - I)}w_0(x) ds \nonumber \\
	= \,&\, \int_0^{\infty} e^{s(\Delta - I)}[-\Delta c_0(x) + c_0(x) -w_0(x)] ds , 
\end{align}
and 
\begin{align}\label{Ini_Cond_RateSys_w}
	\ww(x,0) = \int_0^{\infty} e^{s(\tau \Delta - I)}[-\tau \Delta w_0(x) + w_0(x) - n_0(x)] ds. 
\end{align} 
Then, $L^p-L^q$ estimates for the Neumann heat semigroup in Subsection \ref{Sec:HeatSemi} shows    
\begin{align*} 
& \|\wc(0)\|_{W^{k+1,p}(\Omega)} 
	 \le C \left( \int_0^{\infty} e^{-s} s^{-\frac{1}{2}} ds \right) \|-\Delta c_0 + c_0 -w_0\|_{W^{k,p}(\Omega)}, \\
&     \|\ww(0)\|_{W^{l+1,p}(\Omega)} 
	\le C \left( \int_0^{\infty} e^{-s} s^{-\frac{1}{2}} ds \right) \|-\tau \Delta w_0 + w_0 - n_0\|_{W^{l,p}(\Omega)},
\end{align*}
for any $k,l\in \mathbb N$ and $2\le p\le \infty$. 
\end{proof}

\begin{lemma} \label{est:dtw,dtc} Let $(n,w,c)$ be the solution of System \eqref{Sys:Lim:Main}-\eqref{Sys:Lim:InitCond} as obtained in Theorem \ref{Theo:FSL:1}. Then
	\begin{align*}
		\| \partial_{t}w \|_{L^\infty((0,T);L^2(\Omega))} + \| \partial_{t}w \|_{L^p(\Omega_T)} + \| \partial_{t}c \|_{L^\infty((0,T);H^1(\Omega))} + \| \partial_{t}c \|_{L^p(\Omega_T)}  \leq C_{T}, 
	\end{align*}
    for any $1<p<\infty$. 
\end{lemma}
\begin{proof} Differentiating with respect to time from the representation \eqref{Repre:cw} and using the equation for $n$ from System \eqref{Sys:Lim:Main}-\eqref{Sys:Lim:InitCond},  we get
\begin{align*}
    \partial_{t}w(t)  = \displaystyle \int_0^{\infty} e^{s(\tau\Delta - I)} \big[  \Delta n(t) -  \nabla n(t) \nabla c(t) - n(t) \Delta c(t) \big] ds.
\end{align*}
Moreover, we note from Theorem \ref{Theo:FSL:1} and the standard regularisation that the solution $(n,w,c)$ can be directly regularised to be sufficiently smooth, which allows us applying the $L^p-L^q$ estimate for the heat Neumann semigroup, cf. \eqref{GloEx:HS:LpLq} to obtain the desired estimate for $\partial_t w$. The term $\partial_t c$ is treated similarly using again the representation \eqref{Repre:cw}. On the other hand, the boundedness of $\partial_{t} c $ and $ \partial_{t} w $ in $L^p(\O_T)$ can be obtained directly via the maximal regularity. 
\end{proof}

%\subsection{Starting far away from the equilibrium}

% Since the rate vector $(\wn,\wc,\ww)$ has no sign, to consider the   rate in $L^\infty((0,T);L^p(\Omega))$, for any $1<p<\infty$, it is suitable to consider its components to the power $2k$ for $k\in \mathbb{N}$. More precisely, let us consider the energy  
% \begin{align} \label{energyrate}
% 	\mathcal{H}_{2k}[\wn](t):= \intO \wn ^{2k}(t), \quad   t>0, \, k \geq 1.
% \end{align}
For $1 \le k\in \mathbb N$, based on the uniform regularity of the $\varepsilon$-dependent solution given in Theorem \ref{Theo:FSL:1}, we will obtain an a priori estimate for the $L^{2k}$-energy of $\wn$ due to direct computations from the rate system. 

\begin{lemma} \label{est:engrate}
	For each $k \geq 1$, there exists a constant $C_{k,T}>0$ such that
	\begin{align*}
		\dfrac{d}{dt} \intO \wn ^{2k}(t) \leq -\dfrac{2k-1}{k}  \intO |\nabla \wn^{k}|^2 
		+ C_{k, T}\intO \wn ^{2k} + C_{k, T}\intO |\nabla \wc|^2,
	\end{align*}
for all $0<t<T$.	
\end{lemma}

\begin{proof}
It is obvious from  the equation for $\wn$, we have 
	\begin{align*}
		\dfrac{d}{dt} \intO \wn ^{2k} & =2k  \intO  \wn^{2k-1} \partial_{t} \wn
		= 2k  \intO  \wn^{2k-1} \left[ \Delta \wn -  \nabla \cdot (\wn \nabla c_\varepsilon + n\nabla \wc) \right ]	\\
		& = -\dfrac{2(2k-1)}{k}  \intO |\nabla \wn^{k}|^2 + {2(2k-1) \intO \wn^{k} \nabla \wn^{k} \cdot \nabla c_\varepsilon} \\
		& \hspace{0.45cm} + {2(2k-1) \intO n \wn^{k-1} \nabla \wn^{k} \cdot  \nabla \wc}  \\
        & \leq  -\dfrac{(2k-1)}{k} \int_\Omega |\nabla \wn^{k}|^2 + C_{k,T} \intO \wn ^{2k} +  C_{k,T} \int_\Omega |\nabla \wc|^{2} ,
	\end{align*}
{where we used the Young inequality and $\|n\|_{L^{\infty}(\Omega_T)} + \|\wn\|_{L^\infty{(\Omega_T)}} \le C_T$ at the last step}.  
\end{proof}

The following lemma is {obtained straightforwardly} by testing the
equations for $\wc$, $\ww$ by  $\wc$, $\ww$, and by $ \Delta^2 \wc$, $-\Delta \ww$, respectively, {then using integration by parts as well as Young's inequality. Therefore, its proof is omitted}.
\begin{lemma} 
\label{Sys:Rate:Energies}
There hold that  
\begin{align} 
&         \varepsilon\frac{d}{dt} \int_\Omega \wc^{\,2} + 2\int_\Omega |\nabla \wc|^2 + \int_\Omega \wc^{\,2}  \leq  2\int_\Omega \ww^2 + 2\varepsilon^2 \int_\Omega |\partial_t c|^2, \label{Theo:RC:a:P0} \\
&         \varepsilon \frac{d}{dt} \int_\Omega \ww^2 + 2 \tau \int_\Omega |\nabla \ww|^2 +   \int_\Omega \ww^2  \leq  2 \int_\Omega \wn^2 + 2 \varepsilon^2 \int_\Omega |\partial_t w|^2, \label{Theo:RC:a:P00}
\end{align}	 
and  
\begin{gather}
    \varepsilon\frac{d}{dt} \int_\Omega |\Delta\wc|^{2} +  \int_\Omega |\nabla\Delta\wc|^{2}  + 2\int_\Omega |\Delta \wc|^2 \leq   2\int_\Omega |\nabla\ww|^2 + 2\varepsilon^2 \int_\Omega |\nabla \partial_t c|^2,  
    \label{Theo:RC:a:P1}\\  
        \varepsilon \frac{d}{dt} \int_\Omega |\nabla\ww|^2 + \tau \int_\Omega |\Delta\ww|^2 +  2 \int_\Omega |\nabla \ww|^2 \leq \frac{2}{\tau} \int_\Omega (\wn)^2 + \frac{2\varepsilon^2}{\tau} \int_\Omega |\partial_t w|^2.  
        \label{Theo:RC:a:P2}
\end{gather}
\end{lemma}

We now prove Theorem \ref{Theo:FSL:3}. 

\begin{proof}[Proof of Theorem \ref{Theo:FSL:3}]	a) This part is proved by exploiting Lemma \ref{est:engrate} together with Lemmas \ref{Lem:Layer}-\ref{est:dtw,dtc}. Indeed,
applying Lemma \ref{est:engrate} with $k=1$, we get
\begin{align}
     \dfrac{d}{dt} \intO \wn^2 +  \intO |\nabla \wn|^2  \leq 
		 C_{ T}\intO \wn^2 + C_{ T}\intO |\nabla \wc|^2,
\label{Theo:RC:a:P0b}
\end{align}
{where we note that} the constant $C_T$ does not depend on $\tau$.  
A linear combination of the estimates in \eqref{Theo:RC:a:P0b} and \eqref{Theo:RC:a:P0}, \eqref{Theo:RC:a:P00} yields  
	\begin{align}
    \label{Theo:RC:a:P2b}
    \begin{aligned}
& \frac{d}{dt} \int_\Omega \left[\wn^2(t) + \varepsilon \left( \frac{C_T}{2}  \wc^{\,2}(t) +  C_T \ww^2(t) \right) \right]   +  \int_{\Omega}  |\nabla \wn|^2   \\
		&   \leq  3 C_T \int_{\Omega} \wn^2 +  C_T \varepsilon^2 \int_{\Omega}   |\partial_t c|^2  + 2C_T \varepsilon^2 \int_{\Omega} |\partial_t w|^2 , 
    \end{aligned}
	\end{align}
in which the constant $C_T$ is kept similarly to the first one. Taking into account the boundedness of $\partial_t c, \partial_t w$ given in Lemma \ref{est:dtw,dtc}, the last two terms on the right-hand side are bounded by $C_T \varepsilon^2$. Applying the Gr\"onwall inequality, we obtain for $t\in [0,T]$ that 
\begin{align*}
    \int_\Omega \left[\wn^2(t) + \varepsilon \left( \frac{C_T}{2}  \wc^{\,2}(t) +  C_T \ww^2(t) \right) \right] \le C_T \left[ \varepsilon^2 + \varepsilon \int_\Omega   \left( \frac{C_T}{2}  \wc^{\,2}(0) +  C_T \ww^2(0) \right) \right],    
\end{align*}
where we note from the initial condition \eqref{Sys:Rate:InitialCond} that $\wn(0)=0$. Thanks to Lemma \ref{Lem:Layer}, 
\begin{align*}
    \int_\Omega   \left( \frac{C_T}{2}  \wc^{\,2}(0) +  C_T \ww^2(0) \right) \leq \,&\, C \left( \|-\Delta c_0 + c_0 -w_0\|_{L^2(\Omega)}^2 + \|-\tau \Delta w_0 + w_0 - n_0\|_{L^2(\Omega)}^2 \right) \\
    %= \,&\, C\Big( \big( \mathrm{dist}^{0,2}[c_0;\mathcal C_{\mathsf{PES}}] \big)^2 +   \big( \mathrm{dist}^{0,2}[w_0;\mathcal C_{\mathsf{PES}}] \big)^2 \Big) \\
    = \,&\, {C \big( \mathrm{dist} [(n_0,c_0,w_0);\mathcal C_{\mathsf{PES}}] \big)^2}.
\end{align*}
Therefore, the rate $\wn$, considered in $L^\infty((0,T);L^2(\Omega))$, is of the order $O(\varepsilon + \sqrt{\varepsilon}  \mathrm{dist} [(n_0,c_0,w_0);\mathcal C_{\mathsf{PES}}] )$, which consequently shows {the first part of} \eqref{Theo:RC:Parta:1}. {The second part follows from integrating \eqref{Theo:RC:a:P2b} over the time interval $(0,t)$ and using the first part}.
For estimating the rate component $\ww$, using  
the boundedness of $\partial_t w$ in $L^\infty((0,T);L^2(\Omega))$ {in Lemma} \ref{est:dtw,dtc}, and the rate estimate for $\wn$ as \eqref{Theo:RC:Parta:1}, {we have from \eqref{Theo:RC:a:P2} that}
%\begin{align*}
%    \frac{2}{\tau} \int_\Omega (\wn)^2 + \frac{2\varepsilon^2}{\tau} \int_\Omega |\partial_t w|^2 \le \frac{C_T}{\tau}  \left( \varepsilon^2 + \varepsilon \big( \mathrm{dist} [(n_0,c_0,w_0);\mathcal C_{\mathsf{PES}}] \big)^2 \right), 
%\end{align*}
%which together with  \eqref{Theo:RC:a:P2} implies 
\begin{align*}
     \varepsilon\frac{d}{dt} \int_\Omega |\nabla \ww|^2 + \tau  \int_\Omega |\Delta \ww|^2 + 2 \int_\Omega |\nabla \ww|^2  &{\le \frac{2}{\tau} \int_\Omega (\wn)^2 + \frac{2\varepsilon^2}{\tau} \int_\Omega |\partial_t w|^2}\\
     &{\le \frac{C_T}{\tau}  \left( \varepsilon^2 + \varepsilon \big( \mathrm{dist} [(n_0,c_0,w_0);\mathcal C_{\mathsf{PES}}] \big)^2 \right)},  
\end{align*}
Hence, {by Lemma \ref{aDI}},
 we  obtain 
\begin{align*}
    \int_\Omega |\nabla\ww(t)|^2 & \le  e^{-\frac{2}{\varepsilon}t} \int_\Omega |\nabla\ww(0)|^2 + C_{\tau,T}  \left( \varepsilon  +  \big( \mathrm{dist}[(n_0,c_0,w_0);\mathcal C_{\mathsf{PES}}] \big)^2   \right) \int_0^t e^{-\frac{2}{\varepsilon}s} ds \\
    &\le e^{-\frac{2}{\varepsilon}t} \|\nabla\ww(0)\|_{L^2(\Omega)}^2 + C_{\tau,T}\left( \varepsilon^2 + \varepsilon \big( \mathrm{dist}[(n_0,c_0,w_0);\mathcal C_{\mathsf{PES}}] \big)^2 \right) \\
    & \le C {\|-\tau \Delta w_0 + w_0 - n_0\|_{H^1(\Omega)}^2} + C_{\tau,T}\left( \varepsilon^2 + \varepsilon \big( \mathrm{dist}[(n_0,c_0,w_0);\mathcal C_{\mathsf{PES}}] \big)^2 \right) , 
\end{align*}
where we note that the distances on the right-hand side are less than or equal to {$\mathrm{dist}^{0,1}_2[(n_0,c_0,w_0);\mathcal C_{\mathsf{PES}}]$}. We derive the estimate \eqref{Theo:RC:Parta:2}, where the zeroth order term  $\int_\Omega |\ww(t)|^2$ is estimated in the same way.   
{The proof of \eqref{Theo:RC:Parta:3} follows similarly, so we omit it here}.  

\medskip

\noindent b) {It is sufficient} to prove this part for $p=2k$ ($k\ge 1$). Thanks to the rate estimate for $\wn$ in Part a, the estimate \eqref{Theo:RC:a:P00}, and Lemma \ref{est:dtw,dtc}, we have 
\begin{align*}
    \iint_{\Omega_t} \ww^2 & \le \varepsilon \int_{\Omega} \ww^2(0) + C \left( \iint_{\Omega_t} \wn^2 + \varepsilon^2 \iint_{\Omega_t} |\partial_s w|^2 \right) \\
    & \le \varepsilon {\|-\tau \Delta w_0 + w_0 - n_0\|_{L^2(\Omega)}^2} + C_T   \left( \varepsilon^2 + \varepsilon \big(
     \mathrm{dist}  [(n_0,c_0,w_0);\mathcal C_{\mathsf{PES}}] \big)^2 \right) \\
     & \le C_T \left( \varepsilon^2 + \varepsilon \big(
     \mathrm{dist}  [(n_0,c_0,w_0);\mathcal C_{\mathsf{PES}}] \big)^2 \right).
\end{align*}
Then, {by integrating \eqref{Theo:RC:a:P0} on $(0,t)$}, we get 
\begin{align*}
    \iint_{\Omega_t} |\nabla\wc|^2 
     & \le C_T \left( \varepsilon^2 + \varepsilon \big(
     \mathrm{dist}  [(n_0,c_0,w_0);\mathcal C_{\mathsf{PES}}] \big)^2 \right).
\end{align*}
Now, it follows from Lemma \ref{est:engrate} that
     \begin{align*}      %\label{dt_engrate}
		 \intO \wn ^{p}(t) & \leq \intO \wn ^{p}(0) -\dfrac{2p-2}{p}  \iint_{\Omega_t} |\nabla \wn^{p/2}|^2 
		+ C_{p, T} \iint_{\Omega_t} \wn ^{p}   + C_{p, T} \iint_{\Omega_t} |\nabla \wc|^2 \\
        & \leq \intO \wn ^{p}(0)  + C_{p, T} \iint_{\Omega_t} \wn ^{p} + C_{p, T,\tau} \left( \varepsilon^2 + \varepsilon \big(
     \mathrm{dist}  [(n_0,c_0,w_0);\mathcal C_{\mathsf{PES}}] \big)^2 \right) .
    \end{align*}
The Gr\"onwall inequality directly shows   
\begin{align} 
\|\wn\|_{L^{\infty}((0,T); L^{p}(\Omega))} \leq C_{p,T,\tau}  \left( \varepsilon^{\frac{2}{p}}  + \varepsilon^{\frac{1}{p}} \big(
\mathrm{dist}  [(n_0,c_0,w_0);\mathcal C_{\mathsf{PES}}]\big)^{\frac{2}{p}} \right)  .
\label{Theo:RC:b:P1}
\end{align}
Thanks to the boundedness of $\partial_{t} c $, $ \partial_{t} w $ in $L^q(\Omega_T)$ in Lemma \ref{est:dtw,dtc} and the estimate \eqref{Theo:RC:b:P1}, we   apply the maximal regularity with slow evolution (cf. Lemma \ref{L:MR:SlowEvolution}) to the equation for $\ww$ that 
\begin{align*}
	\| \ww \|_{L^p((0,T);W^{2,p}(\Omega))} & \leq C_{p,\tau} \left( \varepsilon^{\frac{1}{p}} \| \Delta \ww(0) \|_{L^p(\Omega)} +  \| \wn - \varepsilon \partial_{t} w \|_{L^p(\Omega_T)} \right) \\
	& \leq  C_{p,\tau,T}  \left(
    \varepsilon^{\frac{1}{p}} {\|-\tau \Delta w_0 + w_0 - n_0\|_{W^{2,p}(\Omega)}}  + \varepsilon^{\frac{2}{p}}  + \varepsilon^{\frac{1}{p}} \big(
\mathrm{dist}  [(n_0,c_0,w_0);\mathcal C_{\mathsf{PES}}]\big)^{\frac{2}{p}}  \right) \\
	& \le C_{p,\tau,T}   \left( \varepsilon^{\frac{2}{p}} +
 \varepsilon^{\frac{1}{p}} \big(
\mathrm{dist}^{0,2}_p  [(n_0,c_0,w_0);\mathcal C_{\mathsf{PES}}]\big)^{\frac{2}{p}} \right).
\end{align*}
Similarly, we have the following estimates
\begin{align*}
	\| \wc \|_{L^p((0,T);W^{4,p}(\Omega))} & \leq C_{p} \left( \varepsilon^{\frac{1}{p}} \| \Delta^2 \wc(0) \|_{L^p(\Omega)} +   \| \Delta \ww - \varepsilon \partial_{t} \Delta c \|_{L^p(\Omega_T)} \right)\\
	& \leq C_{p,\tau,T}   \left( \varepsilon^{\frac{2}{p}} +
 \varepsilon^{\frac{1}{p}} \big(
\mathrm{dist}^{4,2}_p  [(n_0,c_0,w_0);\mathcal C_{\mathsf{PES}}]\big)^{\frac{2}{p}} \right),
\end{align*}
which completes the proof. 
\end{proof}

%\subsection{The initial layer and starting near the equilibrium}

\section{From indirect signalling to direct signalling}\label{sec:IDS}

We rigorously study the indirect-direct simplification from \eqref{Sys:Eps:Main}-\eqref{Sys:Eps:InitCond} to \eqref{Sys:LimKappa:Main} in this section. The main result of this part was stated in Theorem \ref{Theo:I2D}, including both passing to the limit and estimating the convergence rates. 

\subsection{Balancing the multiple time scale Lyapunov functional}

\begin{lemma}
\label{L:I2D:Balan} It holds that 
\begin{align}
\begin{gathered}
\sup_{t>0}\int_\Omega \left(   n_\kappa \log n_\kappa   +   \frac{1}{4}|\Delta c_\kappa  - c_\kappa + w_\kappa |^2 +  \frac{1}{4} |\nabla c_\kappa |^2 + \frac{1}{2} c_\kappa^2  \right) \\
 + \frac{1}{\varepsilon} \iint_{\Omega_\infty} \Big(    |\nabla(\Delta c_\kappa - c_\kappa + w_\kappa)|^2 +   |\Delta c_\kappa - c_\kappa + w_\kappa|^2 \Big)  
 \le  C .
\end{gathered}
\label{L:I2D:Balan:S}
\end{align} 
\end{lemma}

\begin{proof} Similarly to the prooof of Lemma \ref{L:FSL:Est:c4D}, we will balance the dissipation inequality in Lemma \ref{L:FSL:Lya:Id}. If $N=1$ then the Sobolev embedding $H^1(\Omega) \hookrightarrow L^\infty(\Omega)$ can be utilised to see that 
\begin{align*}
    \int_{\Omega} n_\kappa c_\kappa \le  \|c_\kappa\|_{L^\infty(\Omega)} \int_{\Omega} n_\kappa \le C M \|c_\kappa\|_{H^1(\Omega)} \le \frac{1}{4} \|c_\kappa\|_{H^1(\Omega)}^2 + C^2M^2. 
\end{align*}
By skipping the term including $\tau$ on the left-hand side of   \eqref{L:FSL:DissIneqn.P1}, for all $t>0$ we get 
\begin{align*}
\begin{gathered}
\int_\Omega \left(   n_\kappa \log n_\kappa    +   \frac{1}{2}|\Delta c_\kappa - c_\kappa + w_\kappa|^2 +  \frac{1}{2} |\nabla c_\kappa|^2 + \frac{1}{2} c_\kappa^2 \right) \\
 + \frac{1}{\varepsilon} \iint_{\Omega_t} \Big(    |\nabla(\Delta c_\kappa - c_\kappa + w_\kappa)|^2 +   |\Delta c_\kappa - c_\kappa + w_\kappa|^2 \Big) \\
 \le \mathcal E(n_0,c_0)  + \int_\Omega n_\kappa c_\kappa \le \mathcal E(n_0,c_0)  + \frac{1}{4} \|c_\kappa\|_{H^1(\Omega)}^2 +  C . 
\end{gathered} 
\end{align*}
The estimate \eqref{L:I2D:Balan:S} is  showed by absorbing the term including $\|c_\kappa\|_{H^1(\Omega)}$ to the left-hand side. 

\medskip

Let us consider $N=2$ by exploiting the Moser-Trudinger inequality (instead of the Adam type inequality), which is represented in Part a of Lemma \ref{L:Ineqn:Adam}. Indeed, for any positive real number $\alpha>0$ to be chosen later, a combination of the inequalities \eqref{Ineqn:Pre:Adam} and \eqref{Ineqn:MT} gives 
\begin{align*}
    \int_{\Omega} n_\kappa c_\kappa & \le \frac{1}{e} +  \frac{1}{\alpha} \int_\Omega n_\kappa \log n_\kappa + \frac{\|n_\kappa\|_{L^1(\Omega)}}{\alpha} \log \left( \intO e^{\,\alpha\, c_\kappa} \right) \\
    & \le \frac{1}{e} +  \frac{1}{\alpha} \intO n_\kappa \log n_\kappa + \frac{M}{\alpha} \left[ \frac{\alpha^2}{8\pi} \|\nabla c_\kappa\|_{L^2(\Omega)}^2 + \frac{\alpha}{|\Omega|} \int_\Omega c_\kappa + C_{\alpha} \right] \\
    & =    \frac{1}{\alpha} \intO n_\kappa \log n_\kappa + \frac{\alpha M}{8\pi} \intO |\nabla c_\kappa|^2 + C_{\alpha,M,\Omega}. 
\end{align*}
Consequently, it follows from \eqref{L:FSL:DissIneqn.P1} that  
\begin{align*}
    \begin{aligned} 
& \intO \left( n_\kappa \log n_\kappa + \frac{1}{2}|\Delta c_\kappa - c_\kappa + w_\kappa|^2 +   \frac{1}{2} |\nabla c_\kappa|^2 + \frac{1}{2} c_\kappa^2 \right) \\
& + \frac{1}{\varepsilon} \iint_{\Omega_t} \Big(    |\nabla(\Delta c_\kappa - c_\kappa + w_\kappa)|^2 +   |\Delta c_\kappa - c_\kappa + w_\kappa|^2 \Big) \\
& \le \mathcal E(n_0,c_0) + \frac{1}{\alpha} \intO n_\kappa \log n_\kappa + \frac{\alpha M}{8\pi} \intO |\nabla c_\kappa|^2 + C .  
\end{aligned} 
\end{align*}
Since $M<4\pi$, there always exists $\alpha>1$ such that $\alpha M/(8\pi) < 1/2$, which means that the integrals on the latter right-hand side can be controlled by terms on the left-hand side.
\end{proof}

\begin{lemma}
\label{L:PL:Est:nL2} It holds that   
\begin{align*}
\sup_{\kappa\,\in (0,\infty)^2} \Big( \|n_\kappa \|_{L^2(\Omega_T)} + \|c_\kappa  \|_{L^2((0,T);H^2(\Omega))} \Big) \le C.
\end{align*} 
\end{lemma}

\begin{proof} By considering the Boltzmann entropy for the first equation of \eqref{Sys:Eps:Main}, we have 
\begin{align*}
  \int_\Omega(n_\kappa  \log n_\kappa  +e^{-1})  
&  = \int_\Omega(n_{0} \log n_{0}   +e^{-1})   -  {\iint_{\Omega_t}} \frac{|\nabla n_\kappa |^2}{n_\kappa }  - {\iint_{\Omega_t}} n_\kappa  \Delta c_\kappa  \\
& \le C  -  \intQT \frac{|\nabla n_\kappa |^2}{n_\kappa } + \frac{1}{2} \intQT  n_\kappa  ^2+ \frac{1}{2} \intQT  |\Delta c_\kappa |^2 , 
\end{align*} 
which shows that
\begin{align*}
\int_\Omega(n_\kappa  \log n_\kappa  +e^{-1}) + \intQT \frac{|\nabla n_\kappa |^2}{n_\kappa } \le C  + \frac{1}{2} \intQT  n_\kappa  ^2+ \frac{1}{2} \intQT  |\Delta c_\kappa |^2. 
\end{align*} 
We will balance the two sides of the above estimate. Thanks to the parabolic maximal regularity with slow evolution (see Lemma  \ref{L:MR:SlowEvolution}) applied to the second equation of System \eqref{Sys:Eps:Main}, 
\begin{align}
\begin{aligned}
    \|\Delta c_\kappa  \|_{L^2(\Omega_T)}  & \le  C \left(   \varepsilon^{\frac{1}{2}}  \|\Delta c_0\|_{L^2(\Omega)} +  \|w_\kappa\|_{L^2(\Omega_T)} \right) \\
& \le  C \left(   \varepsilon^{\frac{1}{2}}  \|\Delta c_0\|_{L^2(\Omega)} +  {\left(\varepsilon \intO w_0^2 + \intQT  n_\kappa^2\right)^{1/2}} \right), 
\end{aligned}
\label{L:PL:Est:nL2:P0}
\end{align}
where the $L^2(\Omega_T)$-norm of $w_\kappa$ is controlled by testing the equation for $w_\kappa$ by $w_\kappa$, given as  
\begin{align*}
    \frac{\varepsilon}{2} \intO w_\kappa^2 + \tau \intQT |\nabla w_\kappa|^2 +  \frac{1}{2}\intQT  w_\kappa^2 \le \frac{\varepsilon}{2} \intO w_0^2 + \frac{1}{2} \intQT  n_\kappa^2. 
\end{align*}
Therefore, we obtain 
\begin{align}
\int_\Omega(n_\kappa  \log n_\kappa  +e^{-1}) + \intQT \frac{|\nabla n_\kappa |^2}{n_\kappa } \le C  + C \intQT  n_\kappa  ^2 .
\label{L:PL:Est:nL2:P1}
\end{align} 
Due to Lemma \ref{L:I2D:Balan}, $n_\kappa$ is uniformly-in-$\kappa$ bounded in $L^\infty((0,\infty);L\log L(\Omega))$, which suits to apply Lemma \ref{L:Ineqn:Bala} with $N\le 2$ to have that 
\begin{align}
    \intQT  n_\kappa^2 & \le  \alpha     
\left( \sup_{t>0}  \int_{\Omega} (n_\kappa \log n_\kappa +  e^{-1}) \right) \intQT \frac{|\nabla n_\kappa |^2}{n_\kappa } +   C_\alpha T , 
\label{L:PL:Est:nL2:P2}
\end{align} 
{for any $\alpha >0$}. Consequently,
\begin{align*}
    \intQT  n_\kappa^2 & \le C \alpha \intQT \frac{|\nabla n_\kappa |^2}{n_\kappa } +   C_\alpha T.
\end{align*} 
This estimate yields that the $L^2(\Omega_T)$-norm of $n_\kappa$ in \eqref{L:PL:Est:nL2:P1} can be controlled by the second term on the left-hand side with a sufficiently {small $\alpha>0$}. Hence, we obtain the uniform-in-$\kappa$ boundedness of $|\nabla n_\kappa |^2/n_\kappa$ in $L^1(\Omega_T)$, which in a combination with \eqref{L:PL:Est:nL2:P2} concludes that $n_\kappa$ is uniformly-in-$\kappa$ bounded in $L^2(\Omega_T)$. Then, back  to \eqref{L:PL:Est:nL2:P0}, we obtain a uniform bound for $c_\kappa$ in $L^2((0,T);H^2(\Omega))$.  
\end{proof}

\begin{lemma}
\label{L:PL:Est:nL2+} There is a positive constant $\delta>0$ such that    
\begin{align}
\sup_{\kappa\in(0,\infty)^2} \left( \esssup_{t\in(0,T)} \intO n_\kappa^{1+\frac{\delta}{2}}(t) + \|n_\kappa \|_{L^{2+\delta}(\Omega_T)} + \iint_{\Omega_T} n_\kappa ^{\frac{\delta}{2}-1} |\nabla n_\kappa |^2 \right) \le C_{T}.
\label{L:PL:Est:nL2+:S1}
\end{align}
\end{lemma}

\begin{proof} Considering the $L^p$-energy functional (with $p>1$) for the first component  of \eqref{Sys:Eps:Main},  
\begin{align}
\begin{aligned}
\frac{1}{p} \int_\Omega n_\kappa ^p(t) &= -(p-1) \intQt n_\kappa ^{p-2} |\nabla n_\kappa |^2 + \frac{1}{p} \int_\Omega n_{0}^p  + (p-1) \intQt n_\kappa ^{p-1}\nabla n_\kappa  \cdot \nabla c_\kappa  \\
&= -(p-1) \intQt n_\kappa ^{p-2} |\nabla n_\kappa |^2  + \frac{1}{p} \int_\Omega n_{0}^p  - \frac{p-1}{p} \intQt n_\kappa ^{p} \Delta c_\kappa  \\
&\le -(p-1) \intQt n_\kappa ^{p-2} |\nabla n_\kappa |^2 + \frac{1}{p} \int_\Omega n_{0}^p  + \underbrace{\frac{p-1}{p} \|n_\kappa \|_{L^{2p}(\Omega_t)}^p \|\Delta c_\kappa \|_{L^{2}(\Omega_t)}}_{:=J_\kappa(t)} ,
\end{aligned}
\label{L:PL:Est:nL2+:P1} 
\end{align}
in which, we note from Lemma \ref{L:PL:Est:nL2} that  $\Delta c_\kappa $ is uniformly-in-$\kappa$ bounded in $L^2(\Omega_T)$. To balance the $L^p$-energy functional, we will estimate the temporal supremum of  the integral on the left-hand side, or more precisely,  
 $  \esssup_{t\in(0,T)} \int_\Omega n_\kappa^p(t) $.    
Using the Gagliardo-Nirenberg inequality of the form
\begin{align}
\|f\|_{L^4(\Omega)}^4 \le C \|\nabla f\|_{L^2(\Omega)}^2  \|f\|_{L^2(\Omega)}^2 + C \|f\|_{L^2(\Omega)}^4 
\label{Ineqn:I2D:GN} 
\end{align}
for $f=n_\kappa^{p/2}(s)$ (here, $s\in(0,t)$), we get   
\begin{align*}
\int_\Omega n_\kappa^{2p}(s)  & \le Cp^2    \int_\Omega n_\kappa^{p-2}(s) |\nabla n_\kappa (s)|^2   \int_\Omega n_\kappa^p(s) +  C\int_\Omega  n_\kappa^p(s) .
\end{align*} 
{Thus, we can estimate} 
\begin{align} 
\label{L:PL:Est:nL2+:P2}
\begin{aligned}
    \left(\intQt n_\kappa ^{2p}\right)^{\frac{1}{2}}   
& \le \left( Cp^2 \, \esssup_{s\in(0,T)} \int_\Omega n_\kappa^p(s) \intQt n_\kappa ^{p-2} |\nabla n_\kappa |^2 + C t \esssup_{s\in(0,T)} \int_\Omega n_\kappa^p(s) \right)^{\frac{1}{2}} \\
& = Cp  \left( \esssup_{s\in(0,T)} \int_\Omega n_\kappa^p(s) \right)^{\frac{1}{2}} \left(    \intQt n_\kappa ^{p-2} |\nabla n_\kappa |^2 + \frac{C}{p^2} t \right)^{\frac{1}{2}} .
\end{aligned}
\end{align} 
Then, the last term of the energy estimate \eqref{L:PL:Est:nL2+:P1} can be treated as   
\begin{align*}
J_\kappa(t) & \le C(p-1) \|\Delta c_\kappa \|_{L^{2}(\Omega_T)} \left( \esssup_{s\in(0,T)} \int_\Omega n_\kappa^p(s) \right)^{\frac{1}{2}} \left( \intQt n_\kappa ^{p-2} |\nabla n_\kappa |^2 + \frac{CT}{p^2}  \right)^{\frac{1}{2}} \\
& \le C_T(p-1) \left( \esssup_{s\in(0,T)} \int_\Omega n_\kappa^p(s) \right)^{\frac{1}{2}} \left( \intQt n_\kappa ^{p-2} |\nabla n_\kappa |^2 + \frac{CT}{p^2}  \right)^{\frac{1}{2}}  \\
& \le \frac{p-1}{2} \left( \intQt n_\kappa ^{p-2} |\nabla n_\kappa |^2 + \frac{CT}{p^2}  \right) + C_{T}(p-1) \left( \esssup_{s\in(0,T)} \int_\Omega n_\kappa^p(s) \right) , 
\end{align*}
using the Young inequality. Combining this with the energy estimate \eqref{L:PL:Est:nL2+:P1} (and replacing the variable $s$ by $t$ in the above supremum),   we derive 
\begin{align}
 & \frac{1}{p} \int_\Omega n_\kappa ^p(t) + \frac{p-1}{2}  \intQT n_\kappa ^{p-2} |\nabla n_\kappa |^2
 \le C_{T,\,p}   +  (p-1)C_{T} \left( \esssup_{t\in(0,T)} \int_\Omega n_\kappa^p(t) \right),  \label{L:PL:Est:nL2+:P3}
\end{align}
where it is useful to note that the constant $C_T$ does not depend on $p$. Subsequently, by  
skipping the gradient term and then taking the supremum over time $t\in(0,T)$,
\begin{align*}
  \frac{1}{p} \left( \esssup_{t\in(0,T)} \int_\Omega n_\kappa^p(t) \right) 
 \le C_{T,\,p}   +  C_{T}(p-1) \left( \esssup_{t\in(0,T)} \int_\Omega n_\kappa^p(t) \right).  
\end{align*}
Choosing $p=1+\delta/2$ in which $\delta>0$ is sufficiently small such that {$C_{T}p(p-1)<1$}, we obtain the uniform boundedness  
\begin{align*}
\sup_{\kappa \in (0,\infty)^2} \left( \esssup_{t\in(0,T)} \int_\Omega n_\kappa^p(t) \right)
 \le  \frac{p C_{T,\,p}}{1-C_{T}p(p-1)} , 
\end{align*}
i.e., $n_\kappa$ is uniformly-in-$\kappa$ bounded in $L^\infty((0,T);L^{1+\delta/2}(\Omega))$. 
Then, we obtain the uniform-in-$\kappa$ boundedness of $n_\kappa ^{p-2} |\nabla n_\kappa |^2$ in $L^1(\Omega_T)$ by returning to \eqref{L:PL:Est:nL2+:P3}, and so is $n_\kappa$ in $L^{2p}(\Omega_T) \equiv L^{2+\delta}(\Omega_T)$ due to \eqref{L:PL:Est:nL2+:P2}. The desired estimate \eqref{L:PL:Est:nL2+:S1} is proved.
\end{proof}

\begin{lemma} 
\label{L:I2D:Feed}
Let $\delta$ be the constant obtained in Lemma \ref{L:PL:Est:nL2+}, and define the sequence $\{p_j\}_{j=1,2,\dots}$  by $$p_0>1 \quad \text{and} \quad p_{j+1}:=p_j+\delta/2 \text{ for } j=0,1,\dots$$  If {for some $j\ge 0$}
\begin{align}
    \sup_{\kappa \in (0,\infty)^2} \left( \|n_\kappa \|_{L^{2p_j}(\Omega_T)} \right) \le C_T, 
    \label{L:I2D:Feed:S1}
\end{align}
then 
\begin{align}
    \sup_{\kappa \in (0,\infty)^2} \left( \esssup_{t\in(0,T)} \int_\Omega n_\kappa ^{p_{j+1}}(t) +  \|n_\kappa \|_{L^{2p_{j+1}}(\Omega_T)} \right) \le C_{T,p_{j+1}}.
    \label{L:I2D:Feed:S2}
\end{align}
\end{lemma} 

\begin{proof} The main idea of this proof is to balance the $L^{p_{j+1}}$-energy estimate from the $L^{2p_j}(\Omega_T)$-regularity given in the assumption \eqref{L:I2D:Feed:S1}. Taking $p=p_{j+1}$ in the energy estimate \eqref{L:PL:Est:nL2+:P1}, we have 
\begin{align}
     \begin{aligned}
         &\frac{1}{p_{j+1}} \int_\Omega n_\kappa ^{p_{j+1}}(t) \\
     &= -(p_{j+1}-1) \intQt n_\kappa ^{p_{j+1}-2} |\nabla n_\kappa |^2  + \frac{1}{p_{j+1}} \int_\Omega n_{0}^{p_{j+1}}  - \frac{p_{j+1}-1}{p_{j+1}} \intQt n_\kappa ^{p_{j+1}} \Delta c_\kappa .
     \end{aligned} 
     \label{L:I2D:Feed:P1}
\end{align}
To balance this energy, we also recall from a similar application of the Gagliardo-Nirenberg inequality \eqref{Ineqn:I2D:GN} as the proof of Lemma \ref{L:PL:Est:nL2+} that    
\begin{align}
\begin{aligned}
    \|n_\kappa\|_{L^{2p_{j+1}}(\Omega_T)}^{p_{j+1}} &\le Cp_{j+1} \left( \esssup_{t\in(0,T)} \int_\Omega  n_\kappa^{p_{j+1}}(t) \right)^{\frac{1}{2}}  \left( \intQT n_\kappa^{p_{j+1}-2}(s) |\nabla n_\kappa (s)|^2   +  CT \right)^{\frac{1}{2}} \\
&\le C_{p_{j+1}} \left( \esssup_{t\in(0,T)} \int_\Omega  n_\kappa^{p_{j+1}}(t) + \intQT n_\kappa^{p_{j+1}-2}(s) |\nabla n_\kappa (s)|^2   +  CT \right) .  
\end{aligned}
\label{L:I2D:Feed:P2}
\end{align}
First, let us show that $\Delta c_\kappa$ is uniformly-in-$\kappa$ bounded in $L^{2p_j}(\Omega_T)$, based on the assumption \eqref{L:I2D:Feed:S1}. 
Indeed, the parabolic maximal regularity with slow evolution (see Lemma  \ref{L:MR:SlowEvolution}) yields  
\begin{align}
\begin{aligned}
    \|\Delta c_\kappa \|_{L^{2p_j}(\Omega_T)}  &\le  \left( \frac{\varepsilon}{2p_j} \right)^{\frac{1}{2p_j}} \|\Delta c_0\|_{L^{2p_j}(\Omega)} +  C_{p_j}  \|w_\kappa \|_{L^{2p_j}(\Omega_T)} \\
 &\le  C_{p_j} \left( \|\Delta c_0\|_{L^{2p_j}(\Omega)} +  \|w_\kappa \|_{L^{2p_j}(\Omega_T)} \right). 
\end{aligned}
\label{L:PL:Feedback:P1}
\end{align}
Here, by using the third equation of \eqref{Sys:Eps:Main},   
\begin{align*}
\|w_\kappa \|_{L^{2p_j}(\Omega_T)}^{2p_j}  = \frac{\varepsilon}{2p_j} \int_\Omega (w_0^{2p_j}-w_\kappa ^{2p_j}) - \tau(2p_j-1) \intQT w_\kappa ^{2p_j-2} |\nabla w_\kappa |^2 + \intQT n_\kappa  w_\kappa ^{2p_j-1}.
\end{align*}
Skipping the negative terms on the right-hand side, and applying the Young inequality as follows $$ n_\kappa  w_\kappa ^{2p_j-1}\le \frac{1}{2p_j} n_\kappa ^{2p_j} + \frac{2p_j-1}{2p_j} w_\kappa ^{2p_j}, $$
we obtain the estimate     
\begin{align*}
\|w_\kappa \|_{L^{2p_j}(\Omega_T)}^{2p_j} \le  \varepsilon  \int_\Omega w_0^{2p_j} +  \|n_\kappa \|_{L^{2p_j}(\Omega_T)}^{2p_j} \le \int_\Omega w_0^{2p_j} +  \|n_\kappa \|_{L^{2p_j}(\Omega_T)}^{2p_j} .
\end{align*}
Therefore, we imply from \eqref{L:PL:Feedback:P1} that 
\begin{align*}
\|\Delta c_\kappa \|_{L^{2p_j}(\Omega_T)}  
&\le  C_{p_j} \left(  \|\Delta c_0\|_{L^{2p_j}(\Omega)} +  \left(  \int_\Omega w_0^{2p_j} +  \|n_\kappa \|_{L^{2p_j}(\Omega_T)}^{2p_j} \right)^{\frac{1}{2p_j}} \right) \\
&\le  C_{p_j} \left( \|\Delta c_0\|_{L^{2p_j}(\Omega)} +  \|w_0 \|_{L^{2p_j}(\Omega)} +  \|n_\kappa \|_{L^{2p_j}(\Omega_T)} \right) , 
\end{align*}
i.e., the uniform-in-$\kappa$ boundedness of $\Delta c_\kappa$ in $L^{2p_j}(\Omega_T)$ has been showed. 

\medskip

Now, we can estimate the term including $n_\kappa ^{p_{j+1}} \Delta c_\kappa$ in the $L^{p_{j+1}}$-energy computation as   
\begin{align*}
- \intQt n_\kappa ^{p_{j+1}} \Delta c_\kappa  \le \|n_\kappa\|_{L^{\frac{2p_{j}p_{j+1}}{2p_{j}-1}}(\Omega_T)}^{p_{j+1}} \|\Delta c_\kappa \|_{L^{2p_j}(\Omega_T)} \le C_{T,p_j} \|n_\kappa\|_{L^{\frac{2p_{j}p_{j+1}}{2p_{j}-1}}(\Omega_T)}^{p_{j+1}} .  
\end{align*}
By interpolation in Lebesgue spaces, 
\begin{align*}
    \|n_\kappa\|_{L^{\frac{2p_{j}p_{j+1}}{2p_{j}-1}}(\Omega_T)}^{p_{j+1}} & \le \left( \|n_\kappa\|_{L^{2p_{j+1}}(\Omega_T)}^{\lambda_{j+1}} \|n_\kappa\|_{L^{1}(\Omega_T)}^{1-\lambda_{j+1}} \right)^{p_{j+1}} \le C_{p_{j+1}}\|n_\kappa\|_{L^{2p_{j+1}}(\Omega_T)}^{\lambda_{j+1}p_{j+1}}  , 
\end{align*}
where, by direct computation,  
$$\lambda_{j+1}=\frac{2p_{j}p_{j+1}-2p_{j}+1}{2p_{j}p_{j+1}-p_{j}} \in (0,1).$$
Since $\lambda_{j+1}p_{j+1}<p_{j+1}$, for any constant $\eta>0$ the Young inequality ensures 
\begin{align*}
    - \intQt n_\kappa ^{p_{j+1}} \Delta c_\kappa & \le C_{T,\,p_{j+1},\eta} + \eta \|n_\kappa\|_{L^{2p_{j+1}}(\Omega_T)}^{p_{j+1}} \\
    & \le C_{T,\,p_{j+1},\eta} + \eta C_{p_{j+1}} \left( \esssup_{t\in(0,T)} \int_\Omega  n_\kappa^{p_{j+1}}(t) + \intQT n_\kappa^{p_{j+1}-2}(s) |\nabla n_\kappa (s)|^2   +  CT \right), 
\end{align*}
where we have used \eqref{L:I2D:Feed:P2} at the second estimate. 
This combines with \eqref{L:I2D:Feed:P1} that 
\begin{align*}
& \frac{1}{p_{j+1}} \left( \esssup_{t\in(0,T)} \int_\Omega n_\kappa ^{p_{j+1}}(t) \right) + 
(p_{j+1}-1) \intQT n_\kappa ^{p_{j+1}-2} |\nabla n_\kappa |^2 \\
& \le C_{T,\,p_{j+1},\eta} + \eta C_{p_{j+1}} \left( \esssup_{t\in(0,T)} \int_\Omega  n_\kappa^{p_{j+1}}(t) + \intQT n_\kappa^{p_{j+1}-2}(s) |\nabla n_\kappa (s)|^2   +  CT \right).
\end{align*}
By choosing $\eta$ sufficiently small, we can absorb the integrals on the right-hand side into the left one, which accordingly gives  
\begin{align*}
      \esssup_{t\in(0,T)} \int_\Omega n_\kappa ^{p_{j+1}}(t)  + 
 \intQT n_\kappa ^{p_{j+1}-2} |\nabla n_\kappa |^2  \le C_{T,\,p_{j+1}}.
\end{align*}
With this boundedness, we finally obtain \eqref{L:I2D:Feed:S2} using \eqref{L:I2D:Feed:P2}. 
\end{proof}

\begin{lemma} 
\label{L:I2D:LInf}
It holds that 
    \begin{align} 
\sup_{\kappa\in(0,\infty)^2} \Big( \|n_\kappa\|_{L^\infty(\Omega_T)} + \|n_\kappa\|_{L^2((0,T);H^1(\Omega))} \Big) \le C_{T},
\label{L:I2D:LInf:S1}
\end{align}
and, for any $1<p<\infty$, 
\begin{gather}
\sup_{\kappa\in(0,\infty)^2} \Big(
 \|w_\kappa \|_{L^\infty(\Omega_T)} + \|w_\kappa \|_{L^2((0,T);H^{1}(\Omega))} \Big) \le C_{T},   
\label{L:I2D:LInf:S2} \\
\sup_{\kappa\in(0,\infty)^2} \Big(
 \|c_\kappa \|_{L^\infty((0,T);W^{1,\infty}(\Omega))} + \|c_\kappa \|_{L^p((0,T);W^{2,p}(\Omega))} \Big) \le C_{T}.  
\label{L:I2D:LInf:S3} 
\end{gather}
Consequently, there exists $\gamma \in (0,1)$ such that  
\begin{align}
\sup_{\kappa\in(0,\infty)^2} \left(  \|n_\kappa\|_{C^{\gamma,\gamma/2}(\overline{\Omega}\times[0,T])} \right) \le C_{T} . 
\label{L:I2D:LInf:S4}
\end{align}
\end{lemma}

\begin{proof} %By Lemma \ref{L:PL:Est:nL2+}, the boundedness \eqref{L:I2D:Feed:S1} holds for $p_0:=1+\delta/2$. Therefore, by the feedback argument given in 
	{From Lemma} \ref{L:I2D:Feed}, we obtain for any $1\le p<\infty$ that   
\begin{align}
    \sup_{\kappa \in (0,\infty)^2} \left( \esssup_{t\in(0,T)} \int_\Omega n_\kappa ^{p}(t) +  \|n_\kappa \|_{L^{p}(\Omega_T)} + \intQT |\nabla n_\kappa|^2  \right) \le C_{T,p},
    \label{L:I2D:LInf:P1}
\end{align}
where we note that the limit of \eqref{L:I2D:Feed:S2} as $j\to \infty$ has not been claimed because of the  $p_{j+1}$-dependence   (i.e., the $L^\infty(\Omega_T)$-boundedness is not a consequence of \eqref{L:I2D:Feed:S2}). This implies \eqref{L:I2D:LInf:S1} similarly to Lemma \ref{L:FSL:Est:n}, {noting again that we exploit the boundedness of $n_\kappa$ in $L^\infty(0,T;L^p(\Omega))$ for any $p\ge 1$}, the estimate \eqref{L:FSL:Est:n:P1}. 

\medskip

Now,  it follows from the equation for $w_\varepsilon$ that 
\begin{align*}
    \intQT w_\kappa^p & = -  \frac{\varepsilon}{p}	 \intO w_\kappa^p - \tau (p-1) \intQT w_\kappa^{p-2} |\nabla w_\kappa|^2 + \frac{\varepsilon}{p}	 \intO w_0^p + \intQT n_\kappa  w_\kappa^{p-1} ,
\end{align*}
for $p>1$. Then, by the Young inequality, we get 
\begin{align*}
    \intQT w_\kappa^p \le \frac{\varepsilon}{p}	 \intO w_0^p + \frac{1}{p} \intQT n_\kappa^p + \frac{p-1}{p} \intQT w_\kappa^p,    
\end{align*}
which consequently deduces that   $$ \lim_{p\to \infty} \|w_\kappa\|_{L^p(\Omega_T)}  \le \lim_{p\to \infty} \left( \varepsilon \|w_0\|_{L^p(\Omega)}^p +  \|n_\kappa\|_{L^p(\Omega_T)}^p \right)^{1/p} \le   C \left( \|w_0\|_{L^\infty(\Omega)}  
+  \|n_\kappa\|_{L^\infty(\Omega_T)} \right) , $$  
i.e., $w_\kappa$ is uniformly-in-$\kappa$ bounded in $L^\infty(\Omega_T)$. Based on the boundedness of $\nabla n_\kappa$ in $L^2(\Omega_T)$, we can similarly test the equation for $w_\kappa$ by $-\Delta w_\kappa$ to obtain the same boundedness of $\nabla w_\kappa$ in $L^2(\Omega_T)$, and so is $w_\kappa$ in $L^2((0,T);H^1(\Omega))$ as \eqref{L:I2D:LInf:S2}.  

\medskip

For the component $c_\kappa$, a uniform-in-$\kappa$ bound in $L^\infty((0,T);H^1(\Omega))$ was obtained in the Lemma \ref{L:I2D:Balan}. The first term in the estimate \eqref{L:I2D:LInf:S3} is proved similarly to Lemma \ref{L:FSL:Est:w}, while the second one is directly a consequence of the maximal regularity with slow evolution given in Lemma \ref{L:MR:SlowEvolution}. Finally, one can show \eqref{L:I2D:LInf:S4} similarly to Lemma \ref{L:HolReg}. 
\end{proof}

\subsection{Passage to the limit and convergence rate analysis}

\begin{proof}[Proof of Theorem \ref{Theo:I2D}]   Based on the uniform-in-$\kappa$ regularity obtained in Lemma \ref{L:I2D:LInf}, one can adapt all steps from the proof of Theorem \ref{Theo:FSL:1} to prove the passage to the limit given in this theorem. For the convergence rate estimates, the estimate \eqref{Theo:RC:a:P2b} {still holds}, i.e., 
\begin{gather*}
		  \frac{d}{dt} \int_\Omega \left[(n_\kappa(t)-n(t))^2 + \varepsilon \left( \frac{C_T}{2} (c_\kappa(t)-c(t))^{\,2} +  C_T (w_\kappa(t)-w(t))^2  \right) \right]   +  \int_{\Omega}  |\nabla (n_\kappa(t)-n(t))|^2   \\
		  \leq  3 C_T \int_{\Omega} (n_\kappa(t)-n(t))^2 +  C_T \varepsilon^2 \int_{\Omega}   |\partial_t c|^2  + 2C_T \varepsilon^2 \int_{\Omega} |\partial_t w|^2 , 
	\end{gather*}
since $\nabla c_\kappa $ is uniformly-in-$\kappa$ bounded in $L^\infty(\Omega_T)^N$ as Lemma \ref{L:I2D:LInf}, which consequently shows \eqref{Theo:FSL:3:1}. On the other hand, by skipping the term including $\tau$ at the estimate  \eqref{Theo:RC:a:P00}, and then using the comparison principle for differential equations in Lemma \ref{aDI}, we get \eqref{Theo:FSL:3:2}. The estimate \eqref{Theo:FSL:3:3} is obtained similarly {to the rest of the proof of Theorem \ref{Theo:FSL:3}}. 
\end{proof}

%\subsection{Strong convergence to the critical manifold} 

{It remains to prove} Corollary \ref{Coro:2}.

% \begin{lemma}
% \label{Stro:ConverCri} It holds that 
%     \begin{align}
%         \|n_\kappa - w_\kappa\|_{L^2(\Omega_T)}^2 \le C_T |\kappa| 
%         \label{Stro:ConverCri:S}
%     \end{align}
% 	{where $|\kappa| = \eps + \tau$.}
% \end{lemma}

\begin{proof}[Proof of Corollary \ref{Coro:2}] {The estimate
\begin{equation*}
    \|\Delta c_\kappa - c_\kappa + w_\kappa\|_{L^2(0,T;H^1(\Omega))} \le C\sqrt{|\kappa|}
\end{equation*}
follows immediately from \eqref{L:I2D:Balan:S}. For the remaining part}, we use the equation for $w_\varepsilon$ to write
    \begin{align*}
         (n_\kappa - w_\kappa)^2 =   (n_\kappa - w_\kappa)(\varepsilon \partial_t w_\kappa - \tau \Delta w_\kappa).
    \end{align*}
Therefore, straightforward computations show
    \begin{align*}
        \|n_\kappa - w_\kappa\|_{L^2(\Omega_T)}^2 = \,&\, \varepsilon \iint_{\Omega_T} n_\kappa   \partial_t w_\kappa -  \iint_{\Omega_T} \big( \tau n_\kappa    \Delta w_\kappa  +   w_\kappa (\varepsilon \partial_t w_\kappa - \tau \Delta w_\kappa) \big) \\
        = \,&\, \varepsilon \int_{\Omega} ( n_\kappa(T)    w_\kappa(T)- n_0 w_\kappa(0) ) - \varepsilon \iint_{\Omega_T}     w_\kappa \partial_t n_\kappa \\
         - \,&\,   \iint_{\Omega_T} \big( \tau n_\kappa    \Delta w_\kappa  +   w_\kappa (\varepsilon \partial_t w_\kappa - \tau \Delta w_\kappa) \big)  \\
         = \,&\, \varepsilon \int_{\Omega} ( n_\kappa(T)    w_\kappa(T)- n_0 w_\kappa(0) ) + \varepsilon \iint_{\Omega_T} \big( \nabla n_\kappa \cdot \nabla    w_\kappa -  n_\kappa \nabla c_\kappa \cdot \nabla w_\kappa \big) \\
         + \,&\, \tau \iint_{\Omega_T} \big( \nabla  n_\kappa \cdot \nabla w_\kappa - |\nabla w_\kappa|^2) - \frac{\varepsilon}{2} \int_\Omega (w_\kappa^2(T)-
        w_\kappa^2(0)),  
    \end{align*}
where we have used the equation for $n_\kappa$ and integration by parts in the last computation. Recalling from Theorem \ref{Theo:I2D} that  $n_\kappa$ and $ w_\kappa$ are uniformly-in-$\kappa$ bounded in $L^\infty(\Omega_T)$ that
\begin{align*}
    \left| \varepsilon \int_{\Omega} ( n_\kappa(T)    w_\kappa(T)- n_0 w_\kappa(0) ) \right| \le \big(2 |\Omega|  \|n_\kappa\|_{L^\infty(\Omega_T)} \|w_\kappa\|_{L^\infty(\Omega_T)} \big)\eps \le C_T \varepsilon ,  
\end{align*}
and
\begin{align*}
    \left| \frac{\varepsilon}{2} \int_\Omega (w_\kappa^2(T)-
        w_\kappa^2(0)) \right| \le \big( |\Omega|   \|w_\kappa\|_{L^\infty(\Omega_T)}^2 \big)  \eps \le C_T \varepsilon.
\end{align*}
Thanks to the uniform-in-$\kappa$ boundedness of $\nabla n_\kappa$ and $\nabla w_\kappa$ in $L^2((0,T);H^1(\Omega))$, again from Theorem \ref{Theo:I2D},   
\begin{align*}
    \left| \varepsilon \iint_{\Omega_T} \big( \nabla n_\kappa \cdot \nabla    w_\kappa -  n_\kappa \nabla c_\kappa \cdot \nabla w_\kappa \big) \right| \le \,&\, \left( \|\nabla n_\kappa \|_{L^2(\Omega_T)}  + \|n_\kappa\|_{L^\infty(\Omega_T)} \|\nabla c_\kappa \|_{L^2(\Omega_T)} \right) \|\nabla w_\kappa \|_{L^2(\Omega_T)} \varepsilon \\
    \le \,&\, C_T \varepsilon ,
\end{align*}
as well as 
\begin{align*}
    \left| \tau \iint_{\Omega_T} \big( \nabla  n_\kappa \cdot \nabla w_\kappa - |\nabla w_\kappa|^2) \right| \le \left( \|\nabla n_\kappa \|_{L^2(\Omega_T)}  + \|\nabla w_\kappa \|_{L^2(\Omega_T)} \right) \|\nabla w_\kappa \|_{L^2(\Omega_T)} \tau \le  C_T \tau.
\end{align*}
Altogether, we get the estimate desired estimate.
\end{proof}

%\section{Preliminaries}

\appendix
%\section*{Appendices}
%\addcontentsline{toc}{section}{Appendices}
\section{Appendix}\label{sec:appendix}
%\subsection{Neumann heat semigroup}
\subsection*{Neumann heat semigroup}
% \label{App:A}
\addcontentsline{toc}{subsection}{A. Neumann heat semigroup}
\label{Sec:HeatSemi}

It is well known that the first eigenvalue of the Neumann Laplacian, defined on its domain $$W_N^{2,s}(\Omega):=\{f\in W^{2,s}(\Omega) : \nabla f \cdot \nu =0 \text{ on } \partial \Omega \},$$
is zero when $s=2$, and so, the first eigenvalue of $-\Delta + I$ is $1$. Moreover, the family $\{e^{t(\Delta- I)}\}_{t\ge 0}$, generated by $-\Delta + I$, is an analytic semigroup of linear bounded  operators on $L^2(\Omega)$. Thanks to \cite[Lemma 2.1]{horstmann2005boundedness}, there exists $\lambda >0$ such that  
\begin{align}
\|(\Delta-I)^k e^{t(\Delta-I)}f\|_{L^s(\Omega)} \le C e^{-\lambda t} t^{-k} \|f\|_{L^s(\Omega)} , \quad t>0,  
\label{GloEx:HS:Contraction}
\end{align}
for all $1<s<\infty$. If $k=0$, we can take $\lambda = C=1$ as well as $s=\infty$, see \cite[Theorem 13.4]{amann1984existence}. On the other hand, it holds for all $1\le p\le q\le \infty$ that
\begin{align}
\| e^{t(\Delta-I)}f\|_{L^q(\Omega)} \le C e^{- t} \min(t;1)^{-\frac{N}{2}(\frac{1}{p}-\frac{1}{q})} \|f\|_{L^p(\Omega)} , \quad t>0,  
\label{GloEx:HS:LpLq}
\end{align}  
see \cite[Proposition 48.4]{quittner2019superlinear}.
%In the following lemma, an estimate for the combination of $(-\Delta + I)^{k}e^{t\Delta}$ and the divergence $ \nabla \cdot$ will be presented.   

% \begin{lemma}[{\cite[Lemma 2.1]{horstmann2005boundedness}}]
% \label{L:HeatCombineDiv}
% Let  $1<s<\infty$. There exists $\lambda>0$ such that
% \begin{align*}
% \|(-\Delta + I)^{k}e^{t\Delta} \nabla \cdot f \|_{L^s(\Omega)} \le C_\alpha  e^{-\lambda t} t^{-k-\frac{1}{2}-\alpha} \|f \|_{L^s(\Omega)}, \quad t>0,
% \end{align*}
% for any $\alpha>0$, and all $v \in C_0^\infty(\Omega)$, $k\ge 0$. Consequently, the operator $(-\Delta + I)^{k}e^{t\Delta } \nabla \cdot$ has a unique extension to $L^s(\Omega)$, denoted by the same notation. 
% \end{lemma}

%\subsection{Inequalities for balancing energy functionals}

\subsection*{Inequalities for balancing energy functionals}
% \label{App:B}
 \addcontentsline{toc}{subsection}{B. Inequalities for balancing energy functionals}

Throughout the paper, we denote 
\begin{gather}
     L\log L(\Omega) := \left\{ \phi\in L^1(\Omega) \bigg|   \int_\Omega \max(|\phi|\log|\phi|;\,0) < \infty \right\} .
     \label{Def:LLogLSpace} 
\end{gather} 
Proof of Lemma \ref{L:Ineqn:Bala} in the case $N=4$ can be found in \cite{fujie2017application}. Since we consider $1\le N\le 4$, we present its proof below for convenience.  

\begin{lemma} 
\label{L:Ineqn:Bala}
Assume  $f \in  L\log L(\Omega)$ is a nonnegative function such that $\nabla \sqrt{f} \in L^2(\Omega)$. Then, for any $\alpha>0$, there exists a constant $C_\alpha>0$   such that 
\begin{align}
    \|f\|_{L^{\frac{N}{N-1}}(\Omega)}^2 \le  \alpha    
\left(  \int_{\Omega} (f \log f +  e^{-1}) \right) \|\nabla \sqrt{f} \|_{L^2(\Omega)}^2 +   C_\alpha.
\end{align}   
\end{lemma} 

\begin{proof}[Proof of Lemma \ref{L:Ineqn:Bala}] For $s>1$, we define $g:=(\sqrt{f}-\sqrt{s})_+$. By the Gagliardo-Nirenberg inequality, 
\begin{align*}
    \|g\|_{L^{\frac{2N}{N-1}}(\Omega)}^4 \le C \|\nabla g\|_{L^2(\Omega)}^2 \|g\|_{L^{2}(\Omega)}^2 \le C \|\nabla \sqrt{f} \|_{L^2(\Omega)}^2 \|g\|_{L^{2}(\Omega)}^2 , 
\end{align*}  
where the latter norm can be estimated as follows 
\begin{align*}
    \|g\|_{L^{2}(\Omega)}^2  \le \|\sqrt{f}\|_{L^{2}(\Omega \cap \{f\ge s\})}^2 =    \int_{\Omega \cap \{f\ge s\}} f \le  \frac{1}{\log s} \int_{\Omega} (f \log f +  e^{-1}).     
\end{align*}
Moreover, using the inequality $(a+b)^p \le 2^{p-1}(a^p+b^p)$ for all $a,b\ge 0$ and $p\ge 1$, we  have 
\begin{align*}
    \|f\|_{L^{\frac{N}{N-1}}(\Omega)}^2 & \le \left(    \int_{\Omega\cap \{f\ge s\} } \left( (\sqrt{f}-\sqrt{s}) + \sqrt{s} \right)^{\frac{2N}{N-1}} +  \int_{\Omega\cap \{f< s\} }  (\sqrt{f} )^{\frac{2N}{N-1}}   \right)^{\frac{2N-2}{N}} \\
    & \le  \left( 2^{\frac{N+1}{N-1}}  \int_{\Omega} (\sqrt{f}-\sqrt{s})_+^{\frac{2N}{N-1}} + \max \left( 2^{\frac{N+1}{N-1}} ;\, 1\right) \int_{\Omega} (\sqrt{s})^{\frac{2N}{N-1}}   \right)^{\frac{2N-2}{N}} \\
    & \le 8 \|(\sqrt{f}-\sqrt{s})_+\|_{L^{\frac{2N}{N-1}}(\Omega)}^4 + 2^{\frac{N-2}{N}} \max \left( 2^{\frac{3N-2}{N+2}} ;\, 1\right)^{\frac{2N-2}{N}}    |\Omega|^{\frac{2N-2}{N}} s^2.      
    \end{align*}
Combining the above estimates gives
\begin{align*}
      \|f\|_{L^{\frac{N}{N-1}}(\Omega)}^2 \le     
 \frac{8C}{\log s} \left( \int_{\Omega} (f \log f +  e^{-1}) \right) \|\nabla \sqrt{f} \|_{L^2(\Omega)}^2 +  2^{\frac{N-2}{N}} \max \left( 2^{\frac{3N-2}{N+2}} ;\, 1\right)^{\frac{2N-2}{N}}    |\Omega|^{\frac{2N-2}{N}} s^2,  
\end{align*}
which ends the proof by choosing $s$ such that  $ 8C (\log s)^{-1} =\alpha$.
\end{proof}

Lemmas \ref{L:Ineqn:LlogL}-\ref{L:Ineqn:Adam} below can be respectively found in \cite[Lemmas  7.1 and 3.5]{fujie2017application}.

\begin{lemma} 
\label{L:Ineqn:LlogL} 
    Let $\beta>0$. If $f,g$ are nonnegative functions such that $f\in L^1(\Omega) \cap L\log L  (\Omega)$, then 
\begin{align}
    \int_{\Omega} fg \le \frac{1}{\beta} \int_{\Omega} f\log f + \frac{\|f\|_{L^1(\Omega)}}{\beta} \log \left( \int_{\Omega} e^{\beta g} \right) + \frac{1}{e} ,
    \label{Ineqn:Pre:Adam}
\end{align}
whenever the latter logarithm is finite. 
\end{lemma}

%  This leads to estimating the second term on the right-hand side, for which, an application of the Adam-type inequality is suggested. However, since this inequality involves functions in $H_0^{2}(\Omega)$, while our solution's components are subjected to the homogeneous Neumann boundary conditions, the symmetric setting for the domain $\Omega$ is considered in this part.  

%\medskip

Next, we present two consequences of the Moser-Trudinger and Adam-type inequalities, where the second one is restricted to a radially symmetric setting.  Let $B_R$ be the open ball centred at the origin of a given radius $0<R<\infty$, and $H_{\mathsf{rad}}^2(B_R)$ be the set of all radially symmetric functions in $H^2(B_R)$.   

\begin{lemma} 
\label{L:Ineqn:Adam}
Given $\beta>0$ and $\eta>0$. 
\begin{itemize}
    \item[i)] (A consequence of the Moser-Trudinger inequality) If $N=2$, then there is $C_{\beta}>0$ such that 
\begin{align}
    \log \left( \int_{\Omega} e^{\beta g} \right) \le \frac{\beta^2}{8\pi}  \|\nabla  g\|_{L^2(\Omega)}^2 + \frac{\beta}{|\Omega|} \int_\Omega g + C_{\beta},
    \label{Ineqn:MT}
\end{align}
for all $g \in H^1(\Omega)$. 
    
    \item[ii)] (A consequence of the Adam-type inequality) If $N=4$, then there is $C_{\beta,\eta}>0$ such that
\begin{align}
    \log \left( \int_{B_R} e^{\beta g} \right) \le \left( \frac{\beta^2}{128\pi^2} + \eta \right) \|(\Delta-I)g\|_{L^2(B_R)}^2 + C_{\beta,\eta},
    \label{Ineqn:Adam}
\end{align}
for all $g\in H_{\mathsf{rad}}^2(B_R)$. 
\end{itemize}
\end{lemma}

%\subsection{Linear parabolic equations with slow evolution}

\subsection*{Linear parabolic equations with slow evolution}
% \label{App:C}
\addcontentsline{toc}{subsection}{C. Linear parabolic equations with slow evolution}

For a given small  relaxation parameter $0<\varepsilon\ll1$, we consider in general regularity of the solution $u_\varepsilon$ to the linear parabolic equation 
\begin{align}\label{Sys:SlowEvolution}
\left\{ \begin{array}{lllll}
 \partial_t u_\varepsilon= \dfrac{1}{\varepsilon} \left( d \Delta u_\varepsilon -  u_\varepsilon +   f \right)  & \text{in } \Omega \times(0,T), \vspace{0.1cm}  \\
\nabla u_\varepsilon \cdot \nu = 0 & \text{on } \Gamma\times(0,T), \vspace{0.15cm} \\
u_\varepsilon (0) = u_0   & \text{on } \Omega.
\end{array}
\right.
\end{align}
where $d>0$ is a diffusion coefficient, the functions  $f$ and $u_0$ are given. We focus on the maximal regularity and local-in-space regularity uniformly in the relaxation parameter $\varepsilon$.    

\begin{lemma} 
\label{L:MR:SlowEvolution}   
Let $0<\varepsilon<1$ and $u_\varepsilon$ be the  solution to Problem 
\eqref{Sys:SlowEvolution}. Assume that $u_0$ satisfies the compatibility condition {$\nabla u_0 \cdot \nu = 0$ on $\Gamma$}. 
Then, for any $1<p,q<\infty$,  
\begin{align}
 \sup_{\varepsilon>0} \Big( \|u_\varepsilon \|_{W^{2,p}(\Omega_t)} \Big)   \le  \left( \frac{\varepsilon}{p} \right)^{\frac{1}{p}} \| u_0\|_{W^{2,p}(\Omega)} +  C_p \|f \|_{L^p(\Omega_t)} ,
\label{L:MR:SlowEvolution:S1}
\end{align}
and
\begin{align}
 \sup_{\varepsilon>0} \Big( \| u_\varepsilon \|_{L^q((0,T);W^{2,p}(\Omega))} \Big)   \le  C_{p,q} \left( \|u_0\|_{W^{2,p}(\Omega)} +  \|f \|_{L^q((0,T);L^p(\Omega))} \right) ,
\label{L:MR:LpLq}
\end{align}
where the constants $C_p$, $C_{p,q}$ are independent on $\varepsilon$. 
\end{lemma}

\begin{proof} Estimate \eqref{L:MR:SlowEvolution:S1} was proved in \cite[Lemma 3.4]{reisch2024global}. To prove \eqref{L:MR:LpLq}, we consider the rescaling $t'=t/\varepsilon$ and the substitution 
$$\widehat{u}_\varepsilon(x,t')=u_\varepsilon(x,t) \quad \text{and} \quad \widehat{f}(x,t')=f(x,t),$$ 
which recasts Problem 
\eqref{Sys:SlowEvolution} to the form
\begin{align*} 
\left\{ \begin{array}{lllll}
 \partial_{t'} \widehat{u}_\varepsilon= d \Delta \widehat{u}_\varepsilon -  \widehat{u}_\varepsilon +   \widehat{f} & \text{in } \Omega \times(0,T/\varepsilon), \vspace{0.1cm}  \\
\nabla \widehat{u}_\varepsilon \cdot \nu = 0 & \text{on } \Gamma\times(0,T/\varepsilon), \vspace{0.15cm} \\
\widehat{u}_\varepsilon (0) = u_0   & \text{on } \Omega.
\end{array}
\right.
\end{align*}
Then, by applying $L^p-L^q$ maximal regularity, we get 
    \begin{align}
  \|\Delta \widehat{u}_\varepsilon \|_{L^q(0,T/\varepsilon;L^p(\Omega))} \le C_{p,q} \left( \| u_0\|_{W^{2,p}(\Omega)} +   \|\widehat{f} \|_{L^q(0,T/\varepsilon;L^p(\Omega))} \right) , 
\end{align}
where the constant $C_{p,q}$ is independent on the terminal time $T$ and the parameter $\varepsilon$. By noticing that $\|\widehat{\varphi}\|_{L^q(0,T/\varepsilon;L^p(\Omega))} = \varepsilon^{-1/q} \|\varphi\|_{L^q((0,T);L^p(\Omega))}$ for $\varphi\in\{u;f\}$ as well as $\varepsilon<1$, we obtain \eqref{L:MR:LpLq}. 
\end{proof}

\begin{lemma}
\label{Local:Pre:Lem0} Let {$N\ge 3$ and} $u_\varepsilon$ be the solution to Problem 
\eqref{Sys:SlowEvolution} for each $\varepsilon>0$. Then,  
\begin{align}
     \sup_{\varepsilon>0} \left( \varepsilon \int_\Omega u_\varepsilon^2  +   \iint_{\Omega_t} \left(  |\nabla u_\varepsilon|^2 + u_\varepsilon^2 \right) \right) \le   \int_\Omega u_0^2 +  \frac{C}{d^2}  \int_0^{t}  \|f\|_{L^{\frac{2N}{N+2}}(\Omega)}^2 ,  \label{Local:Pre:Lem0:S}
\end{align} 
for any $t\in (0,T)$, 
provided that the right-hand side exists finitely. 
\end{lemma}

\begin{proof} 
Using the Sobolev embedding $H^{1}(\Omega)\hookrightarrow L^{2N/(N-2)}(\Omega)$ and the Young inequality, we see   
\begin{align}
    \int_\Omega f_\varepsilon u_\varepsilon & \le \|f\|_{L^{\frac{2N}{N+2}}(\Omega)} \|u_\varepsilon\|_{L^{\frac{2N}{N-2}}(\Omega)} \le C \|f\|_{L^{\frac{2N}{N+2}}(\Omega)} \|u_\varepsilon\|_{H^1(\Omega)} \\
    & \le  \frac{d}{2} \int_\Omega \left(  |\nabla u_\varepsilon|^2 + u_\varepsilon^2  \right) + \frac{C}{d} \|f\|_{L^{\frac{2N}{N+2}}(\Omega)}^2.
    \label{Local:Pre:Lem0:P1}
\end{align}
Therefore, testing this equation by $u_\varepsilon$, we get 
\begin{align*}
    & \frac{1}{2} \frac{d}{dt} \int_\Omega u_\varepsilon^2 + \frac{d}{\varepsilon} \int_\Omega |\nabla u_\varepsilon|^2 + \frac{1}{\varepsilon}\int_\Omega u_\varepsilon^2 
      \le  \frac{d}{2\varepsilon} \int_\Omega \left(  |\nabla u_\varepsilon|^2 + u_\varepsilon^2  \right) + \frac{C}{d \varepsilon} \|f\|_{L^{\frac{2N}{N+2}}(\Omega)}^2 , 
\end{align*}
and consequently, 
\begin{align*}
    \frac{d}{dt}  \int_\Omega u_\varepsilon^2 + \frac{d}{\varepsilon} \int_\Omega |\nabla u_\varepsilon|^2 + \frac{1}{\varepsilon}   \int_\Omega u_\varepsilon^2    \le  \frac{C}{d\varepsilon} \|f\|_{L^{\frac{2N}{N+2}}(\Omega)}^2 . 
\end{align*}
Then, integrating the two sides of the latter inequality over time gives
\begin{align*}
    \varepsilon \int_\Omega u_\varepsilon^2  +   \iint_{\Omega_t} \left(  |\nabla u_\varepsilon|^2 + u_\varepsilon^2 \right) \le \varepsilon \int_\Omega u_0^2 +  \frac{C}{d^2}  \int_0^{t}  \|f\|_{L^{\frac{2N}{N+2}}(\Omega)}^2 ,
\end{align*}
and consequently shows  estimate \eqref{Local:Pre:Lem0:S} by noticing that $\varepsilon \ll1$. 
\end{proof}

For the sake of convenience, we also present here a linear differential inequality with slow evolution, {which can be easily proved}. 

\begin{lemma} 
\label{aDI}
Given $\varepsilon >0$. 
    If a continuous function $y:[0,T] \to [0,\infty)$ satisfies that 
    \begin{align*}
        \varepsilon \frac{d}{dt} x(t) + a x(t) \le y(t), \quad 0 \le t\le T,
    \end{align*}
    for some $a>0$, then
    \begin{align*}
        x(t) \le x(0) e^{ - at/\varepsilon }  + \frac{1}{\varepsilon} \int_0^t e^{-a(t-s)/\varepsilon} y(s) ds, \quad 0 \le t\le T. 
    \end{align*}
\end{lemma}
 
\vspace{0.4cm}

\noindent {\Large \bf Acknowledgement}

\medskip

\noindent  This research was funded in whole, or in part, by the Austrian Science Fund (FWF) 10.55776/I5213. The authors gratefully acknowledge the support of ASEA-UNINET project number ASEA 2023-2024/Uni Graz/6. This research is partially completed during the visit of the last author to Vietnam National University Ho Chi Minh City, and the university's hospitality is greatly acknowledged.

% {\small	 
% \bibliographystyle{alpha}
% \bibliography{BHTTRef}
% }  

{\small	
\newcommand{\etalchar}[1]{$^{#1}$}

}

\end{document}